\documentclass[a4paper,11pt]{article}

\usepackage{graphicx,enumitem}
\usepackage{color}
\usepackage{latexsym}
\usepackage{amsmath, amssymb, mathrsfs}
\usepackage{algorithmic}
\usepackage{algorithm}

\usepackage{libertine}
\usepackage{soul,xcolor}
\usepackage[colorlinks]{hyperref}
\usepackage{tikz}
\usetikzlibrary{arrows.meta}
\usetikzlibrary{calc}
\usepackage{subcaption} 
\usepackage{amsfonts,amsmath,amssymb}
\usepackage{amsthm}
\usepackage{geometry}
\usepackage{graphicx} 
\usepackage{xcolor}
\usepackage{hyperref}
\usepackage{footnote}
\usepackage{stackrel}
\usepackage{xfrac}
\usepackage[tikz]{mdframed}

\newtheorem{theorem}{Theorem}[section]
\newtheorem{definition}{Definition}   
\newtheorem{remark}{Remark}
\newtheorem{proposition}[theorem]{Proposition}
\newtheorem{lemma}[theorem]{Lemma}
\usepackage{amsmath}
\usepackage{mathtools}
\usepackage{mathrsfs}
\usepackage{amssymb}
\usepackage{multicol}
\usepackage[colorlinks]{hyperref}
\usepackage{caption,subcaption}
\usepackage{bm}
\newcommand{\vect}[1]{\boldsymbol{\mathbf{#1}}}
\newcommand{\vertiii}[1]{{\left\vert\kern-0.25ex\left\vert\kern-0.25ex\left\vert #1 
    \right\vert\kern-0.25ex\right\vert\kern-0.25ex\right\vert}}
\newtheorem{problem}{Problem}
\newtheorem{corollary}[theorem]{Corollary}
\newtheorem{example}{Example}

\newcommand{\dotu}{\dot{u}}
\newcommand{\un}{u_\text{N}}
\newcommand{\ud}{u_\text{D}}
\newcommand{\cv}{\overline{v}}
\newcommand{\cu}{\overline{u}}
\newcommand{\cphi}{\overline{\varphi}}
\newcommand{\ut}{u^t}
\newcommand{\ur}{u_{1}}
\newcommand{\ui}{u_{2}}
\newcommand{\vr}{p_{1}}
\newcommand{\vi}{p_{2}}
\newcommand{\up}{u'}

\newcommand{\uip}{u_{2}'}

\newcommand{\aaa}{\vect{\mathsf{a}}}
\newcommand{\aat}{\vect{\mathsf{a}}_t}
\newcommand{\GG}{G} 
\newcommand{\LL}{\vect{\mathsf{L}}}
\newcommand{\QQ}{\vect{\mathsf{Q}}}
\newcommand{\SSig}{\vect{\mathsf{S}}}
\newcommand{\sfPsi}{{\mathcal{F}}}
\newcommand{\sfTheta}{\vect{\Theta}}
\newcommand{\sfj}{\mathsf{j}}

\newcommand{\HH}{\vect{\mathsf{H}}}
\newcommand{\HHg}{\vect{\mathsf{V}}}

\newcommand{\VV}{\vect{V}}
\newcommand{\Vn}{V_{n}}
\newcommand{\Wn}{W_\text{n}}
\newcommand{\WW}{\vect{W}}
\newcommand{\svv}{\vect{v}}
\newcommand{\ww}{\vect{w}}
\newcommand{\vn}{v_\text{n}}
\newcommand{\svvs}{\vect{v}_{\Sigma}}

\newcommand{\wws}{\vect{w}_{\Sigma}} 
\newcommand{\nn}{\vect{n}}

\newcommand{\dn}[1]{\partial_{\nn}{#1}}

\newcommand{\intOast}[1]{\int_{\Omega^{\ast}}{#1}{\, {\operatorname{\mathnormal{d}}} x}}
\newcommand{\intO}[1]{\int_{\Omega}{#1}{\, {\operatorname{\mathnormal{d}}} x}}
\newcommand{\intdO}[1]{\int_{\partial\Omega}{#1}{\, {\operatorname{\mathnormal{d}}} \sigma}}
\newcommand{\intOt}[1]{\int_{\Omega_t}{#1}{\, {\operatorname{\mathnormal{d}}} x_t}} 

\newcommand{\intS}[1]{\int_{\Sigma}{#1}{\, {\operatorname{\mathnormal{d}}} \sigma}} 
\newcommand{\intSast}[1]{\int_{\Sigma^\ast}{#1}{\, {\operatorname{\mathnormal{d}}} \sigma}}
\newcommand{\intSt}[1]{\int_{\Sigma_t}{#1}{\, {\operatorname{\mathnormal{d}}} \sigma_t}}

\begin{document}

\title{On the new coupled complex boundary method in shape optimization framework for solving stationary free boundary problems\footnote{Some of the results in this paper were presented at the SAB-ATAN 2021: International  Conference on Mathematical Sciences and Applications held virtually on November 24-26, 2021 and hosted by the University of the Philippines Baguio.}\\
\medskip
}
\author{Julius Fergy T. Rabago}

\date{%
	{\footnotesize
	 Faculty of Mathematics and Physics\\%
	 Institute of Science and Engineering\\%
         Kanazawa University, Kanazawa 920-1192, Japan\\\vspace{-2pt}
        \texttt{rabagojft@se.kanazawa-u.ac.jp,\ jfrabago@gmail.com}}\\[2ex]
    \today
}

\maketitle

\begin{abstract}
We expose here a novel application of the so-called coupled complex boundary method -- first put forward by Cheng et al. (2014) to deal with inverse source problems -- in the framework of shape optimization for solving the exterior Bernoulli problem, a prototypical model of stationary free boundary problems.
The idea of the method is to transform the overdetermined problem to a complex boundary value problem with a complex Robin boundary condition coupling the Dirichlet and Neumann boundary conditions on the free boundary.
Then, we optimize the cost function constructed by the imaginary part of the solution in the whole domain in order to identify the free boundary.
We also prove the existence of the shape derivative of the complex state with respect to the domain.
Afterwards, we compute the shape gradient of the cost functional, and characterize its shape Hessian at the optimal domain under a strong, and then a mild regularity assumption on the domain. 
We then prove the ill-posedness of the proposed shape problem by showing that the latter expression is compact.
Also, we devise an iterative algorithm based on a Sobolev gradient scheme via finite element method to solve the minimization problem.
Finally, we illustrate the applicability of the method through several numerical examples, both in two and three spatial dimensions.
\medskip

\end{abstract}

\newpage
\section{Introduction}
\label{sec:Introduction}
Let $A$ and $B$ be two bounded and simply connected domains in $\mathbb{R}^d$, $d\in\{2,3\}$, with respective boundaries $\Gamma: = \partial A$ and $\Sigma: = \partial B$, such that $B \supset \overline{A}$.
Denote the annular domain $\Omega = B \setminus \overline{A}$ (which is assumed, unless stated otherwise, to be non-empty bounded open
Lipschitz subset of $\mathbb{R}^{d}$ throughout the paper), then an exterior free boundary problem may be given as follows:
for given functions $f$, $g$, $h$, one tries to find $\Omega$ with the associated function $u:=u(\Omega)$ such that the overdetermined problem
\begin{equation}
\label{eq:FBP}
		-\Delta u = f \quad \text{in}\ \Omega,\qquad
		u = g \quad  \text{on}\  \Gamma,\qquad
		u = 0 \quad \text{and} \quad 
	 	\dn{u} = h \quad\text{on} \ \Sigma,
\end{equation}
is satisfied, where $\dn{u}:=\nabla u \cdot \nn$ is the outward normal derivative of $u$.

In this work, we are primarily interested in solving the free boundary problem (FBP) \eqref{eq:FBP} through the novel application of the so-called \textit{coupled complex boundary method} or \textit{CCBM} in solving stationary FBPs through the context of shape optimization.
For simplicity of discussion, we will consider the prototypical case of \eqref{eq:FBP} in two spatial dimensions popularly known as the exterior Bernoulli problem wherein $f \equiv 0$, $g \equiv 1$, and $h = \lambda$, where $\lambda < 0$ is a fixed constant in \eqref{eq:FBP}.
That is, we consider the problem of finding a pair $(\Omega, u) := (\Omega, u(\Omega))$ that solves the system
\begin{equation}
\label{eq:Bernoulli_problem}
		-\Delta u = 0 \ \text{in} \ \Omega,\qquad
		u = 1 \ \text{on} \ \Gamma,\qquad
		u = 0 \quad \text{and} \quad 
	 	\dn{u} = \lambda \ \text{on} \ \Sigma.
\end{equation}
The system of partial differential equations (PDEs) admits a classical solution for simply connected bounded domain $\Omega$ for any $\lambda < 0$, and uniqueness can be guaranteed for bounded convex domains $A$ \cite{FlucherRumpf1997}.
Moreover, in the said case, it was shown in \cite[Thm 1.1]{HenrotShahgholian2000a} that the free boundary is $\mathcal{C}^{2,\alpha}$ regular.
For additional qualitative properties of solutions, as well as their numerical treatment, we refer the readers, for instance, to \cite{FlucherRumpf1997}.
The problem in consideration is also known in the literature as the \textit{Alt-Caffarelli problem} \cite{AltCaffarelli1981}.
It originates from the description of free surfaces for ideal fluids \cite{Friedrichs1934}, but copious industrial applications leading to similar formulations to \eqref{eq:Bernoulli_problem} arises in many other related contexts, see \cite{Crank1984,Fasano1992,FlucherRumpf1997}.
We mention that, in this paper, we do not tackle the question of existence of optimal shape solutions for the proposed shape optimization problem.
Instead, we will tacitly assume the existence of optimal domains which is sufficiently regular to carry out a second-order shape calculus.
Even so, we mention that existence proofs and tools developed in this direction are issued in \cite{Boulkhemairetal2013,HKKP2004a,HKKP2003,HaslingerMakinen2003}.
 
The methods of shape optimization is a well-established tool to solve free boundary problems, and in the case of \eqref{eq:Bernoulli_problem}, the method can be applied in several ways.
The usual strategy is to choose one of the boundary conditions on the free boundary to obtain a well-posed state equation and then track the remaining boundary data in $L^2(\Sigma)$ (see \cite{EpplerHarbrecht2009,EpplerHarbrecht2010,HIKKP2009,HKKP2003,IKP2006,RabagoBacani2017,RabagoBacani2018}) or utilize the Dirichlet energy functional as a shape functional (see \cite{EpplerHarbrecht2006,Harbrecht2008}).
Another option is to consider an energy-gap type cost function which consists of two auxiliary states; one that is a solution of a pure Dirichlet problem and one that satisfies a mixed Dirichlet-Neumann problem (see \cite{BenAbdaetal2013,Bacani2013,EpplerHarbrecht2012a}).
The latter formulation is precisely given as the minimization problem
\begin{equation}\label{eq:KV_method}
	J_{KV}(\Omega):=\frac12 \intO{\left| \nabla\left(\un - \ud\right) \right|^2} \ \to\ \inf,
\end{equation} 
where the state functions $\un$ and $\ud$ respectively satisfy the well-posed systems
\begin{align}
	&-\Delta \un = 0 \ \text{in} \ \Omega,
	&\un = 1 \ \text{on} \ \Gamma,\quad\qquad
	&\dn{\un} = \lambda \ \text{on} \ \Sigma; \label{eq:state_un}&\\
	&-\Delta \ud = 0 \ \text{in} \ \Omega,
	&\ud = 1 \ \text{on} \ \Gamma,\quad\qquad
	&\ud = 0 \ \text{on} \ \Sigma.&			\label{eq:state_ud}
\end{align}
The equivalence between the above shape optimization formulation and the exterior Bernoulli problem \eqref{eq:Bernoulli_problem} issues from the following statement.
If $(u,\Omega)$ is a solution of \eqref{eq:Bernoulli_problem}, then $\un=\ud=u$; therefore, $J_{KV}(\Omega) \equiv 0$.
Conversely, if $J_{KV}(\Omega) = 0$, then $u = \un = \ud$ is a solution of problem \eqref{eq:Bernoulli_problem} because $\ud - \un \in H^{1}_{\Gamma,0}:=\{\varphi \in H^{1}(\Omega) \mid \varphi = 0 \ \text{on}\ \Gamma\}$, $J_{KV}(\Omega) = \frac12\big|\ud - \un\big|_{H^{1}(\Omega)}$ is a norm on $H^{1}_{\Gamma,0}(\Omega)$.
Formulation \eqref{eq:KV_method} is better known in the literature as the Kohn-Vogelius approach (see, e.g. \cite{Bacani2013,EpplerHarbrecht2012a,KohnVogelius1987}).

Recently, some modifications of the above mentioned existing methods were offered and examined in \cite{RabagoAzegami2019a,RabagoAzegami2019b}, including the so-called \textit{Dirichlet-data-gap} cost functional approach -- firstly proposed in \cite{RabagoAzegami2020}.
These approaches make use of a Robin problem as one of the state problem. 
Additionally, we mention that there is another numerical technique -- closely related to shape optimization method and was originally conceived to solve moving boundary problems via finite element method -- called the comoving mesh method or CMM developed in \cite{SunayamaKimuraRabago2022} which can be applied to solve the free boundary problem \eqref{eq:Bernoulli_problem}; see \cite[Section 3]{SunayamaKimuraRabago2022} for details.

In this paper, we want to offer yet another shape optimization approach to solve \eqref{eq:Bernoulli_problem}.
The idea is somewhat similar to \cite{RabagoAzegami2019a,Tiihonen1997}, but applies the concept of complex PDEs.
More exactly, we showcase here a new application of the so-called coupled complex boundary method for solving stationary free boundary problems.
The idea of the method is simple: we couple the Dirichlet and Neumann data in a Robin boundary condition in such a way that the Dirichlet data and the Neumann data are the respective real and imaginary parts of the Robin boundary condition.
As a result, the conditions that have to be satisfied on the free boundary are transformed into one condition that needs to be satisfied on the domain.
With the new method, as in problem \eqref{eq:KV_method}, the objective function can then be defined as a volume integral, see \eqref{eq:cost_function}.

The CCBM was first put forward by Cheng et al. in \cite{Chengetal2014} for solving an inverse source problem (see also \cite{Chengetal2016}) and was then used to solve the Cauchy problem stably in \cite{Chengetal2016}.
It is later on applied to an inverse conductivity problem with one measurement in \cite{Gongetal2017} and also to parameter identification in elliptic problems in \cite{Zhengetal2020}.
Much more recently, CCBM was also applied in solving inverse obstacle problems by Afraites in \cite{Afraites2022}.
To the best of our knowledge -- as the method has not been applied yet to solving FBPs in previous investigations -- this is the first time that CCBM will be explored to deal with stationary FBPs, particularly as a numerical resolution to the exterior Bernoulli problem.
This, in turn, provides new directions in treating related free boundary/surface problems in the context of shape optimization.

The remainder of the paper is as follows.
In Section \ref{sec:shape_optimization_formulation}, we describe how CCBM actually applies to solving problem \eqref{eq:Bernoulli_problem} in the framework of shape optimization. 
We also give further motivations about the present study -- providing merits to considering CCBM as a way to solve FBPs.
Section \ref{sec:Shape_Derivatives} is devoted to proving that the map $t \mapsto \ut$ is $\mathcal{C}^1$ in a neighborhood of zero, and to the characterization of its derivative.
Analogous results for the map $t \mapsto J(\Omega_t)$ are also exhibited therein, including especially the first-order shape derivative of the cost.
We also derive its second-order shape derivative in subsection \ref{subsec:shape_Hessian}, looking particularly on its structure at the solution of  \eqref{eq:Bernoulli_problem}.
With the latter expression, we examine the instability of the shape optimization problem in subsection \ref{subsec:instability_analysis} by proving the compactness of the shape Hessian of $J$ at a critical shape.
In Section \ref{sec:Numerical_Approximation}, we devise a numerical algorithm (subsection \ref{subsec:Numerical_Algorithm}) based on \textit{Sobolev gradient} method to solve the shape optimization problem in consideration.
This is followed by an intermediate subsection (subsection \ref{subsec:properties}), where we state and prove a small result concerning a stationary point of an evolving domain that evolves according to a given pseudo flow field related to Sobolev gradient method used in this investigation.
Then, we illustrate the feasibility of the proposed method through several numerical examples both in two and three spatial dimensions (subsection \ref{subsec:Numerical_Examples}). 
Lastly, in Appendix \ref{subsec:shape_derivative_of_the_state}, a rigorous proof of the existence of the material derivative of the state, as well as the characterization of its corresponding shape derivative, are provided. 
\section{CCBM in shape optimization setting}
\label{sec:shape_optimization_formulation}
We present here the proposed CCBM formulation of \eqref{eq:Bernoulli_problem} and give the motivation of the method for solving the free boundary problem.

\subsection{Formulation and Notations} 
In what follows we present the CCBM formulation of \eqref{eq:Bernoulli_problem} and its shape optimization reformulation by introducing the least-squares fitting for the imaginary part of the complex PDE solution.
The discussion will be issued in the case of two dimensions, but the results easily extend in three dimensions.
The main point of departure of the method is to recast \eqref{eq:Bernoulli_problem} into the complex PDE system
\begin{equation}
\label{eq:complexPDE}
		-\Delta u = 0 \ \text{in} \ \Omega,\qquad\quad
		u = 1 \ \text{on} \ \Gamma,\qquad\quad
	 	\dn{u} + i u= \lambda \ \text{on} \ \Sigma,
\end{equation}
where $i = \sqrt{-1}$ stands for the unit imaginary number.
Letting $u = \ur + i \ui$ denote the solution of \eqref{eq:complexPDE}, it can be verified that the real-valued functions $\ur$ and $\ui$ respectively satisfy the real PDE systems:
\begin{align}
	&-\Delta \ur = 0 \ \text{in} \ \Omega,\qquad\qquad
	\ur = 1 \ \text{on} \ \Gamma,
	&\dn{\ur} - \ui = \lambda \ \text{on} \ \Sigma; \label{eq:real_state}\\
	&-\Delta \ui = 0 \ \text{in} \ \Omega,\qquad\qquad
	\ui = 0 \ \text{on} \ \Gamma,
	&\ur + \dn{\ui} = 0 \ \text{on} \ \Sigma.\label{eq:imaginary_state}
\end{align}
	As mentioned earlier, a Robin-type boundary condition has already been used in \cite{RabagoAzegami2019a,RabagoAzegami2019b,RabagoAzegami2020,Tiihonen1997}.
	The formulation issued here, however, is not covered by any of the aforementioned study.
	This is because the Robin coefficient appearing in \eqref{eq:complexPDE} is a complex number while the one used in \cite{RabagoAzegami2019a,RabagoAzegami2019b,RabagoAzegami2020} is a positive real number.
	Moreover, the Robin coefficient in \cite{Tiihonen1997} is assumed to be a non-negative real number.
\begin{remark}
\label{rem:equivalence}
	Let us note that if $\ui = 0$ in $\Omega$, then we have $\ui = \dn{\ui} = 0$ on $\Sigma$ and $\ur = 0$ on $\Sigma$.
	In view of \eqref{eq:real_state} and \eqref{eq:imaginary_state}, we see that the pair $(\Omega, \ur)$ solves the original free boundary problem \eqref{eq:Bernoulli_problem}. 
	Conversely, if $(\Omega, u)$ is the solution to the free boundary problem \eqref{eq:Bernoulli_problem}, then clearly $\ur$ and $\ui$ satisfy \eqref{eq:real_state} and \eqref{eq:imaginary_state}.
\end{remark}
We infer from the previous remark that the original free boundary problem \eqref{eq:Bernoulli_problem} can be recast into an equivalent shape problem given as follows.
\begin{problem}
\label{eq:main_problem}
	Given a fixed interior boundary $\Gamma$ and a real number $\lambda < 0$, find an annular domain $\Omega$, with the exterior boundary $\Sigma:=\partial \Omega\setminus \Gamma$, and a function $u := u(\Omega)$ such that $\ui = 0$ in $\Omega$ and $u = \ur + i\ui$ solves the PDE system \eqref{eq:complexPDE}.
\end{problem}
\textbf{Notations.} The notations for the function spaces used in the paper are as follows.
We denote by $W^{m,p}(\Omega)$ the standard real Sobolev space with the norm $\|\cdot\|_{W^{m,p}(\Omega)}$.
Let $W^{0,p}(\Omega)=L^p(\Omega)$, and particularly, $H^m(\Omega)$ represents $W^{m,2}(\Omega)$ with the corresponding inner product $( \cdot , \cdot)_{m,\Omega}$ and norm $\|\cdot\|_{H^m(\Omega)}$.
We let $\HH^m(\Omega)$ be the complex version of $H^m(\Omega)$ with the inner product $(\!(\cdot, \cdot)\!)_{m,\Omega}$ and norm $\vertiii{\, \cdot\, }_{\HH^m(\Omega)}$ defined as follows: for all $u, v \in \HH^m(\Omega)$, $(\!( u, v )\!)_{m,\Omega} = (u, \cv)_{m,\Omega}$ and $\vertiii{v}_{\HH^m(\Omega)}^2 = (\!( v,v )\!)_{m,\Omega}$.
Similarly, we denote $V(\Omega):=H^{1}_{\Gamma,0}(\Omega)$, $\HHg(\Omega) :=\HH^{1}_{\Gamma,0}(\Omega)$, $Q = L^2(\Omega)$, $\QQ = \LL^2(\Omega)$, $S = L^2(\Sigma)$, $\SSig = \LL^2(\Sigma)$.
Lastly, we define $\vertiii{v}_{\HHg(\Omega)}^2 :=  \intO{ \nabla {v} \cdot \nabla {\cv}} + \intS{ {v} {\cv}}$.

Before we go further, we remark that the complex PDE system \eqref{eq:complexPDE} is well-posed.
Indeed, with the sesquilinear form $\aaa$ defined on $\HHg(\Omega) \times \HHg(\Omega)$ by
	\[
		\aaa({u},{v}) = \intO{\nabla {u} \cdot \nabla {\cv}} + i \intS{{u} {\cv}}, \quad \forall {u}, {v} \in \HHg(\Omega),
	\]
and the linear form $l(v) = \lambda \intS{ {\cv}}$, we may state \eqref{eq:complexPDE} in variational form as follows:
\begin{equation}\label{eq:state_weak_form}
	\text{find ${u} \in \HH^{1}(\Omega)$, \ $u = 1$ on $\Gamma$, \ \ such that\ \  $\aaa({u},{v}) = l({v})$, \ for all $v \in \HHg(\Omega)$}.
\end{equation}
The existence and uniqueness of solution of the problem follows from the complex version of the Lax-Milgram lemma \cite[p. 376]{DautrayLionsv21998} (see also \cite[Lem. 2.1.51, p. 40]{SauterSchwab2011}).

Now, to solve Problem \ref{eq:main_problem}, we introduce the cost functional
\begin{equation}\label{eq:cost_function}
	J(\Omega) = \frac12\|\ui\|^2_{L^2(\Omega)} = \frac12 \intO{|\ui|^2},
\end{equation}
where $\ui$ is, of course, subject to the state problem \eqref{eq:imaginary_state}.
The shape optimization problem that we consider then is the problem of minimizing $J(\Omega)$ over a set of admissible domains $\mathcal{O}_\text{ad}$.
Here, $\mathcal{O}_\text{ad}$ is essentially the set of $\mathcal{C}^{1,1}$ annular domains $\Omega$ with (fixed) interior boundary $\Gamma$ and (free) exterior boundary $\Sigma$. 
It actually suffices to consider $\Gamma$ to be only Lipschitz regular in deriving the (first-order) shape derivative of the functions and shape functionals involved, but for simplicity, we also assume it to be $\mathcal{C}^{1,1}$ regular.
For second-order shape sensitivity analysis, we will require the domain $\Omega$ be of class $\mathcal{C}^{2,1}$ -- at least when applying the chain rule approach.

To numerically solve the optimization problem $J(\Omega) \to \inf$, we will apply a shape-gradient-based descent method based on finite element method (FEM).
We will not, however, employ any kind of adaptive mesh refinement in our numerical scheme as opposed to \cite{RabagoAzegami2019a,RabagoAzegami2019b,RabagoAzegami2020}.
In this way, we can further assess the stability of the new method in comparison with the Kohn-Vogelius approach.
Our method is of course different from \cite{EpplerHarbrecht2009,EpplerHarbrecht2010,EpplerHarbrecht2012a} which make use of the boundary element method to numerically solve the shape optimization problem studied in the said papers.
The expression for the shape derivative of the cost will be exhibited in the next section using \textit{shape calculus} \cite{DelfourZolesio2011,HenrotPierre2018,MuratSimon1976,Simon1980,SokolowskiZolesio1992}.
We point out though that as opposed to \cite{EpplerHarbrecht2009,EpplerHarbrecht2010,EpplerHarbrecht2012a} where the chain rule approach is used to compute the shape derivative of the cost, our approach does not use the strong form of the shape (Eulerian) derivative of the state, but instead only make use of the weak form of the material (Lagrangian) derivative of the state problem (see Remark \ref{rem:detour_remark}).
This technique also applies to computing the second-order shape derivative of the cost which requires the domain to be only of class $\mathcal{C}^{1,1}$.

Throughout the rest of the paper, we shall refer to the Kohn-Vogelius method as KVM, or just KV, and to our proposed shape optimization method simply by CCBM.
%
%
\subsection{Motivations} To further motivate the study of the proposed shape optimization formulation $J(\Omega) \to \inf$, we discuss here about some key advantages of CCBM over the KVM \eqref{eq:KV_method} (and to other traditional approaches) while giving some additional observations about the two formulations.

From the theoretical point of view, the computation of the first-order shape derivative of the cost function $J$ is simpler compared to $J_{KV}$ when computed via a Lagrange formulation or through the minimax formulation \cite{DelfourZolesio1988a}.
This is because the corresponding Lagrangian expression for CCBM is composed of only one equality constraint (this corresponds to problem \eqref{eq:complexPDE}) in addition to the objective function $J$.
In the case of KVM, the Lagrangian consists of the cost function $J_{KV}$ and two equality constraints corresponding to a mixed Dirichlet-Neumann problem and a pure Dirichlet problem.
Moreover, in order to take into account in the Lagrangian functional the Dirichlet data on the free boundary for KVM, one needs to introduce a suitable Lagrangian multiplier in constructing the functional.
Such additional requirement is not needed in the case of the CCBM formulation.
In addition, when applying instead the chain rule approach to get the shape derivative of the cost $J$ through the shape derivative of the states, one only needs to compute and justify the existence of the shape derivative of a single complex PDE in contrast to the case of KVM where one has to deal with the shape derivative of two systems of PDEs.
Certainly, the latter would be more involved, especially when considering more general and complicated boundary value problems.
Also, we add that the formulation -- under suitable assumptions -- offers more regularity for the solution of the corresponding adjoint problem (see Remark \ref{rem:higher_regularity_of_the_adjoint}) than in the case of formulating \eqref{eq:Bernoulli_problem} into a shape optimization setting that utilizes a boundary-data-tracking-type cost functional.
Meanwhile, from the numerical viewpoint, CCBM appears to have nearly the same computational complexity with respect to KVM.
Nonetheless, it features some notable advantages in terms of overall performance over KVM, but also presents some downsides.
On the one hand, CCBM only needs to solve one complex state problem to evaluate its corresponding cost function.
On the other hand, it also requires to solve another PDE system (this corresponds to the adjoint problem \eqref{eq:adjoint_system} associated with the formulation) in order to calculate its \textit{shape gradient}.
This is as opposed to KVM which does not need the introduction of an adjoint problem.
Also, we stress here that, when applying FEM for instance, one does not need to solve the coupled problems \eqref{eq:real_state} and \eqref{eq:imaginary_state} simultaneously since it is enough to solve the variational problem corresponding to \eqref{eq:complexPDE}. 
Even so, the time needed to compute the solution to \eqref{eq:complexPDE} is practically the same with the time required to solve the two state problems associated with the KV formulation via a finite element scheme.
These observations may not seem to provide any advantages to CCBM over KVM, but computational results -- specifically obtained under large deformations of the domains -- provided in Examples \ref{example2d2}--\ref{example2d3} of subsection \ref{subsec:Numerical_Examples} show otherwise.
In fact, in many of these performed experiments, CCBM seems to require less \textit{overall} computational-time-per-iteration of the approximate shape solution to \eqref{eq:Bernoulli_problem} compared to KVM when solve via a gradient-based scheme, not to mention that, in some instances, KVM converges prematurely and sometimes tends to overshoot the optimal shape not like CCBM.
In general, KVM converges faster than CCBM and that the two methods complete the same number of iterations at nearly the same time.
Moreover, as we observed in our numerical experiments, CCBM seems to have some sort of a similar smoothing effect that has been observed in \cite{RabagoAzegami2019a} when replacing the Dirichlet boundary condition with a Robin boundary condition in a Neumann-data-tracking approach.
In addition, we become aware -- after conducting numerous numerical experiments -- that the cost function $J_{KV}$ is less sensitive compared to $J$.
In fact, we notice that $J_{KV}(\Omega)$ is insensitive to large deformations of $\Omega$.
In some sense, this means that, after some iterations, the $k$th approximant $\Omega^{k}$ of the exact shape solution $\Omega^\ast$ gets stuck to some certain geometry and so will not get closer to the exact solution even after an additional number of iterations.
The insensitivity of $J_{KV}$ to large variations limits the algorithm to take smaller step sizes at each iterations -- so to avoid being stuck at certain point which may not yet be considered optimal. 
Surely, in some situations, smaller step sizes may result to slower convergence behavior.
Be that as it may, the two methods provide almost the same optimal solution to the given problem.
These claims are made evident in subsection \ref{subsec:Numerical_Examples} where numerical examples are provided.
Last but not least, we point out that there is another domain-integral-type penalization that may be considered.
More exactly, one may opt to penalize, instead of $J_{KV}$, by the cost functional 
\[
	J_{L^2}(\Omega) := \frac{1}{2} \intO{(\un - \ud)^2}. 
\]
Compared with $J_{KV}$, however, its shape gradient is more complex in structure and would be numerically expensive to evaluate.
In fact, it can be checked that in this case one would need to solve five real BVPs: two state equations, two adjoint equations (cf. \cite{LaurainPrivat2012}), and one to approximate the mean curvature of the free boundary implicitly.

The discussion given above, as well as the numerical findings issued in subsection \ref{subsec:Numerical_Examples}, clearly warrants the proposal of the new method as a way to solve stationary free boundary problems such as \eqref{eq:FBP} and \eqref{eq:Bernoulli_problem}.
\section{Shape Derivatives}
\label{sec:Shape_Derivatives}
This section is devoted to proving that the map $t \mapsto \ut$ is $\mathcal{C}^1$ in a neighborhood of zero (Proposition \ref{prop:umaps}), and to the characterization of its derivative.
Similar results for the map $t \mapsto J(\Omega_t)$ (Theorem \ref{prop:Jmaps}) is also presented herein.
On this purpose, we need the notions of shape derivatives of functions and shape functionals, and those we give in the first subsection.
Meanwhile, a rigorous proof of existence of the Lagrangian derivative of the state -- presented specifically in its weak form -- will be given in subsection \ref{subsec:Lagrangian_derivative_of_the_state}.
Finally, the shape derivative of the cost functional is exhibited in subsection \ref{subsec:shape_derivative_of_the_cost_using_the_Eulerian_derivatives}.

\subsection{Some elements of shape calculus} To accomplish our main tasks, we apply the concept of \textit{velocity} or \textit{speed} method (see, e.g., \cite[Chap. 4]{DelfourZolesio2011} or \cite{DelfourZolesio1991b}).
We let $\mathcal{D}^{k}(\mathbb{R}^d; \mathbb{R}^d)$ be the space of $k$-times continuously differentiable functions with compact support contained in $\mathbb{R}^d$.
Consider $\VV \in \mathcal{E}^k:=\mathcal{C}([0, \varepsilon); \mathcal{D}^{k}(\mathbb{R}^d; \mathbb{R}^d))$, $\mathbb{N} \ni k \geqslant 2$, and let $\varepsilon > 0$ be a small real number.
The field  $\VV(t)(x) = \VV(t,x)$, $x \in \mathbb{R}^d$, generates the transformations $T_t(\VV)(X) := T_t(X) = x(t; X)$, $t \geqslant 0$, $X \in \mathbb{R}^d$, through the differential equation $\dot{x}(t; x_0) = \VV(t, x(t; x_0))$, $x(0;x_0) = x_0$, with the initial value $x_0$ specified. 
We denote the ``transformed domain'' $T_t(\VV)(\Omega)$ at $t \geqslant 0$ by $\Omega_t=:T_t(\Omega)$, 
and assumed that all \textit{admissible} deformations of $\Omega$ is contained in a larger open, bounded, and connected set $U \subset \mathbb{R}^d$ (e.g., a ball in $\mathbb{R}^d$) of class $\mathcal{C}^{k,1}$.  
In this work, the evolutions of the reference domain $\Omega$ is described using time-independent velocity fields $\VV$ such that
\begin{equation}
\label{eq:space_for_V}
	\VV \in \sfTheta^{k}:=\{\VV \in \mathcal{C}^{k,1}(\overline{\Omega})^{d} \mid \VV = 0 \ \text{on} \ \Gamma \cup \partial U\}.
\end{equation}
Here, we are fixing the interior boundary $\Gamma$ by taking $\VV = 0$ on $\Gamma$.
We note that it is enough to consider transformations $T_t$ that only change the position of the free part $\Sigma$ of $\partial \Omega$, but do not rotate it.
In other words, we may just consider vector fields that have zero tangential part along $\Sigma$; that is, $\VV|_{\Sigma} = (\VV\cdot\nn)\nn = \Vn\nn$. 

\sloppy We say that the function $u$ has a \textit{material} derivative $\dot{u}$ and a \textit{shape} derivative $u'$ at $0$ in the direction of the vector field $\VV$ if the limits $\dot{u} = \lim_{t \searrow 0} \frac{1}{t}\left(u(\Omega_t) \circ T_t - u(\Omega) \right)$ and $u' = \lim_{t \searrow 0} \frac{1}{t}\left( u(\Omega_t) - u(\Omega)\right)$, exist, respectively, where $(u(\Omega_t) \circ T_t)(x) = u(\Omega_t)(T_t(x))$.
These expressions are related by $u' = \dot{u} - (\nabla u \cdot \VV)$ provided that $\nabla u \cdot \VV$ exists in some appropriate function space \cite{DelfourZolesio2011,SokolowskiZolesio1992}. 

Given a shape functional $\sfj : \Omega \to \mathbb{R}$, we say that it has a directional Eulerian derivative at $\Omega$ in the direction of $\VV$ if the limit $\lim_{t \searrow0} \frac{1}{s}\left(\sfj(\Omega_t) - \sfj(\Omega)\right) =: \operatorname{\mathnormal{d}} \sfj(\Omega)[\VV]$ exists (cf. \cite[Eq. (3.6), p. 172]{DelfourZolesio2011}). 
If the map $\VV \mapsto \operatorname{\mathnormal{d}} \sfj(\Omega)[\VV]$ is linear and continuous, then $j$ is \textit{shape differentiable} at $\Omega$, and the map is referred to as the \textit{shape gradient} of $j$.
Similarly, the \textit{second-order} Eulerian derivative of $j$ at $\Omega$ along the two fields $\VV$ and $\WW$ is given by $\lim_{s \searrow0} \frac{1}{s}\left( \operatorname{\mathnormal{d}}\sfj(\Omega_s(\WW))[\VV] - \operatorname{\mathnormal{d}}\sfj(\Omega)[\VV] \right) =: {\operatorname{\mathnormal{d}}}^2\sfj(\Omega)[\VV,\WW]$
if the limit exists \cite[Def. 2.3]{DelfourZolesio1991a}.
In addition, $\sfj$ is said to be \textit{twice shape differentiable} if, for all $\VV$ and $\WW$, ${\operatorname{\mathnormal{d}}}^2\sfj(\Omega)[\VV,\WW]$ exists, and is bilinear and continuous with respect to $\VV, \WW$.
Accordingly, we call the expression as the \textit{shape Hessian} of $\sfj$.

To complete our preparation, we also mention some properties of $T_t := T_t(\VV)$, $\VV \in \sfTheta^1$, that are essential to our investigation.
For sufficiently small $t>0$, $T_t$ and its inverse $T_t^{-1}$ are both in $\mathcal{D}^{1}(\mathbb{R}^d; \mathbb{R}^d)$ \cite[Chap. 4]{DelfourZolesio2011} where $d$ is the space dimension.
Moreover, for $t \in [0, \varepsilon)$, $\varepsilon > 0$ can be chosen sufficiently small so that $I_t := \det \,{\mathnormal{D}}T_t > 0$.
Throughout the paper, we use the notations $A_t := I_t({\mathnormal{D}}T_t^{-1})({\mathnormal{D}}T_t)^{-\top}$ and $B_{t} = I_t |({\mathnormal{D}}T_t)^{-\top} \nn|$, and assume that, for all $t \in [0, \varepsilon)$, $I_t > 0$ and there exist constants $\Lambda_0, \Lambda_1$ ($0 < \Lambda_0 < \Lambda_1$) such that 
$\Lambda_0 \leqslant B_t \leqslant \Lambda_1$ and $\Lambda_0|\xi|^2 \leqslant A_t\xi \cdot \xi \leqslant \Lambda_1|\xi|^2$ for all $\xi \in \mathbb{R}^d$.
We also note the following regularities: $[t \mapsto I_t] \in \mathcal{C}^1(\mathcal{I},\mathcal{C}(\overline{\Omega}))$, $[t \mapsto A_t] \in \mathcal{C}^1(\mathcal{I},\mathcal{C}(\overline{\Omega})^{d \times d})$, and $[t \mapsto B_t] \in \mathcal{C}^1(\mathcal{I},\mathcal{C}(\Sigma))$, on a neighborhood $\mathcal{I}$ of $0$ such that, for $t \in \mathcal{I}$, $T_t$ is a diffeomorphism of $\Omega$ onto $\Omega_t$ (see, e.g., \cite{IKP2006}).
Lastly, we note the following derivatives: $\dot{I}_t \big|_{t=0} := (d/dt)I_t \big|_{t=0} = {\operatorname{div}}\, \VV$, $\dot{A}_t \big|_{t=0} = A := ({\operatorname{div}}\VV)\vect{I} -  \operatorname{\mathnormal{D}}\!\VV - (\operatorname{\mathnormal{D}}\!\VV)^\top$, and $\dot{B}_t \big|_{t=0} = {\operatorname{div}}_{\Sigma} \VV = {\operatorname{div}} \VV \big|_{\Sigma} - (\operatorname{\mathnormal{D}}\!\VV\nn)\cdot\nn$. 
\subsection{Lagrangian derivatives of the state}
\label{subsec:Lagrangian_derivative_of_the_state}
The main point of this section is to prove the following proposition concerning the material derivative of the state.
Throughout the section, $\Omega$ is assumed to be of class $\mathcal{C}^{1,1}$ and $\VV \in \sfTheta^1$.
This implies that, for sufficiently small $t>0$, $T_t$ is a $\mathcal{C}^{1,1}$ diffeomorphism of $\Omega$ onto $\Omega_t$.
\begin{proposition}
	\label{prop:umaps}
	Let $\Omega$ be of class $\mathcal{C}^{1,1}$ and $\VV \in \sfTheta^1$.
	The map $t \mapsto \ut \in \HH^{1}(\Omega)$ is $\mathcal{C}^1$ in a neighborhood of $0$.
	Its Lagrangian derivative at $0$, denoted by $\dotu$, belongs to $\HHg(\Omega)$ and satisfies
	\begin{equation}\label{eq:Lagrangian_derivative_of_the_state}
	\begin{aligned}
	\intO{\nabla {\dotu} \cdot \nabla {\cv}}  + i \intS{\dotu {\cv} }  
		&= -\intO{A \nabla {u} \cdot \nabla {\cv}} - i \intS{ ({\operatorname{div}}_{\Sigma} \VV) {u} {\cv}} \nonumber\\
		&\qquad + \lambda \intS{ ({\operatorname{div}}_{\Sigma} \VV) {\cv}}, \quad \forall {v} \in \HHg(\Omega)
	\end{aligned}
	\end{equation}
\end{proposition}
To prove the above result, we need the following lemmas.
\begin{lemma}\label{lem:sesquiliner_form}
	The sesquilinear form $\aat$ defined on $\HHg(\Omega) \times \HHg(\Omega)$ by
	\[
		\aat({u},{v}) = \intO{A_t \nabla {u} \cdot \nabla {\cv}} + i \intS{B_t {u} {\cv}}, \quad \forall {u}, {v} \in \HHg(\Omega),
	\]
	is bounded and coercive on $\HHg(\Omega) \times \HHg(\Omega)$ for sufficiently small $t > 0$.
\end{lemma}
\begin{proof}
	\sloppy For ${u}$, ${v} \in \HHg(\Omega)$ and sufficiently small $t>0$, we have, by Cauchy-Schwarz and trace inequalities, the estimate $|\aat({u}, {v})| \leqslant c_t \vertiii{u}_{\HHg(\Omega)}  \vertiii{v}_{\HHg(\Omega)}$ where $c_t = \max\left( |A_t|_{\infty}, |B_t|_{\infty} \right)$.
	This shows that $\aat$ is bounded.
	Also, we have that\footnote{Here, $\vect{I}$ stands for the identity matrix in $d=2$ dimensions.}
\[
	\aat({u}, {u}) = \intO{ \left\{ \nabla u \cdot \nabla \cu + (A_t - \vect{I}) \nabla {u} \cdot \nabla {\cu} \right\}}
					+ i \intS{ \left\{ u \cu + (B_t - 1){u}\cu \right\}}.
\]
	Because the maps $t \mapsto A_t$ and $t \mapsto B_t$ are continuous at $0$, then -- for sufficiently small $t$ -- $\max\left( |A_t - \vect{I}|_{\infty}, |B_t - 1|_{\infty} \right) < 1$. 
	Consequently, $\Re\{\aat({u}, {u})\} \gtrsim \left( \vertiii{\nabla {u}}^2_{\QQ} + \vertiii{u}^2_{\SSig} \right)$.
	Hence, $\aat$ is coercive for small enough $t$.
\end{proof}
\begin{lemma}
	\label{lem:transported_problem}
	The function $u^t=u_{1}^t + i u_{2}^t$ uniquely solves in $\HH^{1}(\Omega)$ the equation
	\begin{equation}\label{eq:transformed_state}
	\begin{aligned}
	\intO{A_t\nabla {u^t} \cdot \nabla {\cv}} 
					+ i \intS{ B_{t} {u^t} {\cv} } &= \lambda \intS{B_{t} {\cv}}, \ \ \forall {v} \in \HHg(\Omega),\\
	u^t &= 1\ \text{on}\ \Gamma.	
	\end{aligned}
	\end{equation}
\end{lemma}	
\begin{proof}
	The function ${u_t} \in \HH^{1}(\Omega_t)$ satisfies ${u_t} = 1$ on $\Gamma$, and solves the variational problem
	\[
	\intOt{\nabla {u_t} \cdot \nabla \cphi_t}
					+ i \intSt{{u_t} \cphi_t} = \lambda \intSt{\cphi_t}, \quad \forall \varphi_t \in \HHg(\Omega_t),
	\]
	where $\HHg(\Omega_t) = \{\varphi_t \in \HH^{1}(\Omega_t) \mid  \varphi_t = 0 \ \text{on} \ \Gamma\}$.
	Using the relation $u^t = {u_t} \circ T_t := {u_{1t}} \circ T_t + i {u_{2t}} \circ T_t$, the identity $(\nabla \varphi_t) \circ T_t = DT_t^{-\top} \nabla \varphi^t$ which holds for any $\varphi_t \in \HH^{1}(\Omega_t)$ and $\varphi^t \in \HH^{1}(\Omega)$, and the change of variables (cf. \cite[subsec. 9.4.2--9.4.3, pp. 482--484]{DelfourZolesio2011}), we get ${u^t} = 1$ on $\Gamma$ and
	the variational equation above transforms to
	\[
		\intO{A_t \nabla {u^t} \cdot \nabla \cphi^t} + i \intS{B_t {u^t} \cphi^t } = \lambda \intS{B_t \cphi^t}, \quad \forall \varphi^t \in \HHg(\Omega).
	\]

	We immediately get \eqref{eq:transformed_state} by taking $v = \varphi^t$ above.
	Now, from Lemma \ref{lem:sesquiliner_form}, $\aat(\cdot, \cdot) : [\HHg(\Omega)]^2 \to \mathbb{R}$ is bounded and coercive.
	We let $u_0 \in \HH^{1}(U)$ be a fixed function such that $u_0 = 1$ on $\Gamma$.
	Then, $u^t - u_0 \in \HHg(\Omega)$, and by \eqref{eq:transformed_state}, we have
	\begin{align*}
	 &\intO{A_t\nabla {(u^t - u_0)} \cdot \nabla {\cv}}
		+ i \intS{ B_{t} {(u^t - u_0)} {\cv} } \\
			&\qquad =  - \intO{A_t\nabla u_0 \cdot \nabla {\cv}}
						- i \intS{B_{t} u_0 {\cv}}
							+ \lambda \intS{ B_{t} {\cv} }, \quad \forall {v} \in \HHg(\Omega).
	\end{align*}

	For $t>0$ small enough, the following estimates hold
	\begin{align*}
	\left| \intS{ B_{t} u_0 {\cv}} \right|
		&\lesssim |B_{t}|_{\infty} \vertiii{u_0}_{\HH^{1}(U)} \vertiii{v}_{\HH^{1}(\Omega)}, \\
		\left| \lambda \intS{ B_{t} {\cv}} \right|
			&\lesssim |B_{t}|_{\infty} |\Sigma|^{1/2}\vertiii{v}_{\HH^{1}(\Omega)},\\	
	\left| \intO{A_t\nabla u_0 \cdot \nabla {\cv}} \right|
			&\leqslant |A_{t}|_{\infty} \vertiii{u_0}_{\HH^{1}(U)} \vertiii{v}_{\HH^{1}(\Omega)}.
	\end{align*}

	By Lax-Milgram lemma, $z^t=u^t - u_0 \in \HHg(\Omega)$ is the unique solution in $\HHg(\Omega)$ of
	\begin{align*}
	 &\intO{A_t\nabla {z^t} \cdot \nabla {\cv}}
		+ i \intS{ B_{t} {z^t} {\cv} } \\
			&\qquad =  - \intO{A_t\nabla u_0 \cdot \nabla {\cv}}
						- i \intS{B_{t} u_0 {\cv}}
							+ \lambda \intS{ B_{t} {\cv} }, \quad \forall {v} \in \HHg(\Omega).
	\end{align*}

	Now, let $\ut = z^t + u_0 \in \HH^{1}(\Omega)$.
	Then, we have
	\begin{align*}
	&\intO{A_t\nabla \ut \cdot \nabla {\cv}} + i \intS{ B_{t} \ut {\cv} } \\
	&\quad \qquad \qquad =  \intO{A_t\nabla (z^t + u_0) \cdot \nabla {\cv}} + i \intS{ B_{t} (z^t + u_0) {\cv} }\\
	&\quad \qquad  \qquad = \lambda \intS{ B_{t} {\cv}}, \quad \forall {v} \in \HHg(\Omega).
	\end{align*}

	Clearly, $u^t = z^t + 1 = 1$ on $\Gamma$ because $z^t \in \HHg(\Omega)$.
	Uniqueness of $u^t$ follows from the uniqueness of $z^t$.
	That is, $u^t$ uniquely solves \eqref{eq:transformed_state} in $\HH^{1}(\Omega)$.
\end{proof}
\begin{lemma}
	The map $t \mapsto {u^t}$ is $\mathcal{C}^1$ in a neighborhood of $0$.
\end{lemma}
\begin{proof}
	We prove the statement using the implicit function theorem (IFT).
	In view of \eqref{eq:transformed_state}, we see that ${u^t} - {u}$ is the unique element in $\HHg(\Omega)$ that satisfies
	 \begin{align*}
	&\intO{A_t\nabla ({u^t} - {u}) \cdot \nabla {\cv}} + i \intS{ B_{t} ({u^t} - {u}) {\cv} } \\
	 	&\qquad\qquad =  - \intO{A_t\nabla {u} \cdot \nabla {\cv}} - i \intS{ B_{t} {u} {\cv} }  + \lambda  \intS{B_{t} {\cv}},
		\quad \forall {v} \in \HHg(\Omega).
	 \end{align*}
	 Denoting by $\langle \cdot , \cdot \rangle$ the duality pairing between $\HHg(\Omega)$ and its dual space $\HHg'(\Omega)$, we consider the function $\sfPsi : \mathcal{I} \times \HHg(\Omega) \to \HHg'(\Omega)$ defined by
	 \[
	 	\langle \sfPsi(t,w), v \rangle := \intO{A_t \nabla (w+{u}) \cdot \nabla {\cv}} + {i} \intS{ B_{t} (w+{u}) {\cv} } - \lambda \intS{ B_{t} {\cv}},
	 \]
	 for all $v$, $w\in \HHg(\Omega)$.
	 Clearly, $\sfPsi$ is $\mathcal{C}^1$ because $t \mapsto A_t$ and $t \mapsto B_t$ are $\mathcal{C}^1$ in a neighborhood of $0$.
	 We observe that ${u^t} - {u}$ uniquely solves $\sfPsi(t, {u^t} - {u}) = 0$ in $\HHg(\Omega)$.
	 In addition, $\langle \operatorname{\mathnormal{D}}\!_{w}\sfPsi(0,0)w, v \rangle = \aat(w,v)$. 
	 Using Lemma \ref{lem:sesquiliner_form}, we deduce via the complex version of Lax-Milgram lemma that $\operatorname{\mathnormal{d}}\!_{w}\sfPsi(0,0)$ is an isomorphism from $\HHg(\Omega)$ to $\HHg'(\Omega)$.
	 By the IFT, we conclude that the map $t \mapsto {u^t} - {u}$ is $\mathcal{C}^1$ in a neighborhood of $0$.
	 Now let $\dotu \in \HHg(\Omega)$ be its derivative at $t=0$.
	 Differentiating $\sfPsi(t, {u^t} - {u}) = 0$ with respect to $t$, we obtain
	 $\left\langle \operatorname{\mathnormal{D}}\!_{w}\sfPsi(0,0)\dotu, v \right\rangle + \left\langle \frac{\partial }{\partial t}\sfPsi(0,0), v \right\rangle = 0$,
	 for all $v \in \HHg(\Omega)$, leading to \eqref{eq:Lagrangian_derivative_of_the_state}.	 
\end{proof}
%
\subsection{Computation of the shape gradient}
\label{subsec:shape_derivative_of_the_cost_using_the_Eulerian_derivatives}
We now prove the differentiability of the map $t \mapsto J(\Omega_t)$ and characterize its derivative.

\begin{theorem}[Shape gradient of $J$]
	\label{prop:Jmaps}
	Let $\Omega$ be of class $\mathcal{C}^{1,1}$ and $\VV \in \sfTheta^1$.
	The map $t \mapsto J(\Omega_t)$ is $\mathcal{C}^1$ in a neighborhood of $0$, and its derivative at $0$ is given by ${\operatorname{\mathnormal{d}}}J(\Omega)[\VV] = \intS{\GG \nu \cdot \VV}$ where
	\begin{equation}\label{eq:shape_gradient}
	\begin{aligned}
		\GG &= \frac12 |\ui|^2  - \Big[  \nabla_{\Sigma} \vr \cdot \nabla_{\Sigma} \ui - \nabla_{\Sigma} \vi \cdot \nabla_{\Sigma} \ur \\
			& \qquad\qquad\qquad + \vr \left( \dn{\ur} + \kappa \ur \right) + \vi \left( \dn{\ui} + \kappa \ui \right) + \lambda \kappa \vi  \Big],
	\end{aligned}
	\end{equation}
	$\kappa$ stands for the mean curvature of the free boundary $\Sigma$, $u = \ur + i \ui$ is the unique solution to \eqref{eq:complexPDE}, and $p = \vr + i \vi$ uniquely solves the adjoint system
	\begin{equation}\label{eq:adjoint_system}
		-\Delta p 		=\ui \ \text{in $\Omega$},\qquad\quad
		p 			=0 \ \text{on $\Gamma$},\qquad\quad
		\dn{p} - i p		=0 \ \text{on $\Sigma$}.
	\end{equation}		
\end{theorem}
\begin{remark}
	The weak formulation of \eqref{eq:adjoint_system} reads as follows:	
	\begin{equation}\label{eq:adjoint_equation}
	\text{	find $p \in \HHg(\Omega)$ such that} \intO{ \nabla p \cdot \nabla {\cphi}} - i \intS{ p {\cphi} } = \intO{ \ui {\cphi} }, \ \forall \varphi \in \HHg(\Omega).
	\end{equation}
	The existence and uniqueness of solution to the variational problem \eqref{eq:adjoint_equation} is again a consequence of the Lax-Milgram lemma.
	\end{remark}
%

\begin{remark}\label{rem:higher_regularity_of_the_adjoint}
	In general, for $\Omega$ of class $\mathcal{C}^{k+1,1}$, $k$ a non-negative integer, it can be shown that the weak solution $u \in \HH^{1}(\Omega)$ to the variational problem \eqref{eq:state_weak_form} is also $\HH^{k+2}(\Omega)$.
	In particular, $u_{2} \in H^{k+2}(\Omega)$.
	Consequently, we find that the weak solution $p$ of problem \eqref{eq:adjoint_equation} is not only in $\HH^{1}(\Omega)$, but is also an element of $\HH^{k+4}(\Omega)$.
\end{remark}

Theorem \ref{prop:Jmaps} relies on following lemma.
The $\HH^2(\Omega)$ regularity of ${u}$, which holds true since $\Omega$ is assumed to be $\mathcal{C}^{1,1}$ and $\VV \in \sfTheta^1$, will be used subsequently without further notice.	
\begin{lemma}
\label{eq:integral_equivalence}
	Let $\Omega$ be of class $\mathcal{C}^{1,1}$ and $\VV \in \sfTheta^1$.
	The solution ${u}$ of the state problem \eqref{eq:complexPDE} and the adjoint variable $p$ which is the solution to \eqref{eq:adjoint_system} satisfy the equation
	\begin{equation} \label{eq:domain_identity}
	\begin{aligned}
	\intO{A \nabla {u} \cdot \nabla \overline{p}}  
	&= - \intO{ u_2 \VV \cdot \nabla {u} } + \intS{(\nabla \overline{{p}} \cdot \nabla {u})\Vn} 
		+ i \intS{{u}(\VV \cdot \nabla \overline{{p}})}\\
	&\qquad\quad + i \intS{ \overline{p} (\VV \cdot \nabla {u})} - \lambda \intS{ (\VV \cdot \nabla \overline{{p}})}.
	\end{aligned}
	\end{equation}
\end{lemma}
\begin{proof}
The result follows from the formula
\begin{equation}\label{eq:expansion}
\begin{aligned}
	\intO{A\nabla \varphi \cdot \nabla \overline{\psi} }
	&= \intO{ (\Delta \varphi) \VV \cdot \nabla \overline{\psi} }
		+ \intO{ (\Delta \overline{\psi}) \VV \cdot \nabla \varphi }
		- \intS{\dn{\varphi}(\VV \cdot \nabla \overline{\psi})} \\
	&\quad \qquad 
		 - \intS{\dn{\overline{\psi}}(\VV \cdot \nabla \varphi)}
		+ \intS{(\nabla \overline{\psi} \cdot \nabla \varphi)\Vn},
\end{aligned}
\end{equation}
which holds for all functions $\varphi$, $\psi \in \HHg(\Omega) \cap \HH^2(\Omega)$ and $\VV \in \sfTheta^1$, where $A = ({\operatorname{div}}\VV)\vect{I} -  \operatorname{\mathnormal{D}}\!\VV - (\operatorname{\mathnormal{D}}\!\VV)^\top$ (see, e.g., \cite{DelfourZolesio2011}). 
Letting $\varphi = {u}$ and $\psi = p$, and noting that $\dn{\overline{p}} = -i \overline{p}$ and $\dn{u} = -i {u} + \lambda$ on $\Sigma$ and that $-\Delta \overline{p} = u_2$ in $\Omega$, we obtain
\[
	\begin{aligned}
	\intO{A \nabla {u} \cdot \nabla \overline{p}}  
	&= - \intO{ u_2 \VV \cdot \nabla {u} } + i \intS{{u}(\VV \cdot \nabla \overline{{p}})} + i \intS{ \overline{p} (\VV \cdot \nabla {u})}\\
	&\qquad - \lambda \intS{ (\VV \cdot \nabla \overline{{p}})} + \intS{(\nabla \overline{{p}} \cdot \nabla {u})\Vn},
	\end{aligned}
\]
as desired.
This proves the lemma.
\end{proof}
%
%
%
\begin{proof}[Proof of Theorem \ref{prop:Jmaps}]
By change of variables, we write $J(\Omega_t) = \frac12 \intO{ I_t |u_2^t|^2}$.
Because $[t \mapsto I_t] \in \mathcal{C}^1(\mathcal{I},\mathcal{C}(\overline{\Omega}))$ and $[t \mapsto \ut] \in \mathcal{C}^1(\mathcal{I},\HH^{1}(\Omega))$, then the mapping $t\mapsto J(\Omega_t)$ is also $\mathcal{C}^1$ in the neighborhood $\mathcal{I}$ of $0$, and we have
\begin{equation}\label{eq:material_derivative_of_the_cost}
	{\operatorname{\mathnormal{d}}}J(\Omega)[\VV]
		= \frac12 \intO{ ({\operatorname{div}}\, \VV) |\ui|^2} + \intO{ \ui \dotu{}_2 }.
\end{equation}
On the one hand, using the identity $-\intO{({\operatorname{div}}\,\vect{v}) \phi} = \intO{\vect{v} \cdot \nabla \phi} - \intS{\phi(\vect{v} \cdot \nu)}$, which holds (on bounded Lipschitz domain $\Omega$) for any vector field $\vect{v} \in \mathcal{C}^{1}(\overline{\Omega})^{2}$ and scalar function $\phi \in W^{1,1}(\Omega)$, the first integral can be equivalently written as
\begin{equation}\label{eq:first_integral}
	\frac12 \intO{({\operatorname{div}}\,\VV) |\ui|^2} 
	= - \intO{\ui \VV \cdot \nabla \ui } + \frac12 \intS{ |\ui|^2\Vn}.
\end{equation}
On the other hand, the second integral is actually not useful for practical applications, especially in the numerical realization of the present shape minimization problem because it requires the solution of \eqref{eq:Lagrangian_derivative_of_the_state} for each velocity field $\VV$.
A way to resolve this issue is to rewrite the integrand in terms of another variable -- getting rid of the term $\dotu{}$ -- through the adjoint method.
To this end, we utilize the variational problem \eqref{eq:adjoint_equation} with the test function $\varphi = \dotu{} \in \HHg(\Omega)$ to obtain
\begin{equation}\label{eq:identity_one}
\intO{ \ui \overline{\dotu{}} } = \intO{ \nabla p \cdot \nabla \overline{\dotu{}} } - i \intS{ p \overline{\dotu{}}}.
\end{equation}
Now, choosing $v = p \in \HHg(\Omega)$ as the test function in \eqref{eq:Lagrangian_derivative_of_the_state} gives us
\begin{equation}\label{eq:identity_two}
	\begin{aligned}
	\intO{\nabla {\dotu} \cdot \nabla {\overline{p}}}  + i \intS{\dotu {\overline{p}} }  
		&= -\intO{A \nabla {u} \cdot \nabla {\overline{p}}} - i \intS{ ({\operatorname{div}}_{\Sigma} \VV) {u} {\overline{p}}}\\
		&\qquad + \lambda \intS{ ({\operatorname{div}}_{\Sigma} \VV) {\overline{p}}}.
	\end{aligned}	
\end{equation}
Hence, taking the complex conjugate of both sides of \eqref{eq:identity_one} and then comparing the resulting equation to \eqref{eq:identity_two} leads us to the identity
\[
	\begin{aligned}
	\intO{ \ui \dotu{} }  
		&= -\intO{A \nabla {u} \cdot \nabla {\overline{p}}} - i \intS{ ({\operatorname{div}}_{\Sigma} \VV) {u} {\overline{p}}} + \lambda \intS{ ({\operatorname{div}}_{\Sigma} \VV) {\overline{p}}}.
	\end{aligned}	
\]
Applying \eqref{eq:domain_identity} in Lemma \ref{eq:integral_equivalence}, we can further write
\[
	\begin{aligned}
	\intO{ \ui \dotu{} }  
		&=  \intO{ u_2 \VV \cdot \nabla {u} } -i \intS{ \overline{p} (\VV \cdot \nabla {u})} - i \intS{{u}(\VV \cdot \nabla \overline{{p}})} \\
		&\qquad - i \intS{ {u} {\overline{p}}\, {\operatorname{div}}_{\Sigma} \VV} - \intS{(\nabla \overline{{p}} \cdot \nabla {u})\Vn}\\
		&\qquad + \lambda \intS{ \nabla \overline{{p}} \cdot \VV } + \lambda \intS{ {\overline{p}}\, {\operatorname{div}}_{\Sigma} \VV }.
	\end{aligned}	
\]
At this point we apply the following version of the tangential Green's formula\footnote{A proof of this formula can be found in \cite{MuratSimon1976}.}, which is valid, for instance, when $\Sigma$ is $\mathcal{C}^{1,1}$,
\begin{equation}
	\label{eq:tangential_Greens_formula}
	\intS{\left( \nabla \phi \cdot \VV + \phi \, {\operatorname{div}}_{\Sigma} \VV \right) } 
		=  \intS{ \left( \dn{\phi}  +\phi\, {\operatorname{div}}_{\Sigma} \nn \right) {\Vn}}.
\end{equation}
Here, the function $\phi$ is supposed to be $W^{2,1}(U)$ regular.
We also note the fact that $\kappa =  {\operatorname{div}}_{\Sigma} \nn$, where $\kappa$ is the mean curvature of $\Sigma$, and utilize the identity $\nabla \overline{{p}} \cdot \nabla {u} = \nabla_{\Sigma} \overline{{p}} \cdot \nabla_{\Sigma} {u} + \dn{\overline{p}} \dn{u}$\footnote{Here $\nabla_{\Sigma}$ is the tangential gradient operator on $\Sigma$. The intrinsic definition of the operator is given in \cite[Chap. 5., Sec. 5.1, p. 492]{DelfourZolesio2011}. More discussion on tangential calculus can be found in the same referenced text.} on $\Sigma$ to obtain
\[
	\begin{aligned}
	\intO{ \ui \dotu{} }  
		&=  \intO{ u_2 \VV \cdot \nabla {u} } - i \intS{ \overline{p}  \left( \dn{u} + \kappa u \right) \Vn} \\
		&\qquad -  \intS{\nabla_{\Sigma} \overline{{p}} \cdot \nabla_{\Sigma} {u} \Vn}
				 + \lambda  \intS{ \kappa \overline{p} \Vn }.
	\end{aligned}	
\]
The desired expression immediately follows by comparing the respective complex parts on both sides of the above equation, and together with \eqref{eq:first_integral}.
\end{proof}
%
%
%
The next conclusion can be drawn easily from \eqref{eq:shape_gradient}, \eqref{eq:adjoint_system}, and Remark \ref{rem:equivalence}.
\begin{corollary}[Necessary condition]\label{cor:necessary_condition}
	Let the domain $\Omega^\ast$ be such that the state $u=u(\Omega^\ast)$ satisfies the overdetermined boundary-value problem \eqref{eq:Bernoulli_problem}, i.e., there holds
	\begin{equation}
	\label{eq:imaginary_is_zero}
		\ui = 0 \quad \text{on $\Omega^\ast$}. 
	\end{equation}
	Then, the domain $\Omega^\ast$ is stationary for the shape problem $\frac12 \intO{|\ui|^2} \to \inf$, where $\ui$ is subject to \eqref{eq:imaginary_state}.
	That is, it fulfills the necessary optimality condition
	\begin{equation}
	\label{eq:optimality_condition}
		{\operatorname{\mathnormal{d}}}J(\Omega^\ast)[\VV] = 0, \quad \text{for all $\VV \in \sfTheta^{2}$}.
	\end{equation}
\end{corollary}
\begin{proof}
	By the assumption that $\ui = 0$ on $\Omega^\ast$, one finds that $p=p(\Omega^\ast)$ on $\Omega^\ast$.
	Thus, it follows that $\GG \equiv 0$ on $\Sigma^\ast$ which implies that ${\operatorname{\mathnormal{d}}}J(\Omega^\ast)[\VV] = 0$, for any $\VV \in \sfTheta^1$.
\end{proof}
\begin{remark}
	In connection with the previous result, we remark that solutions of the necessary condition \eqref{eq:optimality_condition} might exist such that the state does not satisfy equation \eqref{eq:imaginary_is_zero}.
	However, only in the case of exact matching of boundary data a stationary domain $\Omega^\ast$ is a global minimum because $J(\Omega^\ast)=0$.
\end{remark}
\begin{remark}\label{rem:detour_remark}
To obtain the form ${\operatorname{\mathnormal{d}}}J(\Omega)[\VV] = \intS{\GG \nu \cdot \VV}$ -- in accordance with Hadamard--Zol\'{e}sio's structure theorem -- via the chain rule approach, the authors in \cite{EpplerHarbrecht2009,EpplerHarbrecht2010,EpplerHarbrecht2012a,RabagoAzegami2019b,RabagoAzegami2020} utilize the strong form of the equation as well as the boundary conditions satisfied by the shape derivative of $u$.
The existence of the said derivative, however, requires more regularity on $u$, in addition to assuming that the domain is at least $\mathcal{C}^{2,\alpha}$ ($\alpha \in (0,1]$).
Allowing this higher regularity assumptions and taking the shape differentiability of the state $u$ for granted, an alternative and more direct argument to derive the shape derivative of $J$ could be given, see Appendix \ref{subsec:shape_derivatives_of_the_cost}.
The derivation of the shape gradient of the cost issued above is based on the Lagrangian derivative $\dotu{}$ which is rigorously accounted for in the proof of Theorem \ref{prop:umaps}.
We emphasize that only the weak form of the Lagrangian derivatives is utilized. 
This bypasses the need to use the strong form of the shape derivative of $u$.
Alternatively, Theorem \ref{prop:Jmaps} could be shown via the strong form of the equation, and of the Eulerian derivatives of the state, as well as the boundary conditions satisfied by the shape derivative of ${u}$ coupled with Hadamard's domain differentiation formula (see, e.g., \cite[Thm. 4.2, p. 483]{DelfourZolesio2011}), \cite[eq. (5.12), Thm. 5.2.2, p. 194]{HenrotPierre2018} or \cite[eq. (2.168), p. 113]{SokolowskiZolesio1992}):
\begin{align}
	\left. \left\{ \frac{\operatorname{\mathnormal{d}}}{\operatorname{\mathnormal{d}}t} \intOt{f(t,x)} \right\} \right|_{t=0}
		&= \intO{\frac{\partial }{\partial t} f(0,x)}
			+ \intdO{ f(0,\sigma) \Vn}. \label{eq:Hadamard_domain_formula}
\end{align}
\end{remark}
\section{Instability analysis of the critical shape}
\label{sec:instability_analysis_of_the_critical_shape}
In this section, we investigate the question of stability of the proposed shape optimization formulation of \eqref{eq:Bernoulli_problem} at a critical shape $\Omega^\ast$.
To do this, we need the expression for the shape Hessian of $J$ at $\Omega^\ast$ which we exhibit in subsection \ref{subsec:shape_Hessian}.
Afterwards, we deal with the stability condition of the present shape optimization problem in subsection \ref{subsec:instability_analysis}.
A quick review of the sufficient second-order optimality condition for the present shape optimization problem is given first in the following subsection.
\subsection{Sufficient Conditions}
Before we proceed to the computation of the shape Hessian, we briefly provide here a rough discussion about the regular local optimality of second-order of the critical shape $\Omega^{\ast}$ (the optimal solution of the proposed shape problem) on the basis of \cite[Thm. 3.3]{EpplerHarbrechtSchneider2007}.
Hereinafter, whenever we state the phrase `$\mathcal{C}^{\circ}$ regular/smooth' we mean that the object it describes is either $\mathcal{C}^{0,1} \cap \mathcal{C}^{2,\alpha}$ regular, $\alpha > 0$, or simply $\mathcal{C}^{2,1}$ regular.
Similar definition is given for $\mathcal{C}^{\circ}(\Omega)$ and $\mathcal{C}^{\circ}(\Sigma)$.
Moreover, $\vect{\mathcal{C}}^{\circ}(\cdot):=[\mathcal{C}^{\circ}(\cdot)]^{2}$ (e.g., $\vect{\mathcal{C}}^{2,\alpha}(\Omega):=[ \mathcal{C}^{2,\alpha}(\Omega) ]^{2}$).
In the sequel, we say that $\Omega \in \mathcal{C}^{\circ}$ is an \textit{admissible} perturbation of $\widehat{\Omega} \in \mathcal{C}^{\circ}$ if for some fixed small number $\delta \in (0,1)$ the following inequality condition holds:
\[
	\|\phi - \hat{\phi}\|_{\mathcal{C}^{\circ}(\mathbb{R};\mathbb{R}^{2})} < \delta,
\]
where $\phi, \hat{\phi} \in \mathcal{C}^{\circ}(\mathbb{R};\mathbb{R}^{2})$ are respectively a periodic parametrization of the free boundary of $\Omega$ and $\widehat{\Omega}$.
We denote the collection of such perturbations of $\widehat{\Omega}$ by $U_{\delta}({\hat{\phi}})$.
Note that, essentially, $U_{\delta}({\hat{\phi}})$ is the set of all perturbations $\Omega$ that are `sufficiently near' -- based on some appropriate metric -- to $\widehat{\Omega}$.
In fact, in the next proposition, the result only holds for all perturbations $\Omega$ that lies in some specific vicinity of the stationary shape $\Omega^{\ast}$ (i.e., domains whose free boundaries are contained within some certain tubular neighborhood of $\Sigma^{\ast}$).
We let $U:=U_{1}(\cdot)$ be another open bounded set containing all sets $U_{\delta}(\cdot)$, $\delta \in (0,1)$.
\begin{proposition}[sufficient second-order optimality condition]
	Let the necessary condition
	\begin{equation}
	\label{eq:assumption1}\tag{A1}
		{\operatorname{\mathnormal{d}}}J(\Omega^\ast)[\VV] = 0, \quad \text{for all $\VV \in \sfTheta^{2} \cap \vect{\mathcal{C}}^{\circ}(\overline{U})$},
	\end{equation}
	holds for a certain critical shape $\Omega^{\ast} \in \mathcal{C}^{\circ}$. 
	For all admissible perturbation $\Omega = \Omega(\phi) \in U_{\delta}(\phi^{\ast})$ of $\Omega^{\ast}$, we suppose that there is a constant $c_{b} > 0$ which depends continuously on $\Omega$ via the parametric function $\phi$ such that the bilinear form imposed by the shape Hessian satisfies the following inequality 
	\begin{equation}
	\label{eq:assumption2}\tag{A2}
			|{\operatorname{\mathnormal{d}}}^2J(\Omega)[\VV,\WW]| \leqslant c_{b}(\Omega) \|\VV\|_{\vect{H}^{1}(\Sigma)} \|\WW\|_{\vect{H}^{1}(\Sigma)}, 
	\end{equation}
	if $\Omega \in \overline{U_{\delta}(\phi^{\ast})}$,
	and the remainder estimate
	\begin{equation}
	\label{eq:assumption3}\tag{A3}
			|{\operatorname{\mathnormal{d}}}^2J(\Omega)[\VV,\WW] - {\operatorname{\mathnormal{d}}}^2J(\Omega^{\ast})[\VV,\WW]|
			\leqslant \eta\left(\|\phi - \phi^{\ast}\|_{\mathcal{C}^{\circ}(\mathbb{R};\mathbb{R}^{2})}\right) \|\VV\|_{\vect{H}^{1}(\Sigma)} \|\WW\|_{\vect{H}^{1}(\Sigma)}, 
	\end{equation}	
	for all $\VV, \WW \in \sfTheta^{2} \cap \vect{\mathcal{C}}^{\circ}(\overline{U})$,
	where $\eta:\mathbb{R}^{+}_{0} \to \mathbb{R}^{+}_{0}$ is a decreasing function that satisfies the condition $\eta(s) \to 0$ as $s \to 0$.
	Then, the domain $\Omega^{\ast}$ is a strong regular local optimum of second-order with respect to a specific constant $\hat{c}_{e} > 0$,
	\begin{equation}\label{eq:strong_regular_local_optimum_of_second_order}
		J(\Omega) - J(\Omega^{\ast}) \leqslant \hat{c}_{e} \|\phi - \phi^{\ast}\|_{\vect{H}^{1}(\Sigma^{\ast})},
		\quad \text{for all $\Omega \in \overline{U_{\hat{\delta}}(\phi^{\ast})}$}.
	\end{equation}
	if and only if the shape Hessian satisfies the strong coercivity estimate
	\begin{equation}
	\label{eq:assumption4}\tag{A4}
			{\operatorname{\mathnormal{d}}}^2J(\Omega^\ast)[\VV,\VV] \geqslant c_{e} \|\VV\|^2_{\vect{H}^{1}(\Sigma^\ast)}, 
			\quad \text{for all $\VV \in \sfTheta^{2} \cap \vect{\mathcal{C}}^{\circ}(\overline{U})$},
	\end{equation}	
	for some constant $c_{e} > 0$.
\end{proposition}

As remarked in many references (see, e.g., \cite{EpplerHarbrecht2006nm,EpplerHarbrecht2006,EpplerHarbrechtSchneider2007,EpplerHarbrecht2008,EpplerHarbrecht2009,EpplerHarbrecht2010,EpplerHarbrecht2012a}), it is generally impossible to realize coercivity with respect, for instance, to the space ${\mathcal{C}^{2,\alpha}}$ where the involved objects are defined and differentiation is undertaken (hence Assumption \eqref{eq:assumption4}).
In subsection \ref{subsec:instability_analysis}, we will show that the energy space of the bilinear form imposed by the shape Hessian ${\operatorname{\mathnormal{d}}}^2J(\Omega)$ is the Sobolev space $\vect{H}^{1}(\Sigma)$, and that coercivity of the shape Hessian $	{\operatorname{\mathnormal{d}}}^2J(\Omega^\ast)$ can only be achieved in the \textit{weaker} space $\vect{H}^{1/2}(\Sigma^{\ast})$.
The latter result implies that the shape problem under consideration is \textit{mildly} ill-posed.


\subsection{Shape Hessian of the cost function at the critical shape}
\label{subsec:shape_Hessian}
Our goal here is to get the structure of the shape Hessian at a critical shape (see Proposition \ref{prop:shape_Hessian}).
To do this we shall assume that $\Omega$ is of class $\mathcal{C}^{2,1}$ and take the deformation fields from the set $\sfTheta^2$.
Given this assumption, the existence of the shape derivative of the state is guaranteed.
This implies that we only need to apply the chain rule to get the desired expression.
By this approach, however, we first need to exhibit the shape derivative of the states which we give in the next lemma.
A revision of Proposition \ref{prop:shape_Hessian} with the mild $\mathcal{C}^{1,1}$ regularity assumption on the domain, however, is given in Proposition \ref{prop:shape_Hessian_weaker_assumption}.

\begin{lemma}
	\label{lem:tiihonen}
	Let $\Omega \in \mathcal{C}^{2,1}$ and $\VV \in \in \sfTheta^2$.
	Then, $u \in \HH^3(\Omega)$ is shape differentiable with respect to $\Omega$ in the direction of $\VV$, 
	and its shape derivative $u' \in \HH^{1}(\Omega)$ uniquely satisfies the boundary value problem
	\begin{equation}
	\label{eq:shape_derivative_of_the state}
		-\Delta {u}' 	=0 \ \text{in $\Omega$},\quad
		{u}' 			=0 \ \text{on $\Gamma$},\quad
		\dn{{u}'} + i {u}'	=\Upsilon(u)[\Vn] \ \text{on $\Sigma$},
	\end{equation}
	where $\Upsilon(u)[\Vn] = {\operatorname{div}}_{\Sigma} (\Vn\nabla_{\Sigma} {u}) 
						- i (\dn{{u}}+\kappa {u})\Vn 
						+ \lambda \kappa \Vn$.
\end{lemma}
Because the above results is new and is in fact not yet available in the literature, we provide the proof of the lemma in Appendix \ref{subsec:shape_derivative_of_the_state}. 
Now with the lemma at hand, the first of the two main results of this subsection is now in order.
\begin{proposition}
	\label{prop:shape_Hessian}
	Let $\Omega \in \mathcal{C}^{2,1}$ and $\VV, \WW \in \sfTheta^2$. Then, the shape Hessian of $J$ at the solution $\Omega^\ast$ of the Bernoulli problem \eqref{eq:Bernoulli_problem} has the following structure:
	\begin{equation}\label{eq:shape_Hessian_at_stationary_domain}
		{\operatorname{\mathnormal{d}}}^2J(\Omega^\ast)[\VV,\WW] = - \intSast{ \lambda (p_{1,W}' + \kappa p_{2,W}') {\Vn}  },
	\end{equation}
where $p'_{W} = p_{1,W}' + i p_{2,W}'$ satisfies the complex PDE system 	
\begin{equation}
	\label{eq:shape_derivative_of_the_adjoint state_at_critical_shape}
		-\Delta {p}'_{W} 	=u'_{2,W} \ \text{in $\Omega^\ast$},\qquad
		{p}'_{W} 			=0 \ \text{on $\Gamma$},\qquad
		\dn{{p}'_{W}} - i {p}'_{W}	=0 \ \text{on $\Sigma^\ast$}.
\end{equation}
\end{proposition}
\begin{proof}
The higher regularity assumption on $\Omega$ and on the deformation fields allows us to apply formula \eqref{eq:Hadamard_domain_formula} twice.
Indeed, we may write ${\operatorname{\mathnormal{d}}}J(\Omega)[\VV]  =  \intO{\ui \uip } + \intS{ \operatorname{div}\left( \frac12  |\ui|^2 \VV \right)}$ by Stokes' theorem.
Then, by the same approach, we get 
\begin{align*}
{\operatorname{\mathnormal{d}}}^2J(\Omega)[\VV,\WW] &= \intO{\ui u''_{2,V,W} } + \intO{ u'_{2,V} u'_{2,W } } \\
	&\qquad + \intO{ \operatorname{div}\left[ \ui u'_{2,W} \VV + \operatorname{div}\left(\frac12 \ui^2\VV \right)\WW \right] }.
\end{align*}
\sloppy Clearly, at $\Omega = \Omega^\ast$, the above expression for the shape Hessian reduces to ${\operatorname{\mathnormal{d}}}^2J(\Omega^\ast)[\VV,\WW] = \intOast{ u'_{2,V} u'_{2,W } }$.
To obtain \eqref{eq:shape_Hessian_at_stationary_domain} from this equation, we utilize the adjoint problem \eqref{eq:shape_derivative_of_the_adjoint state_at_critical_shape}.
Note that the shape derivative of the state and of the adjoint variable at the stationary solution $\Omega^\ast$ of \eqref{eq:Bernoulli_problem} are given by	\begin{equation}
	\label{eq:shape_derivative_of_the state_at_optimal_solution}
		-\Delta {u}'_{V} 			=0 \ \text{in $\Omega^\ast$},\quad
		{u}'_{V} 				=0 \ \text{on $\Gamma$},\quad
		\dn{{u}'_{V}} + i {u}'_{V}	= \lambda (\kappa - i) V_{n}\ \text{on $\Sigma^\ast$},
	\end{equation}
and by system \eqref{eq:shape_derivative_of_the_adjoint state_at_critical_shape}, respectively.
Multiplying each of these systems by $\overline{p}'_{W} \in \HHg(\Omega)$ and ${u}'_{V} \in \HHg(\Omega)$, respectively, and then applying integration by parts will eventually lead us to the identity $\intOast{u'_{2,W} {u}'_{V}} = \intS{ \lambda \left( \kappa-i \right) \overline{p}'_{W} V_{n} }$.
Comparing the real and imaginary parts on both sides of this equation in the end will give us \eqref{eq:shape_Hessian_at_stationary_domain}.
\end{proof}
\begin{remark}
	Observing from problem \eqref{eq:shape_derivative_of_the state_at_optimal_solution}, we see that $u'_{2,V} \neq 0$ provided $\VV \not\equiv \vect{0}$ on $\Sigma^\ast$.
	Therefore, from the proof of Proposition \ref{prop:shape_Hessian}, we have ${\operatorname{\mathnormal{d}}}^2J(\Omega^\ast)[\VV,\VV] = \intOast{ |u'_{2,V}|^2 } = \|u'_{2,V}\|_{L^2(\Omega^\ast)}^2 > 0$ for any $\VV \not\equiv \vect{0}$ on $\Sigma^\ast$. 
	This inequality, however, does not mean that the shape optimization problem is already well-posed. 
	The matter will be discussed further in the next subsection.
\end{remark}
%
%
%
%
	We remark that the second-order shape derivative ${\operatorname{\mathnormal{d}}}^2J(\Omega)[\VV,\WW]$ can actually be expressed using the material derivatives of $u$ and $p$ instead of their shape derivatives.
	This allows for a more general formula since one does not need additional regularity for $u$ and $p$.
	Howbeit, the order two analysis was only carried out in this work in order to examine the ill-posedness of the proposed shape optimization problem (see also Remark \ref{rem:Newton_method}).
	Hence, we only need the expression for the shape Hessian at a critical shape ${\operatorname{\mathnormal{d}}}^2J(\Omega^\ast)[\VV,\WW]$, and since we already assume that $\Omega$ is of class $\mathcal{C}^{2,1}$, the second-order shape derivative ${\operatorname{\mathnormal{d}}}^2J(\Omega^\ast)[\VV,\WW]$ can be obtained with less effort and computations as the shape derivative for $u$ and $p$ are already available, and that we do not need to find a simplified form of $\intO{\ui u''_{2,V,W} }$ in terms of $\ui$ because the function already vanishes at $\Omega = \Omega^\ast$.
	Still, we stress that it is possible to obtain the second-order shape derivative ${\operatorname{\mathnormal{d}}}^2J(\Omega^\ast)[\VV,\WW]$ without using $u''_{2,V,W}$.
	This can be done by writing ${\operatorname{\mathnormal{d}}}^2J(\Omega^\ast)[\VV,\WW]$ using the material derivative $\dot{u}_2$ which then also allow one to weaken the regularity assumption on the domain.
	In this case, we only need $\Omega$ be of class $\mathcal{C}^{1,1}$.
	\sloppy Indeed, from \eqref{eq:material_derivative_of_the_cost}, we have 
	${\operatorname{\mathnormal{d}}}J(\Omega_s)[\VV]
		= \frac12 \intO{ ({\operatorname{div}}\, \VV) I_s |\ui^s|^2} + \intO{ \ui^s \dotu{}_2^s }$.
	Note here that the deformation field $\VV$ is independent of $s>0$.
	Therefore, differentiating the integral expression with respect to $s$, and then evaluating at $s=0$, we get
	\begin{align*}
		{\operatorname{\mathnormal{d}^2}}J(\Omega)[\VV,\WW]
		&= \frac12 \intO{ \left[ ({\operatorname{div}}\, \VV)^2 |\ui|^2 + 2 ({\operatorname{div}}\, \VV) \ui \dot{u}_{2,W} \right]}\\
		&\qquad	+ \intO{ \left( \dot{u}_{2,V} \dot{u}_{2,W} + \ui \ddot{u}_{2,V,W} \right)},
	\end{align*}
	for any given vector fields $\VV, \WW \in \sfTheta^1$.
	At $\overline{\Omega} = \overline{\Omega}^\ast$, $\ui \equiv 0$, and so we obtain ${\operatorname{\mathnormal{d}^2}}J(\Omega^\ast)[\VV,\WW] = \intO{ \dot{u}_{2,V} \dot{u}_{2,W} }$.
	This integral can of course be expressed in terms of the material derivative of $p$ using the same technique used in the proof of Theorem \ref{prop:Jmaps}. 
	In fact, however, we can just simply introduce an adjoint problem in order to get rid of the first-order derivative terms.
	To this end, let us consider the PDE system 
	\begin{equation}
	\label{eq:material_derivative_of_the_adjoint state_at_critical_shape}
		-\Delta q_{W} 	=\dot{u}_{2,W} \ \text{in $\Omega$},\qquad
		q_{W} 		=0 \ \text{on $\Gamma$},\qquad
		\dn{}q_{W} - i q_{W}	=0 \ \text{on $\Sigma$},
	\end{equation}
	at $\overline{\Omega} = \overline{\Omega}^\ast$ which can be shown to be well-posed, and that $q_{W} \in \HHg(\Omega)$ because of Proposition \ref{prop:umaps}.
	Here, of course, $q_{W} = q_{1,W} + i q_{2,W}$, where $q_{1,W}$ and $q_{2,W}$ denote the real and imaginary parts of $q_{W}$, respectively.
	Multiplying the above equation by $\dot{u}_{V}$, and then applying integration by parts, yields
	$\intOast{ \nabla \bar{q}_{W} \cdot \nabla \dot{u}_{V}} + i \intSast{ \bar{q}_{W} \dot{u}_{V} } = \intOast{ \dot{u}_{2,W} \dot{u}_{V} }$.
	Let us also consider equation \eqref{eq:Lagrangian_derivative_of_the_state} with ${v} = q_{W} \in \HHg(\Omega)$, $\overline{\Omega} = \overline{\Omega}^\ast$, and $\dot{u}$ replaced by $\dot{u}_{V}$.
	Then,
	$\intOast{\nabla \dot{u}_{V} \cdot \nabla \bar{q}_{W}}  + i \intSast{\dot{u}_{V} \bar{q}_{W} } 
		= -\intOast{A \nabla {\ur} \cdot \nabla \bar{q}_{W}} 
				+ \lambda \intSast{ ({\operatorname{div}}_{\Sigma} \VV) \bar{q}_{W}}$.
	Comparing this equation with the previous one, we get the equation
	$\intOast{ \dot{u}_{2,W} \dot{u}_{V} }
		= -\intOast{A \nabla {\ur} \cdot \nabla \bar{q}_{W}}  
			+ \lambda \intSast{ ({\operatorname{div}}_{\Sigma} \VV) \bar{q}_{W}}$.
	Meanwhile, it can be verified that a similar identity to \eqref{eq:domain_identity} holds for $\dot{u}_{2} \in \HHg(\Omega)$ and $\bar{q}_{W} \in \HHg(\Omega)$ at $\overline{\Omega} = \overline{\Omega}^\ast$ that is given by
	$- \intOast{A \nabla {\ur} \cdot \nabla \bar{q}_{W}}  
	= - \intSast{(\nabla \bar{q}_{W} \cdot \nabla {\ur})\Vn} 
		- i \intSast{ \bar{q}_{W} (\VV \cdot \nabla {\ur})}
		+ \lambda \intSast{ (\VV \cdot \nabla \bar{q}_{W})}$.
	Putting together the last two equations, and then appealing to formula \eqref{eq:tangential_Greens_formula}, we get
	\begin{align*}
		\intOast{ \dot{u}_{2,W} \dot{u}_{V} }
		%
		%
		%
		%
		%
		%
		&= - i \intSast{ \lambda \bar{q}_{W} \Vn } + \lambda \intSast{ \kappa \bar{q}_{W} \Vn }.
	\end{align*}	
	Finally, comparing the real and imaginary parts on both sides of the above equation, we arrive at the final expression for the shape Hessian of $J$ at a critical shape $\Omega^\ast$:
	\[
		{\operatorname{\mathnormal{d}}}^2J(\Omega^\ast)[\VV,\WW] = - \intSast{ \lambda (q_{1,W} + \kappa q_{2,W}) {\Vn}  }.
	\]
	The above integral expression is clearly identical to \eqref{eq:shape_Hessian_at_stationary_domain} in terms of structure.
	
	To close this subsection, we formally summarize the previous result into a proposition given below which is essentially a revision of Proposition \ref{prop:shape_Hessian} with weaker regularity assumptions on the domain $\Omega$ and on the deformation fields $\VV$ and $\WW$, and combined with a different adjoint problem.
	\begin{proposition}
	\label{prop:shape_Hessian_weaker_assumption}
	Let $\Omega \in \mathcal{C}^{1,1}$ and $\VV, \WW \in \sfTheta^1$. Then, the shape Hessian of $J$ at the solution $\Omega^\ast$ of the Bernoulli problem \eqref{eq:Bernoulli_problem} has the following structure:
	\begin{equation}\label{eq:shape_Hessian_at_a_critical_shape_version_2}
		{\operatorname{\mathnormal{d}}}^2J(\Omega^\ast)[\VV,\WW] = - \intSast{ \lambda (q_{1,W} + \kappa q_{2,W}) {\Vn}  },
	\end{equation}
	where $q_{W} = q_{1,W} + i q_{2,W}$ satisfies the complex PDE system 	
	\begin{equation}
		-\Delta {q}_{W} 	=\dot{u}_{2,W} \ \text{in $\Omega^\ast$},\qquad
		{q}_{W} 			=0 \ \text{on $\Gamma$},\qquad
		\dn{{q}_{W}} - i {q}_{W}	=0 \ \text{on $\Sigma^\ast$}.
\end{equation}	
	\end{proposition}
%
%
%
%
\begin{remark}
	As was shown in the computation of expression \eqref{eq:shape_Hessian_at_a_critical_shape_version_2} issued above, the shape derivative of the state is not needed to exhibit the structure of the shape Hessian at a critical shape.
	It therefore goes without saying that the expression for the shape Hessian of the cost function, for general domains of class $\mathcal{C}^{1,1}$, can also be obtained without using the second-order shape derivative of the state by using the same technique applied to show \eqref{eq:shape_Hessian_at_a_critical_shape_version_2}.
	Here, we omit the computation of the said expression since we are only interested in the structure of ${\operatorname{\mathnormal{d}}}^2J(\Omega^\ast)$ (see Remark \ref{rem:Newton_method}).
\end{remark}

\subsection{Compactness of the Hessian at the optimal domain}
\label{subsec:instability_analysis}
Having computed the form of the shape Hessian at a stationary domain $\Omega^\ast$, we now examine the stability or instability of the presently proposed shape optimization formulation. 
Before that, we first make a few remarks regarding the structure of the shape Hessian.
%
%

It can be verified that the exact form of the shape Hessian ${\operatorname{\mathnormal{d}}}^2J(\Omega)[\VV,\WW]$ depends on the shape derivative $\kappa'$ of the mean curvature $\kappa$.
	In fact, with the first-order shape derivative of $J$ given as ${\operatorname{\mathnormal{d}}}J(\Omega)[\VV] = \intS{G\nn\cdot\VV}$, where $G$ is given by \eqref{eq:shape_gradient}, the shape Hessian can be shown to have the structure
	\begin{align*}
	{\operatorname{\mathnormal{d}}}^2J(\Omega)[\VV,\WW]
	= \intS{ \left[ {G}'_{W} \Vn + \left( \dn{G} + \kappa {G}\right) \Vn \Wn
			- {G}K + {G}(\operatorname{\mathnormal{D}}\!\VV) \Wn\right]},
	\end{align*}
	where $K = \svvs \cdot (\operatorname{\mathnormal{D}}\!_{\Sigma} \nn) \wws + \nn \cdot (\operatorname{\mathnormal{D}}\!_{\Sigma} \svv) \wws + \nn \cdot (\operatorname{\mathnormal{D}}\!_{\Sigma} \ww) \svvs$, $\svv = \VV|_{\Sigma}$, $\svv = \svv_{\Sigma} + \vn \nn:= (\svv \cdot \tau)\tau + (\svv \cdot \nn)\nn$ and $\operatorname{\mathnormal{D}}\!_{\Sigma}$ denotes the tangential differential operator called the \emph{tangential Jacobian matrix} given as $\operatorname{\mathnormal{D}}\!_{\Sigma} \svv = \operatorname{\mathnormal{D}}\!\VV|_{\Sigma} - (\operatorname{\mathnormal{D}}\!\VV\nn)\nn^\top$ (see, e.g., \cite[Eq. (5.2), p. 495]{DelfourZolesio2011}).
	Obviously, from \eqref{eq:shape_gradient}, the term ${G}'$ appearing in the shape Hessian is composed of the derivative $\nn'$ of the normal vector and and $\kappa'$ of the mean-curvature.
	The form of $\nn'_{W}$ and $\kappa'_{W}$ obtained along the deformation field $\WW \in \sfTheta^2$ are respectively given by $(\operatorname{\mathnormal{D}}\!\WW \nn \cdot \nn)\nn - (\operatorname{\mathnormal{D}}\!\WW)^\top \nn - (\operatorname{\mathnormal{D}}\!\nn)\WW$ and $\text{trace}\left\{ \operatorname{\mathnormal{D}}\!\left[ (\operatorname{\mathnormal{D}}\!\WW \nn \cdot \nn)\nn - (\operatorname{\mathnormal{D}}\!\WW)^\top\nn \right] - \operatorname{\mathnormal{D}}\!\nn \operatorname{\mathnormal{D}}\!\WW\right\} - \nabla \kappa \cdot \WW$ (see \cite{DelfourZolesio2011,SokolowskiZolesio1992}).
	Clearly, the latter expression consists of a second-order tangential derivative of the velocity field $\WW$, and this derivative actually exists due to our assumption that $\Omega$ is of class $\mathcal{C}^{2,1}$ \cite{DelfourZolesio2011,SokolowskiZolesio1992}.
	\sloppy Hence, we can actually decompose the shape Hessian into three parts: $h_1(\Omega)[\VV,\WW]$, $h_2(\Omega)[\VV,\WW]$, and $h_3(\Omega)[\VV,\WW]$, where $h_2$ is composed of the term $\nn'_{W}$ while $h_3$ consists of the expression $\kappa'_{W}$, and such that we have the estimates $|h_1(\Omega)[\VV,\WW]| \lesssim \|\VV\|_{\vect{L}^2(\Sigma)} \|\WW\|_{\vect{L}^2(\Sigma)}$, $|h_2(\Omega)[\VV,\WW]| \lesssim \|\VV\|_{\vect{H}^{1/2}(\Sigma)} \|\WW\|_{\vect{H}^{1/2}(\Sigma)}$, and $|h_3(\Omega)[\VV,\WW]| \lesssim \|\VV\|_{\vect{H}^{1}(\Sigma)} \|\WW\|_{\vect{H}^{1}(\Sigma)}$ (cf. \cite[proof of Thm. 3]{EpplerHarbrecht2012a}).
	From these, we can infer that the shape Hessian defines a continuous bilinear form ${\operatorname{\mathnormal{d}}}^2J(\Omega) : \vect{H}^{1}(\Sigma) \times \vect{H}^{1}(\Sigma) \to \mathbb{R}$;
	that is, $|{\operatorname{\mathnormal{d}}}^2J(\Omega)[\VV,\WW]| \lesssim \|\VV\|_{\vect{H}^{1}(\Sigma)} \|\WW\|_{\vect{H}^{1}(\Sigma)}$.\footnote{The notation $\vect{H}^{1}(\cdot)$ stands for the Sobolev space $\vect{H}^{1}(\cdot):=\{\vect{v}:=(v_1,v_2) : v_1, v_2 \in H^{1}(\cdot)\}$ and is equipped with the norm $\|\vect{v}\|^2_{\vect{H}^{1}(\cdot)} = \|v_1\|^2_{H^{1}(\cdot)} + \|v_2\|^2_{H^{1}(\cdot)}$. Similar definition is also given to the $\vect{H}^{1}_{\Gamma, 0}(\cdot)$-space.}
In addition, one can also obtained a remainder estimate given by $|{\operatorname{\mathnormal{d}}}^2J(\Omega)[\VV,\WW] - {\operatorname{\mathnormal{d}}}^2J(\Omega^{\ast})[\VV,\WW]|
			\leqslant \eta\left(\mathcal{M}(\Omega,\Omega^{\ast})\right) \|\VV\|_{\vect{H}^{1}(\Sigma)} \|\WW\|_{\vect{H}^{1}(\Sigma)}$,
where $\mathcal{M}(B_{1},B_{2})$ is some appropriate metric measuring the distance between the two sets $B_{1}$ and $B_{2}$ in $\mathbb{R}^{d}$ while $\eta:\mathbb{R}^{+}_{0} \to \mathbb{R}^{+}_{0}$ is a decreasing function that satisfies $\eta(s) \to 0$ as $s \to 0$ (cf. Assumption \ref{eq:assumption3}).	
%
%

In connection with the above discussion, it is natural to ask whether we also have the estimate $ {\operatorname{\mathnormal{d}}}^2J(\Omega^\ast)[\VV,\VV] \gtrsim \|\VV\|^2_{\vect{H}^{1}(\Sigma^\ast)}$.
This inequality condition actually has something to do with the \textit{stability} of a local minimizer $\Omega^\ast$ of $J$.
A result regarding \textit{sufficient} second-order conditions from \cite{Dambrine2002,DambrinePierre2000,Eppler2000b}\footnote{Independently of \cite{Dambrine2002,DambrinePierre2000}, also Eppler derived second-order sufficient optimality conditions in \cite{Eppler2000b}; see particularly Section 4 of the said paper.} in fact states that a local minimizer $\Omega^\ast$ is stable \textit{if and only if} the shape Hessian ${\operatorname{\mathnormal{d}}}^2J(\Omega^\ast)$ is strictly coercive in its corresponding energy space which is the $\vect{H}^{1}(\Sigma^\ast)$ space in the present case. 
Such strict coercivity, however, cannot be established for the shape Hessian \eqref{eq:shape_Hessian_at_stationary_domain}.
The aforesaid \textit{lack of coercivity} is known, especially in the shape optimization literature, as the \textit{two-norm discrepancy}, see \cite{Dambrine2002,Eppler2000a,EpplerHarbrechtSchneider2007} for more details.
Meanwhile, for a more recent study concerning the question of stability in the field of shape optimization -- focusing especially on the strategy using second-order shape derivatives -- we refer the readers to \cite{DambrineLamboley2019}.

To analyze the shape Hessian \eqref{eq:shape_Hessian_at_stationary_domain}, we will write it into an equivalent expression (see \cite{EpplerHarbrecht2006nm,EpplerHarbrecht2006,EpplerHarbrecht2009,EpplerHarbrecht2010,EpplerHarbrecht2012a}) and adapt the method already used in \cite{Afraites2022,AfraitesMasnaouiNachaoui2022}.
We first introduce the operators $ \mathcal{L}: \ \vect{H}^{1/2}(\Sigma^\ast) \to H^{1/2}(\Sigma^\ast)$ and $\mathcal{K}: \ H^{1/2}(\Sigma^\ast) \to H^{1/2}(\Sigma^\ast)$ (see \cite[Sec. 3.4]{RabagoAzegami2019b}), defined respectively as 
$\mathcal{L}\VV := \lambda \Vn$ and $\mathcal{K}v := \kappa v$,
which we shall utilize in our argumentation.
The continuity of the operators $\mathcal{L}$ and $\mathcal{K}$ follow from the next lemma (cf. \cite[Lem. 3.3]{EpplerHarbrecht2006nm} and see also \cite{EpplerHarbrecht2009, EpplerHarbrecht2010, EpplerHarbrecht2012a}).
\begin{lemma}
\label{lem:aux_result}
	Let $\Omega \subset \mathbb{R}^2$ be a bounded Lipschitz domain with boundary $\Gamma := \partial\Omega$.
	Then, the map $v \mapsto \phi v$ is continuous in $H^{1/2}(\Gamma)$ for any $v \in H^{1/2}(\Gamma)$ and $\phi \in \mathcal{C}^{0,1}(\Gamma)$.
\end{lemma}
\begin{proof}
	By McShane-Whitney extension theorem, there is some function $\tilde{\phi} \in \mathcal{C}^{0,1}(\overline{\Omega})$ such that $\tilde{\phi}|_{\Gamma} = \phi$.
	Also, from \cite[Thm. 3.37, p. 102]{McLean2000}, there is a bounded linear extension operator $\tilde{E}:H^{1/2}(\Gamma) \to H^{1}(\Omega)$.
	Hence, in view of \cite[Thm. 1.4.1.1, p. 21]{Grisvard1985} and by trace theorem, we see that the operator $\mathcal{K} v = \phi v$ is continuous in $H^{1/2}(\Gamma)$ since $\mathcal{K} v = {\operatorname{trace}}(\tilde{\phi} \tilde{E}v)$ is a composition of bounded operators.
\end{proof}
\begin{remark}
\label{rem:positivity}
Let us note that for $\mathcal{C}^{1,1}$ smooth boundary $\Sigma$, $\kappa$ is well defined almost everywhere on the boundary and actually belongs to $L^{\infty}$ on account of Rademacher's theorem (see, e.g., \cite[Thm. 2.7.1, p. 67]{SauterSchwab2011}).
Moreover, the mean curvature is continuous for boundaries of class $\mathcal{C}^{2,1}$.
Due to our smoothness assumption on the free boundary, the operator $\mathcal{K}$ is actually a continuous map from $H^s(\Sigma^\ast)$ to $H^s(\Sigma^\ast)$ for all $s \in [0,1]$, and the same is also true for the (bijective) map $\mathcal{L}$. 
\end{remark}
In addition to $\mathcal{L}$ and $\mathcal{K}$, we also introduce the operator $\mathcal{P}_i$, $i \in \{1,2\}$, by $\mathcal{P}_i: \vect{H}^{1/2}(\Sigma^\ast) \to H^{-1/2}(\Sigma^\ast)$, $\VV \mapsto p_{i,V}'$, for all $i\in\{1,2\}$.
Accordingly, we may write the shape Hessian at the optimal domain $\Omega^\ast$ as
\begin{align*}
	{\operatorname{\mathnormal{d}}}^2J(\Omega^\ast)[\VV,\VV]
		&= -\intSast{ (\lambda {\Vn}  p_{1,V}' + \kappa \lambda {\Vn}  p_{2,V}')  }\\
		&= - \langle \mathcal{L}\VV, \mathcal{P}_1\VV\rangle - \langle \mathcal{K} \mathcal{L}\VV, \mathcal{P}_2\VV\rangle,
\end{align*}
where $\langle \, \cdot\, , \, \cdot\, \rangle$ is the duality product between $H^{1/2}(\Sigma^\ast)$ and $H^{-1/2}(\Sigma^\ast)$.
%
%
%

Now, in relation to Remark \ref{rem:positivity}, we state our final result which claims the ill-posedness of the proposed CCBM formulation of \eqref{eq:Bernoulli_problem}.
\begin{proposition}\label{prop:compactness_of_shape_Hessian}
Let $\Omega^\ast$ be the stationary solution to \eqref{eq:Bernoulli_problem}, then the Riesz operator associated to ${\operatorname{\mathnormal{d}}}^2J(\Omega^\ast)[\VV,\VV]: \vect{H}^{1/2}(\Sigma^\ast) \to H^{-1/2}(\Sigma^\ast)$ is compact.
\end{proposition}
\begin{proof}
The idea of the proof is to express the shape Hessian as a composition of linear continuous operators and a compact one (the compactness being obtained using the compactness of the imbedding between two Sobolev spaces).
As shown above, the operators $\mathcal{L}$ and $\mathcal{K}$ are continuous, but the operators $\mathcal{P}_1$ and $\mathcal{P}_2$ are compact.
We verify this claim by decomposing these operators as composition of continuous and compact ones.
For the product $-\langle \mathcal{L}\VV, \mathcal{P}_1\VV\rangle$, we decompose the map $\mathcal{P}_1$ by first expressing it as the composition $\mathcal{P}_1 = \mathcal{Q}_2 \circ \mathcal{Q}_1$
where $\mathcal{Q}_{1}: \vect{H}^{1/2}(\Sigma^\ast) \to H^{1}(\Omega^\ast)$, $\VV \mapsto \uip$, and $\mathcal{Q}_{2}: H^{1}(\Omega^\ast) \to H^{-1/2}(\Sigma^\ast)$, $\phi \mapsto {w}$.
Here, $\uip$ solves \eqref{eq:shape_derivative_of_the state} and ${w} = {w}_1 + i {w}_2$ satisfies
\begin{equation}\label{eq:Phi_equation}
		-\Delta {w}		=\phi \ \text{in $\Omega^\ast$},\qquad
		{w} 			=0 \ \text{on $\Gamma$},\qquad
		\dn{{w}} - i {w}	=0 \ \text{on $\Sigma^\ast$}.
\end{equation}
Clearly, $\mathcal{Q}_1$ is continuous.
Next, we further write $\mathcal{Q}_{2}$ as  $\mathcal{Q}_{2} := \mathcal{R}_{3} \circ \mathcal{R}_{2} \circ \mathcal{R}_{1}$ where 
\begin{align*}
	\mathcal{R}_{1}: H^{1}(\Omega^\ast) \to H^3(\Omega^\ast), \quad \phi \mapsto {w}_1,&\qquad
	\mathcal{R}_{2}: H^3(\Omega^\ast) \to H^{5/2}(\Sigma^\ast), \quad {w}_1 \mapsto {w}_1,\\
	\qquad\text{and}\qquad\mathcal{R}_{3}: H^{5/2}(\Sigma^\ast) &\to H^{-1/2}(\Sigma^\ast), \quad {w}_1 \mapsto {w}_1.
\end{align*}
The operators $\mathcal{R}_{1}$ and $\mathcal{R}_{2}$ are continuous while $\mathcal{R}_{3}$ is the compact embedding of $H^{5/2}(\Sigma^\ast)$ into $H^{-1/2}(\Sigma^\ast)$.\footnote{An embedding result for fractional Sobolev spaces $H^s(\Omega)$ under bounded $\mathcal{C}^{k,\alpha}$-domains, $k = 0, 1, \ldots$, $\alpha \in [0,1]$, can be found in \cite[Thm. 7.9, p. 119]{Wloka1987}. See also \cite[Thm. 2.5.5, p. 61]{SauterSchwab2011}, but for Lipschitz and $\mathcal{C}^k$ domains.}

For the product $- \langle \mathcal{K} \mathcal{L}\VV, \mathcal{P}_2\VV\rangle$, a similar decomposition is applied, but with an additional decomposition to get the compactness result.
Indeed, let us decompose the map $\mathcal{P}_2$ as the composition $\mathcal{P}_2 = \mathcal{Q}_2 \circ \mathcal{Q}_1$
where $\mathcal{Q}_{1}: \vect{H}^{1/2}(\Sigma^\ast) \to H^{1}(\Omega^\ast)$, $\VV \mapsto \uip$, and $\mathcal{Q}_{2}: H^{1}(\Omega^\ast) \to H^{-1/2}(\Sigma^\ast)$, $\phi \mapsto {w}$.
Here, again, $\uip$ and ${w}$ respectively solves \eqref{eq:shape_derivative_of_the state} and \eqref{eq:Phi_equation}.
Next, we write $\mathcal{Q}_{2}$ as $\mathcal{Q}_{2} := \mathcal{S}_{4} \circ \mathcal{S}_{3} \circ \mathcal{S}_{2} \circ \mathcal{S}_{1}$ where 
\begin{align*}
	\mathcal{S}_{1}: H^{1}(\Omega^\ast) \to H^3(\Omega^\ast), \ \ \phi \mapsto {w}_1;\qquad
	\mathcal{S}_{2}: H^3(\Omega^\ast) \to H^{3/2}(\Sigma^\ast), \ \ {w}_1 \mapsto \dn{{w}_1};\\
	\mathcal{S}_{3}: H^{3/2}(\Sigma^\ast) \to H^{5/2}(\Sigma^\ast), \ \ \dn{{w}_1} \mapsto {w}_2;\qquad
	\mathcal{S}_{4}: H^{5/2}(\Sigma^\ast) \to H^{-1/2}(\Sigma^\ast), \ \ {w}_2 \mapsto {w}_2.
\end{align*}
The first three operators $\mathcal{S}_{1}$, $\mathcal{S}_{2}$, and $\mathcal{S}_{3}$ are continuous while $\mathcal{S}_{4}$ is the compact embedding of $H^{5/2}(\Sigma^\ast)$ into $H^{-1/2}(\Sigma^\ast)$.
This proves the compactness result.
\end{proof}
%
%
%
Related results, specifically given for $L^2$-tracking type and compact gradient tracking functionals, are issued in \cite[Prop. 3.1]{EpplerHarbrecht2006nm} and \cite[Prop. 2.10]{EpplerHarbrecht2008}.
It is worth to remark that the compactness of the shape Hessian stated in Proposition \ref{prop:compactness_of_shape_Hessian} opposes the strict coercivity ${\operatorname{\mathnormal{d}}}^2J(\Omega^\ast)[\VV,\VV] \gtrsim \|\VV\|^2_{\vect{H}^{1}(\Sigma^\ast)}$ that corresponds to the well-posedness of the proposed shape optimization formulation of the free boundary problem in consideration (on a related note, see \cite[Rem. 3.2--3.3]{EpplerHarbrecht2006nm} and \cite[Rem. 2.11]{EpplerHarbrecht2008}).  

\section{Numerical approximation}
\label{sec:Numerical_Approximation}
The numerical resolution to our proposed shape optimization approach to \eqref{eq:Bernoulli_problem} is carried out using a Sobolev gradient-based method.
The implementation is realized in line with the author's previous work using the finite element method, see \cite{RabagoAzegami2019a,RabagoAzegami2019b,RabagoAzegami2020}, but with some changes which are crucial for assessing the numerical performance of the new method over the classical Kohn-Vogelius cost functional approach.
The proposed shape optimization reformulation can of course be solved numerically using other methods such as the \textit{level-set method} (see \cite{OsherSethian1998}) -- an Eulerian-like type numerical scheme -- employed, for instance, in \cite{BenAbdaetal2013,HIKKP2009,IKP2006}, or via a \textit{boundary element method} through the concept of boundary integral equations used in \cite{EpplerHarbrecht2006,EpplerHarbrecht2009,EpplerHarbrecht2010,EpplerHarbrecht2012a,Harbrecht2008}.

\subsection{Numerical algorithm}
\label{subsec:Numerical_Algorithm}
For completeness, we give below the important details of our algorithm.

\textit{Choice of descent direction.} 
The choice $\VV = V_{n}\nn = -G\nn$, $G \in L^2(\Sigma)$, $G \not\equiv 0$, provides a descent direction for the cost function $J(\Omega)$ given in \eqref{eq:cost_function}.
Indeed, in general, the inequality condition $J(\Omega_t) = J(\Omega) + t \left. \frac{\operatorname{\mathnormal{d}}}{\operatorname{\mathnormal{d}}\varepsilon} J(\Omega_{\varepsilon})\right|_{\varepsilon = 0} + O(t^2)
		= J(\Omega) + t \intS{GV_{n}} + O(t^2)
		= J(\Omega) - t \intS{|G|^2} + O(t^2)
		< J(\Omega)$,
holds for sufficiently small real number $t>0$.	
However, as alluded above, we make use of the Riesz representation of the shape gradient.
More exactly, we apply an extension-regularization technique by taking the descent direction $\VV$ as the solution in $\vect{H}^{1}_{\Gamma, \vect{0}}(\Omega)$ to the variational problem $a(\VV,\vect{\varphi})= - \intS{\GG\nn \cdot \vect{\varphi}}$, for all $\vect{\varphi} \in \vect{H}^{1}_{\Gamma, \vect{0}}(\Omega)$, where $a$ is the ${\vect{H}^{1}(\Omega)}(:=H^{1}(\Omega)^{d}$)-inner product in $d$-dimension, $d \in \{2, 3\}$ (cf. eq. \eqref{eq:extended_regularized} in next subsection).
In this sense, the \textit{Sobolev gradient} $\VV$ \cite{Neuberger1997} becomes a smoothed preconditioned extension of $-\GG\nn$ over the entire domain $\Omega$.
For more discussion about discrete gradient flows for shape optimization, we refer the readers to \cite{Doganetal2007}.

\textit{The main algorithm.}
The main steps of the iterative algorithm, computing the $k$th domain $\Omega^{k}$, is summarized as follows:
\begin{description}
\setlength{\itemsep}{0.5pt}
	\item[1. \it{Initilization}] Choose an initial shape $\Omega^{0}$.  
	\item[2. \it{Iteration}] For $k = 0, 1, 2, \ldots$
		\begin{enumerate}
			\item[2.1] Solve the state and adjoint state systems on the current domain $\Omega^{k}$.
			\item[2.2] Choose $t^{k}>0$, and compute the Sobolev gradient $\VV^{k}$ in $\Omega^{k}$.
			\item[2.3] Update the current domain by setting $\Omega^{k+1} = (\operatorname{id}+t^{k}\VV^{k})\Omega^{k}$. 
		\end{enumerate}
	\item[3. \it{Stop Test}] Repeat the \textit{Iteration} until convergence.
\end{description}

\begin{remark}[Step-size computation]\label{rem:step_size}
The step size $t^{k}$ is computed via a backtracking line search procedure using the formula $t^{k} = \mu J(\Omega^{k})/|\VV^{k}|^2_{\mathbf{H}^{1}(\Omega^{k})}$ at each iteration, where $\mu > 0$ is a given real number.
This choice of the step size is based on an Armijo-Goldstein-like condition for the shape optimization method, see, for example, \cite[p. 281]{RabagoAzegami2020}.
In our application, however, the step size parameter $\mu$ is not limited to the interval $(0,1)$ in contrast to \cite{RabagoAzegami2019a,RabagoAzegami2019b,RabagoAzegami2020}.
The less restrictive choice of the value for $\mu$ allows us to numerically evaluate the sensitivity of the cost functions $J(\Omega)$ and $J_{KV}(\Omega)$ through large deformations of the domain $\Omega$.   
\end{remark}

\begin{remark}[Stopping conditions] \label{rem:stopping_condition}
The algorithm is stopped as soon as $\Omega^{k}$, $\Sigma^{k}$, and $\VV^{k}$ satisfy the inequality condition 
\[
	\max\left( \sqrt{a(\VV^{k},\VV^{k})}, \|\VV^{k}\|_{\mathcal{C}(\Sigma^{k})^{d}}, J(\Omega^{k}) \right) < \texttt{Tol},
\]
for some fixed small value $\texttt{Tol} > 0$.
We also terminate the algorithm as soon as the absolute difference $|J(\Omega^{k}) - J(\Omega^{k-1}) |$ is small enough, or after a finite number of iterations.
\end{remark}
%
%
%
\begin{remark}\label{rem:Newton_method}
The convergence behavior of a gradient-based iterative scheme can be improved by incorporating the Hessian information in the numerical procedure.
The drawback, however, of a second-order method is that, typically, it demands additional computational burden and time to carry out the calculation, especially when the Hessian is complicated \cite{NovruziRoche2000,Simon1989}. 
Here, we will not employ a second-order method to numerically solve the optimization problem.
The order two analysis was performed here only to carry out the stability analysis for the proposed optimization problem.
\end{remark}
\subsection{The extension-regularization technique preserves the critical shape}
\label{subsec:properties}
In this intermediate subsection, we issue a small result concerning the stationary point $\Omega^\ast$ of the evolving boundary $\Omega(t)$, where $t \geqslant 0$ (interpreted here as a ``pseudo-time'' step), that evolves from an initial geometric profile $\Omega(0)$ under the pseudo flow field $V_{n}\nn$.
This flow field is related to the extension-regularization technique used to compute the descent field at the beginning of the previous section.
Here, there will be a slight abuse of notation.
In previous discussions $\VV$ stands for the deformation field that deforms the reference domain $\Omega$ using the application of the operator $T_t$.
In the arguments given below, we will be using the same notation to represent an extended-regularized (normal) flow field for the evolving boundary $\Sigma(t)$ which, in some sense, related to the extension-regularization technique presented in \ref{subsec:Numerical_Algorithm} for the computation of the associated descent direction (i.e., in other words, we somehow view the evolution of the free boundary $\Sigma$ generated by the approximation process as an evolving boundary problem).
We emphasize that the discussion given below does not attempt to prove the existence and/or convergence of approximate shape solutions concerning the present shape optimization problem.
Indeed, a careful analysis along the lines of arguments used in \cite[Sec. 3]{EpplerHarbrechtSchneider2007} (see also \cite{HaslingerMakinen2003}) is needed to address such delicate issue and actually goes beyond the scope of the present study.
Besides, some key assumptions on the admissible set and on the functional $J$ have to be imposed (see \cite[Sec. 3]{EpplerHarbrechtSchneider2007}).
Moreover, as in many optimization problems, it is important to define the notion of convergence of $\Omega(t)$ to the stationary point $\Omega^{\ast}$ if one wants to view the sequence $\{\Omega(t_{k})\}$, $k=0,1,\ldots$, as a sequence of approximations of $\Omega^{\ast}$ in the discrete setting; see \cite[Chap. 2]{HaslingerMakinen2003}.
For a closely related topic concerning a finite element approximation for shape optimization problems with mixed boundary conditions, we refer the readers to \cite{Tiba2011}.
In the sequel, the phrases `moving boundary', `evolving boundary', and `free boundary' are used interchangeably. 

To proceed, let us consider the following abstract autonomous evolving boundary problem.
\begin{problem}
\label{prob:abstractMBP}
	Let $\Omega$ be a bounded annular domain with $\mathcal{C}^{\circ}$ regular boundaries $\Gamma$ and $\Sigma$ such that $\Sigma$ is exterior to the fixed (non-moving) boundary $\Gamma$.
	Given an initial geometry $\Sigma^{0}$ that is $\mathcal{C}^{\circ}$ regular and a real-valued function $\Phi$ defined on $\Sigma$ (i.e., $\Phi(\,\cdot\,;\Sigma):\Sigma \to \mathbb{R}$), find a moving boundary/surface $\Sigma(t)$, $t\geqslant 0$, with the normal speed $V_{n}$ which satisfies
	\begin{equation}
	\label{eq:abstractMBP}
		V_{n}(x,t)	= \Phi(x;\Sigma(t)),\quad x \in \Sigma(t), \quad t \geqslant 0,\qquad
		\Sigma(0)	= \Sigma^{0}.
	\end{equation}
\end{problem}

In \eqref{eq:abstractMBP}, to keep the regularity of the initial domain -- which we assume to be at least $\mathcal{C}^{\circ}$ regular --
during evolution, one needs the function $\Phi(x;\Sigma(t))$ to also be at least $\mathcal{C}^{\circ}$ for all $x \in \Sigma(t)$, for $t>0$ (unless specified, this will be assumed in the rest of the discussion).
Given this assumption on $\Phi$, for sufficiently small $\varepsilon > 0$, it can be shown that the moving boundary $\Sigma(t)$ actually maintains the $\mathcal{C}^{\circ}$ regularity throughout the short time interval $[0,\varepsilon)$.

We next define a stationary solution to Problem \ref{prob:abstractMBP} as follows. 
\begin{definition}
	A domain $\Omega^{\ast}$ is said to be a \textit{stationary solution} to Problem \ref{prob:abstractMBP} if $\Sigma^{\ast} = \partial \Omega^{\ast}\setminus \Gamma$, and $\Phi(x;\Sigma^{\ast}) = 0$, for all $x \in \Sigma^{\ast}$.
\end{definition}
%
%
Also, in accordance with the discussion issued in the previous subsection about the extension-regularization technique used to compute for the perturbation field, we associate with Problem \ref{prob:abstractMBP} the extended-regularized evolving boundary problem stated as follows.
\begin{problem}
\label{prob:extended_regularized}
	Given an annular domain $\Omega^{0} \in \mathcal{C}^{\circ}$ with boundary $\Gamma \cup \Sigma^{0}$ ($\Sigma^{0}$ is exterior to $\Gamma$) and a real-valued function $\Phi(\,\cdot\, ;\Sigma(t)) \in L^2(\Sigma(t)) \cap \mathcal{C}^{1,\alpha}(\Sigma(t))$, $\alpha > 0$, $t>0$, we seek to find an evolving boundary/surface $\Sigma(t)$ with initial profile $\Sigma(0)=\Sigma^{0}$ that solves the problem
	\begin{equation}
	\label{eq:extended_regularized}
	\left\{\arraycolsep=1.4pt\def\arraystretch{1}
	\begin{array}{rcll}
		- \Delta\VV + \VV &=& \vect{0}	&\text{in $\Omega(t)$},\\
		{\VV} 			&=& \vect{0} 	&\text{on $\Gamma$},\\
 		\nabla \VV \cdot \nn 	&=& \Phi(\,\cdot\, ;\Sigma(t)) {\nn} & \text{on $\Sigma(t)$},\\
		V_{n} 			&=& \VV \cdot \nn	& \text{on $\Sigma(t)$,\quad $t \geqslant 0$},\\
		\Omega(0) 		&=& \Omega^{0}.
	\end{array}
	\right.
	\end{equation}
\end{problem}
For a domain $\Omega$ that is $\mathcal{C}^{2,\alpha}$ regular, $\alpha > 0$, the (outward unit) normal vector $\nn$ is $\mathcal{C}^{1,\alpha}(\partial\Omega)$ smooth.
Hence, for fixed $t\geqslant 0$ and $\Omega(t) \in \mathcal{C}^{2,\alpha}$, it can be proved (using, for example, the results from \cite{GilbargTrudinger2001,LadyzenskajaUralceva1968}) that the first three equations in \eqref{eq:extended_regularized} admits a unique (classical) solution $\VV \in \vect{\mathcal{C}}^{2,\alpha}(\Omega(t))$, for any given $\Phi(\,\cdot\, ;\Sigma(t)) \in \mathcal{C}^{1,\alpha}(\Sigma(t))$, $\alpha > 0$.
Notice here that, for fixed $t$, $\Phi(\,\cdot\, ;\Sigma(t))$ only needs to be $\mathcal{C}^{1,\alpha}(\Sigma(t))$ regular and $\Omega(t) \in \mathcal{C}^{2,\alpha}$ for $\VV$ to be $\vect{\mathcal{C}}^{2,\alpha}(\Omega(t))$ smooth.
We also remark that the evolution of $\Omega(t)$ is essentially assumed here as perturbations of $\Omega^{0}$ that can also be obtained via a diffeomorphic map which is close to the identity.
%
%
%
%

	We called equation \eqref{eq:extended_regularized} an extended-regularized evolving boundary problem because, originally in Problem \ref{prob:abstractMBP}, $V_{n}$ is only supported on $\Sigma(t)$,
	and, in \eqref{eq:extended_regularized}, we want not only to extend the vector $V_{n}\nn$ in the entirety of $\Omega(t)$, for $t>0$ via equation \eqref{eq:extended_regularized}$_1$, but also to add more regularity on the normal flow field through equation \eqref{eq:extended_regularized}$_3$.
	Meanwhile, we have equation \eqref{eq:extended_regularized}$_2$ on $\Gamma$ since we want the interior boundary to remain fixed during evolution.
%
%
%

Finally, in relation to Problem \ref{prob:extended_regularized}, a stationary solution $\Omega^{\ast}$ is define next.
\begin{definition}
	A domain $\Omega^{\ast} \in \mathcal{C}^{\circ}$ is said to be a \textit{stationary solution} to Problem \ref{prob:extended_regularized} if $\Sigma^{\ast} = \partial \Omega^{\ast}\setminus \Gamma$ and ${\VV} \in \vect{H}^{1}_{\Gamma, \vect{0}}(\Omega^\ast) \cap \vect{\mathcal{C}}^{\circ}(\Omega^\ast)$\footnote{Here we define $\vect{H}^{1}_{\Gamma, \vect{0}}(\Omega^\ast):=H_{\Gamma, \vect{0}}^1(\Omega^\ast ;\mathbb{R}^d)$, $d\in\{2,3\}$.} satisfies the equation
\begin{equation}\label{eq:Sobolev_gradient_equation}
  \begin{aligned}\
    \intOast{ \left( \nabla {\VV} : \nabla \vect{\varphi} +  {\VV} \cdot \vect{\varphi}\right) } 
    	 &= \intSast{ \Phi(\cdot;\Sigma ^{\ast}) {\nn} \cdot \vect{\varphi}},
            \quad \forall \vect{\varphi} \in \vect{H}^{1}_{\Gamma, \vect{0}}(\Omega^\ast),\\
      	\text{and}\qquad \VV \cdot \nn &= 0 \quad \text{on $\Sigma^\ast$}.
  \end{aligned}
\end{equation}
\end{definition}
For a domain $\Omega^\ast$ of class $\mathcal{C}^{0,1}$ and function $\Phi(\,\cdot\, ;\Sigma) \in L^2(\Sigma)$, the variational problem \eqref{eq:Sobolev_gradient_equation} can be shown to have a weak solution ${\VV} \in \vect{H}^{1}_{\Gamma, \vect{0}}(\Omega^\ast)$ via Lax-Milgram lemma.
With the definitions given above, we will now issue the main point of this subsection which is given in the next proposition.
Here, we will tacitly assume -- for the sake of argument -- that $\Phi$ vanishes within a short time interval and that we have the convergence of the evolving domains to a stationary point $\Omega^{\ast}$ at that time interval without referring to a formal mathematical notion of convergence of sets.
On a related note, we point out that small and smooth perturbations of a regular domain may be ``uniquely'' described by normal deformations of the boundary of the domain, see \cite{NovruziPierre2002}.
\begin{proposition}
\label{prop:convergence_to_a_stationary_point}
	Let $\Omega^{\ast} \in \mathcal{C}^{\circ}$ and $\Phi(\,\cdot\, ;\Sigma) \in L^2(\Sigma) \cap \mathcal{C}^{\circ}(\Omega)$.
	Then, $\Omega^{\ast}$ is a stationary solution to Problem \ref{prob:abstractMBP} if and only if $\Omega^{\ast}$ is a stationary solution to Problem \ref{prob:extended_regularized}.
\end{proposition}
\begin{proof}
	Consider equation \eqref{eq:extended_regularized} over the stationary shape $\overline{\Omega^\ast}$ with Lipschitz boundary $\partial\Omega^{\ast} = \Gamma \cup \Sigma^\ast$.
	For the necessity part, we assume that $L^2(\Sigma) \cap \mathcal{C}^{\circ}(\Omega) \ni \Phi(\,\cdot\, ;\Sigma) = 0$, and we need to verify that $\VV \cdot \nn = 0$ on $\Sigma^\ast$.
	To this end, we multiply the first equation in \eqref{eq:Sobolev_gradient_equation} by $\VV \in \vect{H}^{1}_{\Gamma, \vect{0}}(\Omega^\ast) \cap \vect{\mathcal{C}}^{\circ}(\Omega^\ast)$ and then apply integration by parts -- noting that $\VV = \vect{0}$ on $\Gamma$ -- to obtain
	$0 \leqslant
		\intOast{|\VV|^2}
		= - \intSast{\dn{\VV} \cdot \VV}
		= - \intOast{|\nabla \VV|^2}
		\leqslant 0$.
	Clearly, $\VV = \vect{0}$ in $\Omega^{\ast}$.
	Moreover, because $\VV\big|_{\Gamma} = \vect{0}$, then, by the maximum principle, $\VV \equiv \vect{0}$ on $\overline{\Omega}^{\ast}$,
	In particular, we have $\VV\cdot {\nn} = 0$ on $\Sigma^\ast$.

  	For the sufficiency part, we need to prove that if $\VV \cdot \nn = 0$ on $\Sigma^\ast$, where $\VV$ solves problem \eqref{eq:extended_regularized} on $\overline{\Omega^\ast}$, then $\Phi = 0$ on $\Sigma^\ast$.
	We take $\vect{\varphi} = \VV \in \vect{H}_{\Gamma, \vect{0}}^1(\Omega^\ast)$ in \eqref{eq:Sobolev_gradient_equation} from which we get
	$\intOast{ \left( \nabla {\VV} : \nabla {\VV} +  {\VV} \cdot {\VV}\right) } 
			= \intSast{\Phi(\,\cdot\, ;\Sigma^{\ast}) {\nn}\cdot\VV} = 0$.
	Clearly, $\VV \equiv \vect{0}$ on $\overline{\Omega}^{\ast}$.
	Going back to \eqref{eq:Sobolev_gradient_equation}, we obtain $\intSast{\Phi(\,\cdot\, ;\Sigma^{\ast}) {\nn} \cdot \vect{\varphi}} = 0$, for all $\vect{\varphi} \in \vect{H}_{\Gamma, \vect{0}}^1(\Omega^\ast)$, from which we conclude that $\Phi = 0$ on $\Sigma^\ast$.
	This proves the assertion.
\end{proof}
%
%
%

Let us look at the situation when $\Phi = -G$ in \eqref{eq:abstractMBP}.
In this case, $\Phi$ not only depends on some geometric quantities on the free boundary, but also to some functions which are solutions to specific equality constraints (the state and adjoint state equations to be exact).
For a domain $\Omega$ that is $\mathcal{C}^{k,\alpha}$ regular, $k \in \mathbb{N}$, $k \geqslant 2$, it can be shown that both the state and the adjoint state systems \eqref{eq:complexPDE} and \eqref{eq:adjoint_system}, respectively, admit a unique (classical) solution in the space $\mathcal{C}^{k,\alpha}(\overline{\Omega}; \mathbb{C}^{d})$.
In particular, the imaginary part $u_{2}$ of the state variable $u$ is an element of the set $\mathcal{C}^{k,\alpha}(\overline{\Omega}; \mathbb{C}^{d})$.
This implies that the adjoint state $p$ is even more regular, and is in fact $\mathcal{C}^{k+2,\alpha}$ smooth in $\overline{\Omega}$.
Meanwhile, the normal vector $\nn$ to $\Omega \in \mathcal{C}^{3,\alpha}$ is only $\vect{\mathcal{C}}^{2,\alpha}$ regular, implying that the mean curvature $\kappa$ is $\mathcal{C}^{1,\alpha}(\partial\Omega)$ smooth.
Since $G$ consists of $\kappa$, it appears that one even needs the initial domain $\Omega^{0}$ be at least $\mathcal{C}^{4,\alpha}$ regular for the evolving domain $\Omega(t)$ to be $\mathcal{C}^{2,\alpha}$ smooth for some short time interval $[0, \varepsilon)$.
On the other hand, considering Problem \ref{prob:extended_regularized} with $\Phi = -G$, it appears that it is enough to assume that $\Omega^{0} \in \mathcal{C}^{3,\alpha}$ for the moving domain $\Omega(t)$ to be $\mathcal{C}^{2,\alpha}$ smooth in some short time interval $[0, \varepsilon) \ni t$.
Indeed, in this case, $\Phi\nn = -G\nn \in \vect{\mathcal{C}}^{1,\alpha}(\Sigma(t))$, and so $\VV \in \vect{\mathcal{C}}^{2,\alpha}(\Omega(t))$, for $t \in [0, \varepsilon)$.
Hence, it can actually be shown that the weak solution ${\VV} \in \vect{H}^{1}_{\Gamma, \vect{0}}(\Omega(t))$ of \eqref{eq:extended_regularized} is also a classical solution of the problem over the small time interval $[0, \varepsilon)$.
Therefore, $\Omega(t)$ in \eqref{eq:extended_regularized} remains $\mathcal{C}^{2,\alpha}$ smooth throughout the short time interval $[0, \varepsilon)$.

	Although not so important, we provide additional comments about the evolving boundaries. 
	The evolution of the moving boundary $\Sigma(t)$ due to \eqref{eq:abstractMBP} and that of \eqref{eq:extended_regularized} through time are not the same.
	Indeed, only in the case that Proposition \ref{prop:convergence_to_a_stationary_point} is true we are sure that the two evolving boundaries coincide (except of course with the initial shape).
	Nevertheless, a more accurate extended-regularized version of \eqref{eq:abstractMBP} can be formulated by replacing the Neumann condition in \eqref{eq:extended_regularized} by the Dirichlet condition $\VV = \Phi(\,\cdot\, ;\Sigma(t)) {\nn}$ on $\Sigma(t)$, for $t>0$.
	In this case, however, the regularity of the moving boundary $\Sigma(t)$ is not preserved when the evolving boundary evolves according to the normal speed $V_{n} = \VV \cdot \nn$.
	Even so, one can address the issue to some extent by considering an approximation of the Dirichlet condition using a Robin condition.
	That is, we can add more regularity to the vector $V_{n}\nn$ by setting $\beta \nabla \VV + \VV = \Phi(\,\cdot\, ;\Sigma(t)) {\nn}$ on $\Sigma(t)$, for $t>0$, where $\beta > 0$.
	Still, with this alternative formulation, there is a trade-off (controlled by $\beta$) between the accuracy of the evolution and the regularity of the moving boundary.
	Here, we do not bother about the accuracy of the extended-regularized evolving boundary with respect to the corresponding original one as we are only interested in extending in $\Omega(t)$ the normal velocity vector $\Phi\nn$ of the evolving boundary $\Sigma(t)$ while adding more regularity to it.


\subsection{Numerical examples}\label{subsec:Numerical_Examples}
We now illustrate the feasibility and applicability of the new method in solving concrete examples of the free boundary problem \eqref{eq:Bernoulli_problem}.
We first test our method in two dimensions, and carry out a comparison with KVM (Examples \ref{example2d1}--\ref{example2d3}).
Also, since most of the previous studies only dealt with problems in two dimensions (except in \cite{EpplerHarbrecht2009} and also in \cite{Harbrecht2008} which applies the Newton scheme to the Dirichlet energy functional), we also put our attention on testing CCBM to three dimensional cases (Examples \ref{example3d1}--\ref{example3d5}).
In the case of three spatial dimensions, we let $\lambda = -10$, unless specified.
\begin{remark}[Details of the computational setup and environment]\label{rem:details_or_algorithm}
The numerical simulations conducted here are all implemented in the programming software \textsc{FreeFem++}, see \cite{Hecht2012}.
Every variational problem involved in the procedure is solved using $\mathbb{P}1$ finite element discretization
and are solved in \textsc{FreeFEM++} via the command line \texttt{problem} with the default solver (\texttt{sparsesolver} or \texttt{LU} if any direct sparse solver is available\footnote{see, e.g., p. 380 of FreeFem Documentation, Release 4.8, May 16, 2022, www.freefem.org.}).
Moreover, all mesh deformations are carried out without any kind of adaptive mesh refinement as opposed to what has been usually done in earlier works, see \cite{RabagoAzegami2019a,RabagoAzegami2019b,RabagoAzegami2020}.
We emphasize that, in this way, we can further assess the stability of CCBM in comparison with KVM.
The computations are all performed on a MacBook Pro with Apple M1 chip computer having 16GB RAM processors.
\end{remark}
\begin{remark}
As pointed out in Remark \ref{rem:details_or_algorithm}, we avoid the generation of a new triangulation of the domain at every iterative step. 
Obviously, this is achieved by moving not only the boundary, but also the internal nodes of the mesh triangulation at every iteration (as mentioned earlier).
By doing so, the mesh only needs to be generated at initial iteration. 
In order to move the boundary and internal nodes simultaneously, we solve the discretized version of \eqref{eq:extended_regularized} and then move the domain in the direction of the resulting vector field scaled with the pseudo-time step size $t^{k}$ (see Remark \ref{rem:step_size} for the computation of $t^{k}$).
That is, we find ${{\VV_{h}^{k}}} \in \mathbb{P}_1(\Omega_{h}^{k})^{d}$ such that it solves the equation
\[
	- \Delta {\VV_{h}^{k}} + {\VV_{h}^{k}}  		=  \vect{0} \ \ \text{in $\Omega_{h}^{k} $},\qquad
		{\VV_{h}^{k}} 				=  \vect{0} \ \ \text{on $\Gamma^h$},\qquad
 		{\nabla{\VV_{h}^{k}}} \cdot {\nn_{h}^{k}}	= -G^{k} {\nn_{h}^{k}} \ \ \text{on $\Sigma_{h}^{k}$},
\]
where we suppose a polygonal domain $\overline{\Omega_{h}^{k}}$ and its triangulation $\mathcal{T}_h(\overline{\Omega_{h}^{k}}) = \{ K^{k}_{l} \} ^{N_e}_{l=1}$ ($K^{k}_{l}$ is a closed triangle for $d=2$, or a closed tetrahedron for $d=3$) are given, and $\mathbb{P}_1(\Omega_{h}^{k})^{d}$ denotes the $\mathbb{R}^d$-valued piecewise linear function space on $\mathcal{T}_h(\overline{\Omega_{h}^{k}})$.
Then, we update the domain or, equivalently, move the nodes of the mesh by defining $\Omega^{k+1}_h$ and $\mathcal{T}_h(\overline{\Omega^{k+1}_h}) = \{ K^{k+1}_{l} \} ^{N_e}_{l=1}$ respectively as $\overline{\Omega^{k+1}_h}:=\left \{ x + t^k {\VV_{h}^{k}}(x) \ \middle\vert \ x \in \overline{\Omega_{h}^{k}} \right\}$ and $K^{k+1}_{l} := \left\{ x + t^k {\VV_{h}^{k}}(x) \ \middle\vert \ x \in K^{k}_l \right\}$, for all $k = 0,1,\ldots$. 
\end{remark}
We are now ready to give our first numerical example.
\begin{example}[Testing the accuracy of the gradient]\label{example2d1}
Consider two concentric circles centered at the origin $\mathbf{0}$ with radius $r > 0$ and $R > r$ given by $C(\boldsymbol{0},r)$ and $C(\boldsymbol{0},R)$, respectively.
Then, problem \eqref{eq:state_ud} can be expressed as the PDE system
 \[
	- \partial_{\rho\rho}^2 u  - \rho^{-1} \partial_{\rho} u = 0 \ \ \text{for $r < \rho < R$}, \qquad
	u(r) = 1, \qquad \text{and} \qquad 
	u(R) = 0,
\]
whose exact solution can be computed as $u(\rho) = \log{(\rho/R)}/\log{(r/R)}$.
Moreover, in this case, $\dn{u(R)} = 1/[R\log{(r/R)}]$.
Therefore, the exterior Bernoulli FBP \eqref{eq:Bernoulli_problem} with $\Gamma = \{x \in \mathbb{R}^2 : |x| = r\}$ and $\lambda = 1/\left[ R\log{\left(r/R\right)} \right]$, $0 < r < R$, has the unique exact free boundary solution $\Sigma^\ast = C(\boldsymbol{0},R)$. 
With this in mind, we let $r=0.5$ (i.e., $\Gamma=\Gamma_{\texttt{C}}:=C(\boldsymbol{0},0.5)$) and $R^\ast=0.7$, which gives us $\lambda = -4.24573$, and set $\Sigma^{0}=C(\boldsymbol{0},1.25)$ as our initial guess.
In this experiment, we examine the sensitivity of the cost functions $J$ and $J_{KV}$ by testing the methods with large variations.
To this end, we consider three test cases for KV by varying the step size parameter $\mu$: (KV)$_{1}$: $\mu = 2.0$; (KV)$_{2}$: $\mu = 1.0$; and (KV)$_{3}$: $\mu = 0.5$, while keeping $\mu = 2.0$ for CCBM.
Recall that $\mu$ dictates how large the magnitude of $t^{k}$ can be at every iteration, see Remark \ref{rem:step_size}.
Moreover, we discretize the initial domain with uniform mesh sizes and look at the effect of accuracy of the methods under different mesh widths.
Finally, we stop our algorithm as soon as the absolute difference between consecutive cost values is less than $10^{-6}$.
Figure \ref{fig2D1:cost_values_and_HD} shows the histories of cost values for KVM and CCBM, as well as the histories of Hausdorff distances $d_{\text{H}}(\Sigma^{k},\Sigma^\ast)$.
A summary of Hausdorff distances, cpu-time, and cpu-time-per-iteration against mesh sizes $h = 0.2, 0.1, 0.05, 0.025, 0.0125$ in the form of plots are plotted in Figure \ref{fig2D1:summary}.
Based from these results, we draw the following observations:
\begin{itemize}
	\item $J_{KV}$ is less sensitive to $J$ in terms of large variations;
	\item KVM tends to converge prematurely for coarse meshes unlike CCBM;
	\item KVM and CCBM nearly have the same convergence behavior for finer meshes;
	\item KVM and CCBM almost have the same accuracy for finer meshes, but CCBM is more accurate in general (primarily because of the second point);
	\item KVM and CCBM complete the iteration procedure at almost the same time,
	\item however, CCBM requires less computing-time-per-iteration than KVM,
	\item and the latter converges in less number of iterations (hence, the previous point).
\end{itemize}
Based from these observations, it seems that, in the case of two dimensional problems, CCBM features some merits over the KVM in terms of computational aspects.
Of course, we expect that these advantages can be exploited especially when dealing with three dimensional problems as we can instead resort to coarse meshes without concerning much ourselves with the accuracy of the approximation process -- at least when dealing with axisymmetric cases.
Also, it appears that we have more freedom to take large step sizes in the case of CCBM than when applying KVM as the former is less prone to premature convergence. 
However, as we already mentioned, KVM requires less number of iterations than CCBM, and that, under small step sizes, the two methods converges to almost identical optimal shape solutions.
\end{example}
%
%
%
\begin{figure}[htp!]
\centering
\resizebox{0.19\linewidth}{!}{\includegraphics{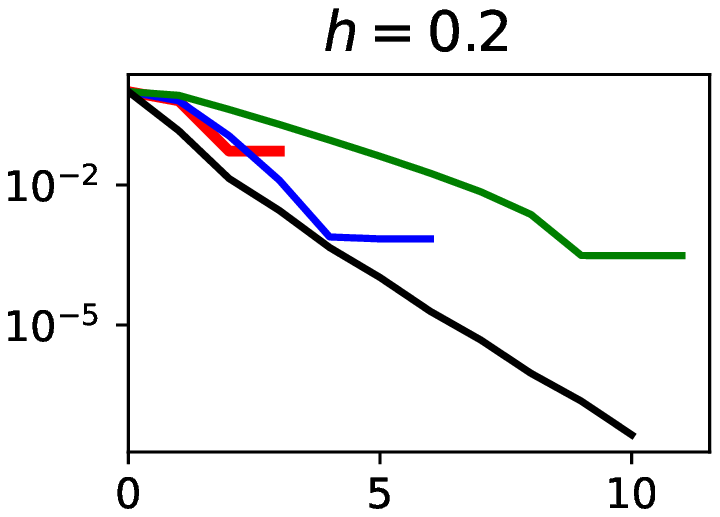}}
\resizebox{0.19\linewidth}{!}{\includegraphics{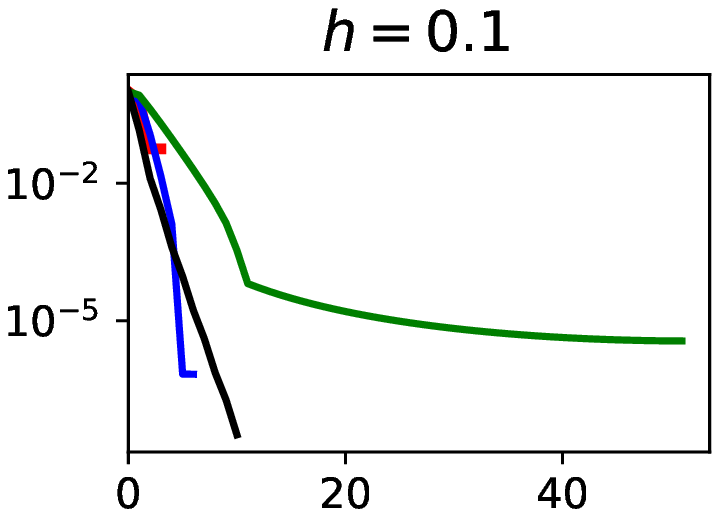}}
\resizebox{0.19\linewidth}{!}{\includegraphics{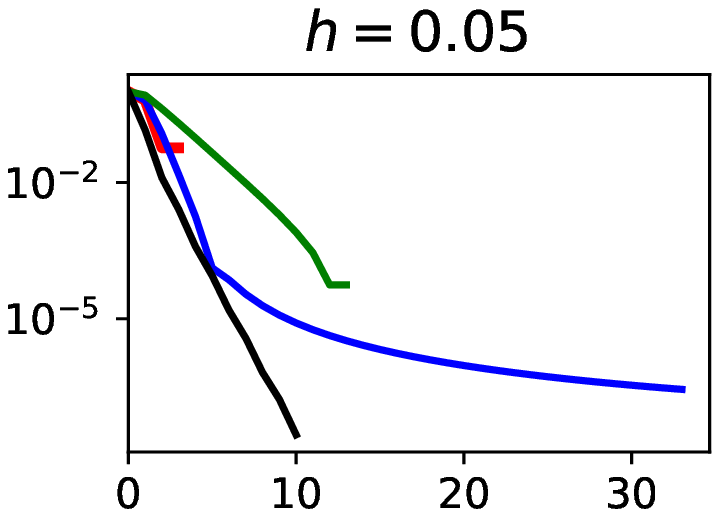}}
\resizebox{0.19\linewidth}{!}{\includegraphics{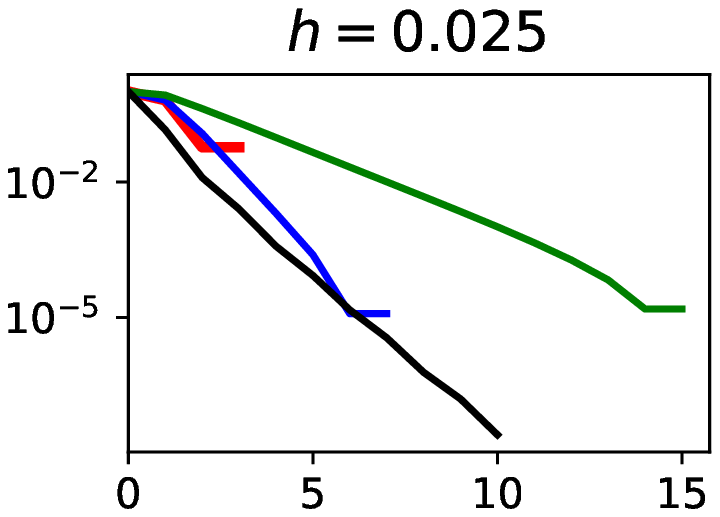}}
\resizebox{0.19\linewidth}{!}{\includegraphics{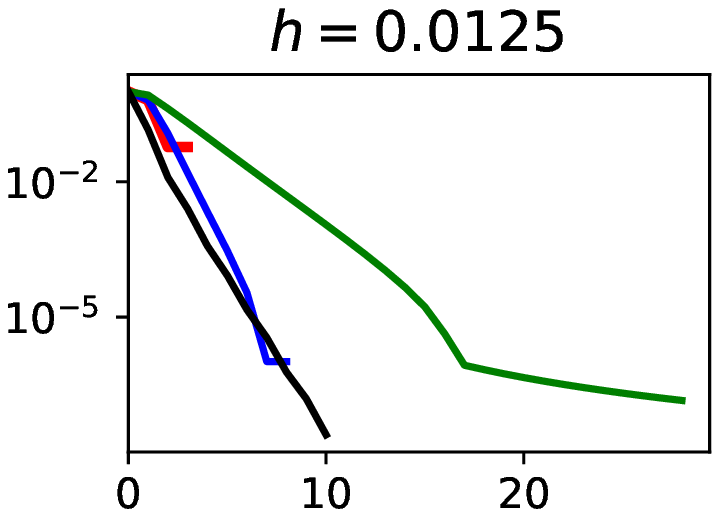}}\\
\resizebox{0.19\linewidth}{!}{\includegraphics{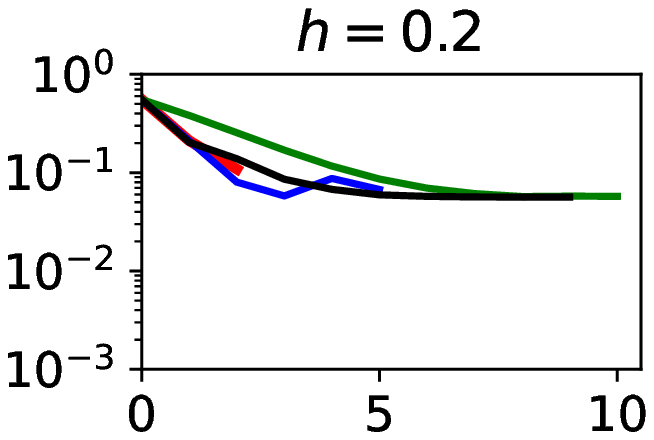}}
\resizebox{0.19\linewidth}{!}{\includegraphics{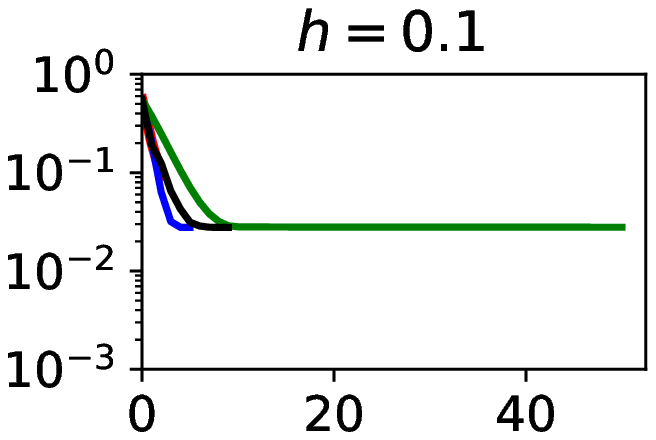}}
\resizebox{0.19\linewidth}{!}{\includegraphics{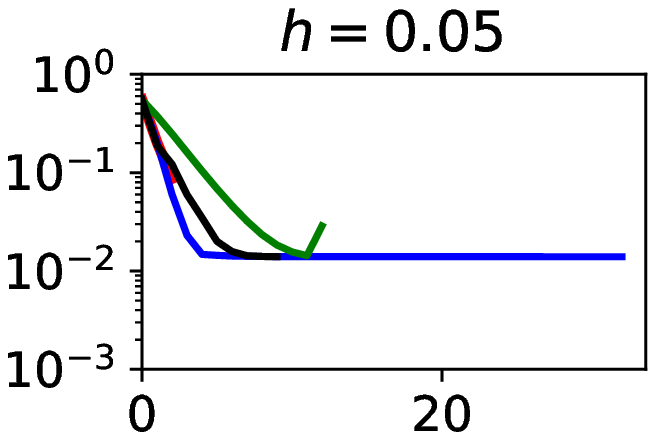}}
\resizebox{0.19\linewidth}{!}{\includegraphics{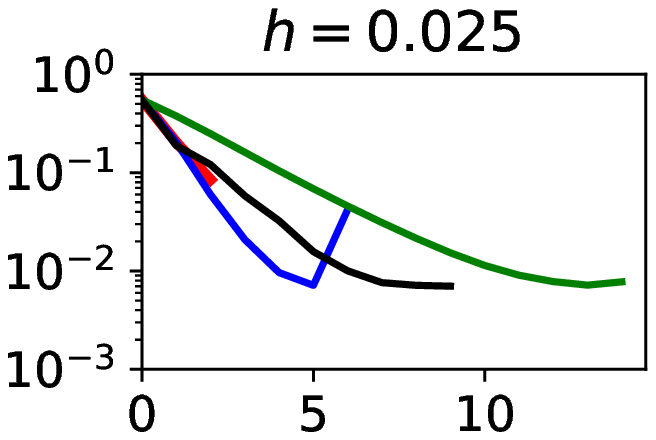}}
\resizebox{0.19\linewidth}{!}{\includegraphics{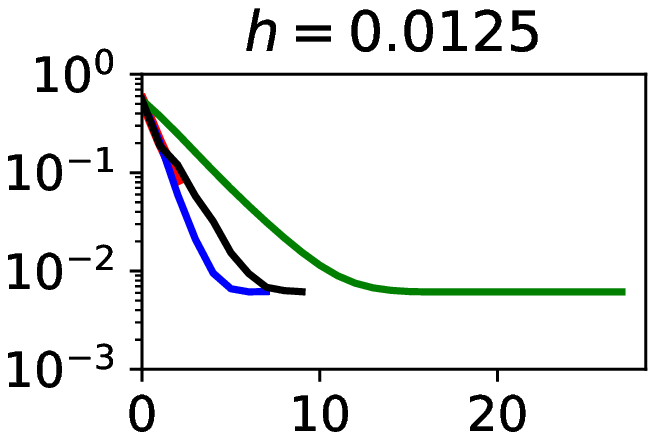}}
\caption{Histories of cost values (first row) and Hausdorff distances (second row) for Example \ref{example2d1}}
\label{fig2D1:cost_values_and_HD}
\end{figure}
%
%
%
%
%
%
%
%
%
%
%
%
\begin{figure}[htp!]
\centering
\resizebox{0.28\linewidth}{!}{\includegraphics{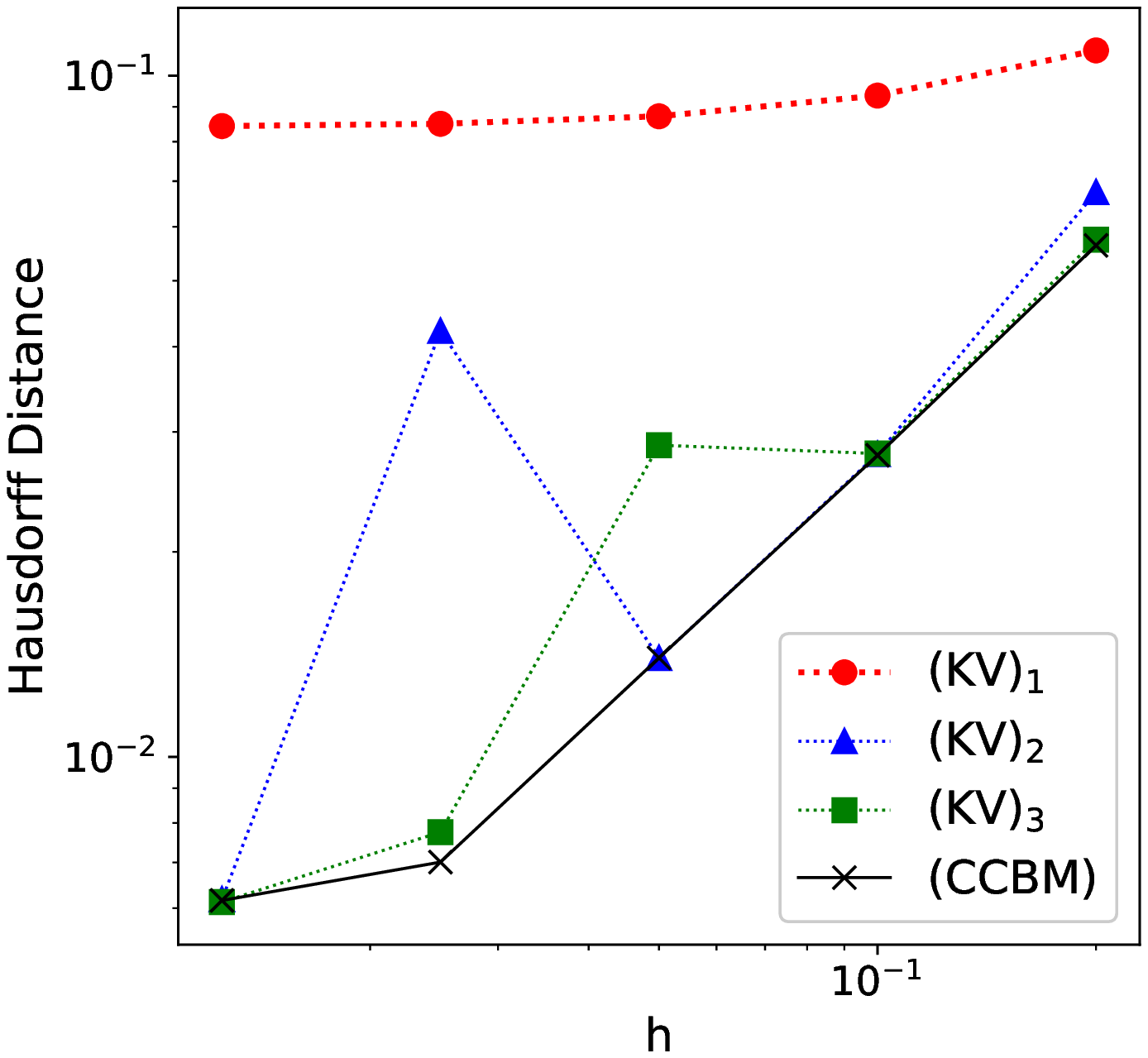}}\quad
\resizebox{0.28\linewidth}{!}{\includegraphics{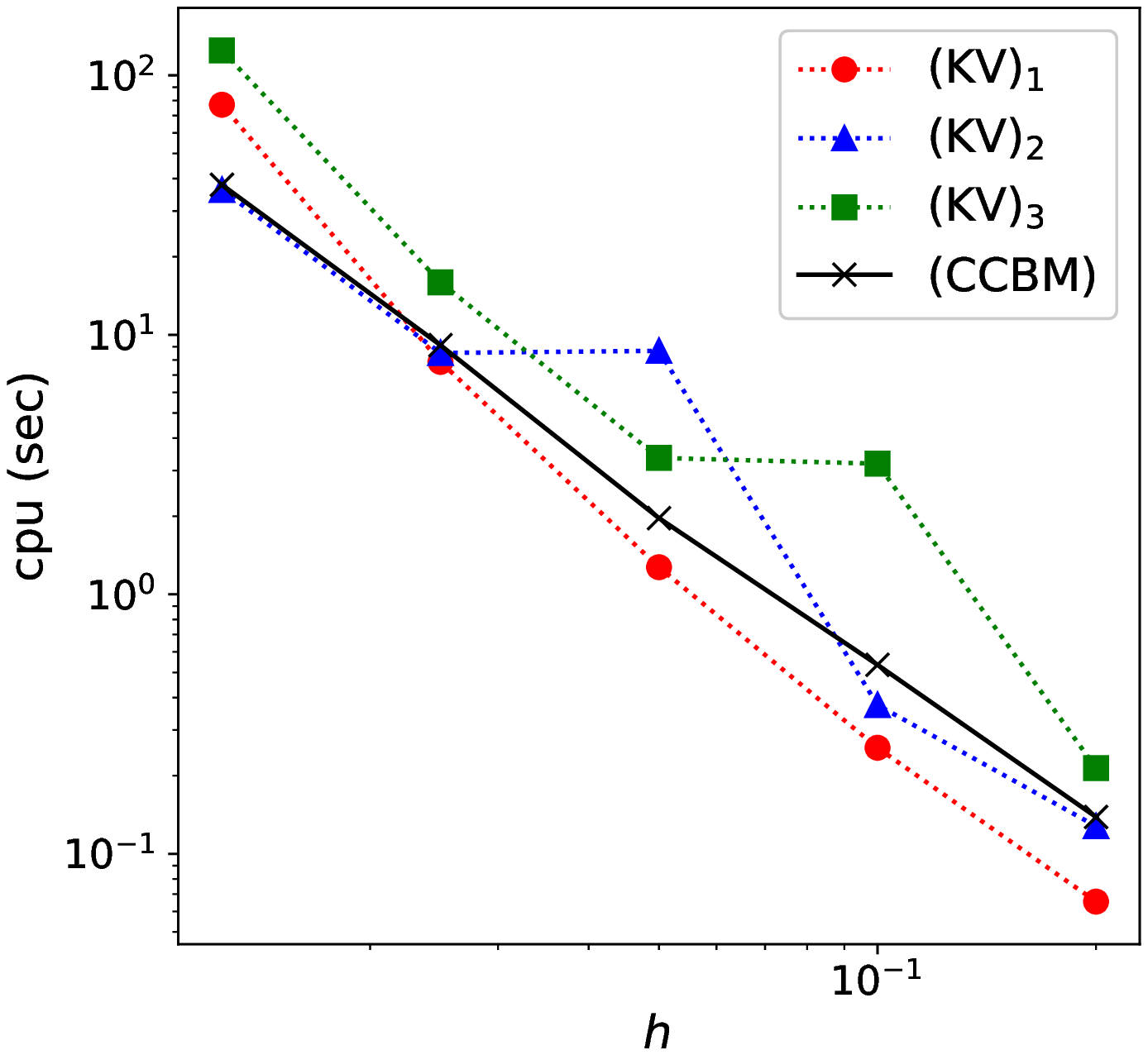}}\quad
\resizebox{0.28\linewidth}{!}{\includegraphics{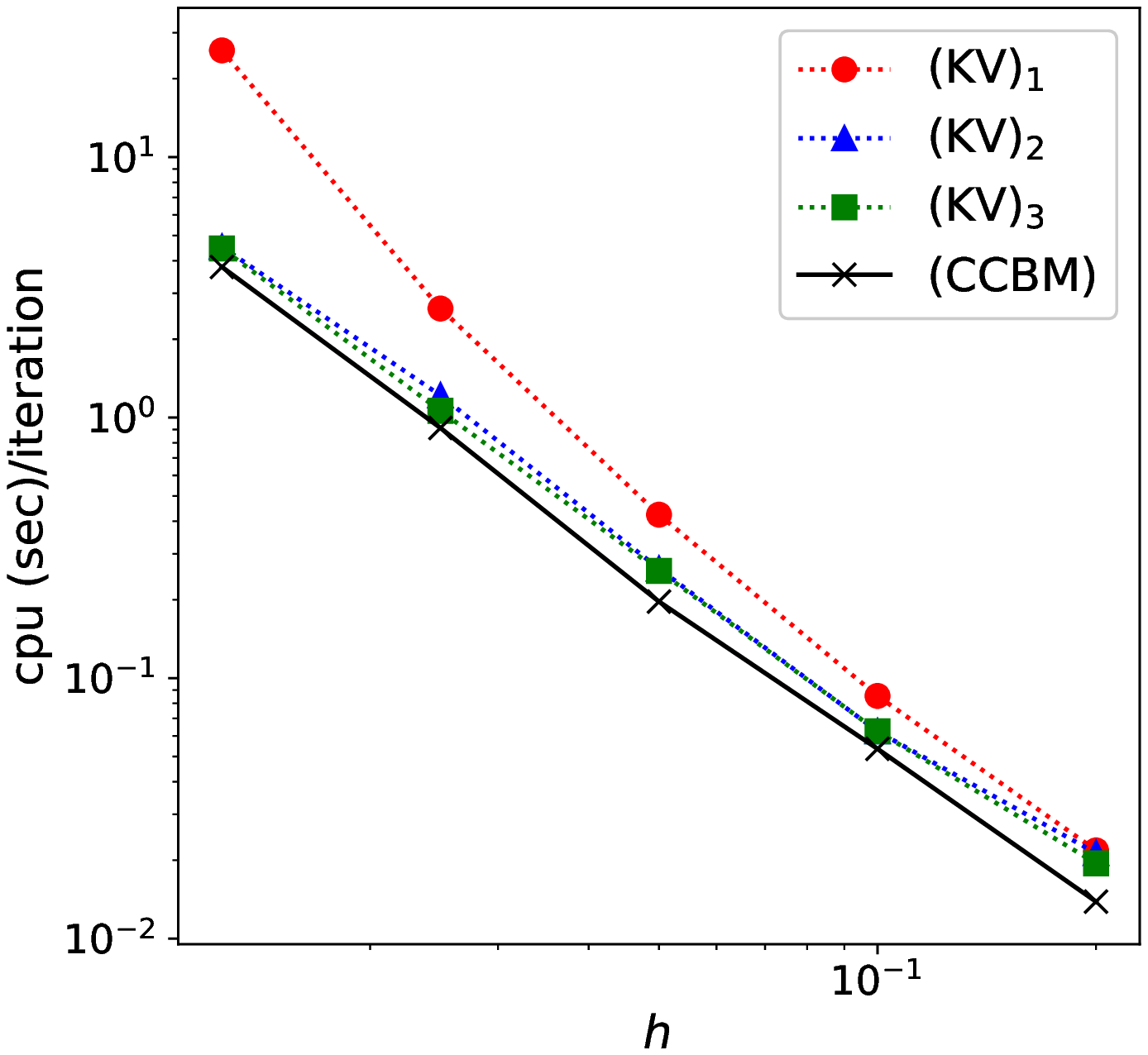}} 
\caption{Hausdorff distances, cpu-time, and cpu-time-per-iteration against the mesh size $h$}
\label{fig2D1:summary}
\end{figure}
In the next two examples, we carry out three different experiments for KVM to further highlight the sensitivity of $J$ in comparison with $J_{KV}$ in the case of a bit more complex geometries for the fixed boundary with concave regions.
We again test the effect of the step size parameter $\mu$ in the approximation process.
For CCBM, we fix $\mu$ to $1.0$.
The test experiments are as follows: test (KV)$_{a}$: $\mu = 1.0$; test (KV)$_{b}$: $\mu = 0.5$; and test (KV)$_{c}$: $\mu = 0.25$, and we look at the problems where $\lambda = -10, -9, \ldots, -1$.
Also, we choose $C(\boldsymbol{0},1.25)$ as the initial profile $\Sigma^{0}$ of the free boundary, and we stop the algorithm after completing $100$ iterations. 
\begin{example}[An \texttt{L}-shape fixed boundary]\label{example2d2} Let us consider an \texttt{L}-shape geometry for $\Gamma$ given by the boundary of the domain $D=(-0.25,0.25)^{2}\setminus[0,0.25]^{2}$, i.e., we set $\Gamma=\Gamma_{\texttt{L}}:=\partial D$.
A comparison between the computed free boundaries using the KVM and CCBM for each test experiments are shown in Figure \ref{fig2D} (first row).
In those figures, the outermost boundary corresponds to $\lambda = -1$ while the innermost (exterior) boundary corresponds to $\lambda = -10$.
The histories of $d_{\text{H}}(\Sigma^{k}, \Sigma_{100})$-values are shown in Figure \ref{fig2D:hd_Lshape} while the histories of cost values are plotted in Figure \ref{fig2D:J_Lshape}.
It appears based on the results that CCBM almost has the same convergence behavior with KVM in the case of larger values for $\lambda$, and a bit faster for the case of smaller values of $\lambda$.
It must be evident, however, that KVM tends to converge to a stationary shape in a fewer number of iterations than CCBM, but to a less accurate geometry for the free boundary.
In fact, KVM converges prematurely in some instances; see, for example, Figure \ref{fig2D_illustration} where the evolutions of the shapes (plotted at every ten iterates with $\mu = 2.0$) are illustrated in the case $\lambda = -5$.
Nevertheless, the two methods with small step sizes provide almost identical optimal solutions to the minimization problem.
Furthermore, we notice from the histories of Hausdorff distances and cost values that $J$ is more sensitive to large variations than $J_{KV}$.
Based on these observations, we can say that CCBM has some advantages over KVM in terms of computational performance, specifically with respect to overall computing-time-per-iteration (see left plot in Figure \ref{fig2D:timeGraphs}) and accuracy when taking large deformations of the domain.
However, as evident in Figures \ref{fig2D:hd_Lshape}--\ref{fig2D:J_Lshape}, KVM converges in fewer iterations than CCBM.
\end{example}
\begin{example}[A ribbon shape fixed boundary]\label{example2d3} Next, we consider a ribbon shape fixed boundary similar to the one examined in \cite{EpplerHarbrecht2006} which is parametrized as $\Gamma_{\texttt{R}}:=\{(0.45\cos \theta, 0.3\sin \theta(1.25 + \cos 2\theta))^\top \mid 0 \leqslant \theta \leqslant 2\pi\}$.
The results of the test experiments are depicted in Figure \ref{fig2D} (second row).
Meanwhile, the histories of the Hausdorff distance values $d_{\text{H}}(\Sigma^{k}, \Sigma_{100})$ and of the cost values are shown in Figure \ref{fig2D:hd_Rshape} and Figure \ref{fig2D:J_Rshape}, respectively.
In these computational results, we observe similar findings to the previous example in that $J_{KV}$ is less sensitive to large variations than $J$.
Moreover, in some cases, KVM  tends to overshoot the approximate optimal shape unlike CCBM.
Even so, as in the case of the previous test experiments, the present algorithm with the KVM uses fewer number of iterations to converge compared to CCBM.
\end{example}
\begin{figure}[htp!]
\centering
\resizebox{0.32\linewidth}{!}{\includegraphics{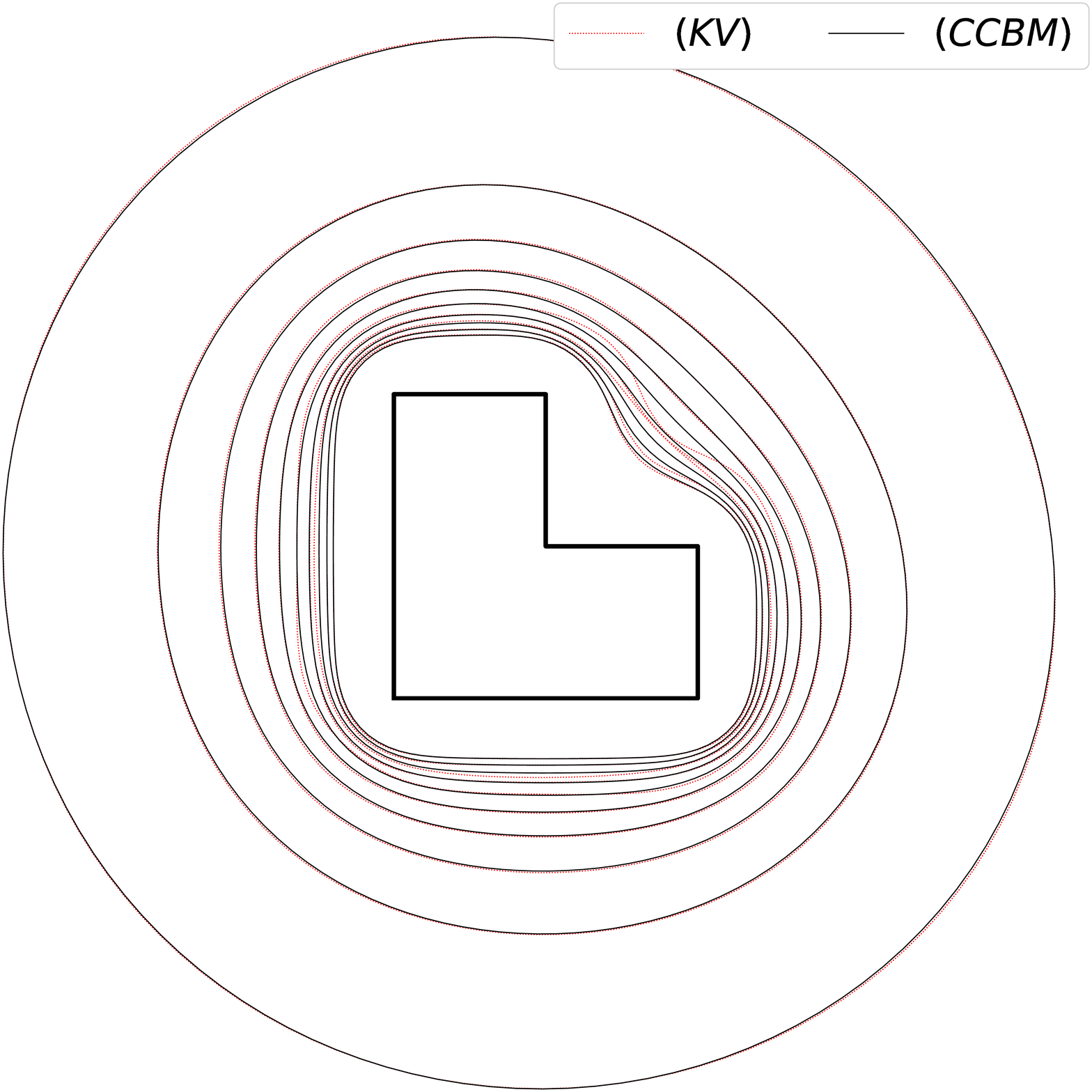}}
\resizebox{0.32\linewidth}{!}{\includegraphics{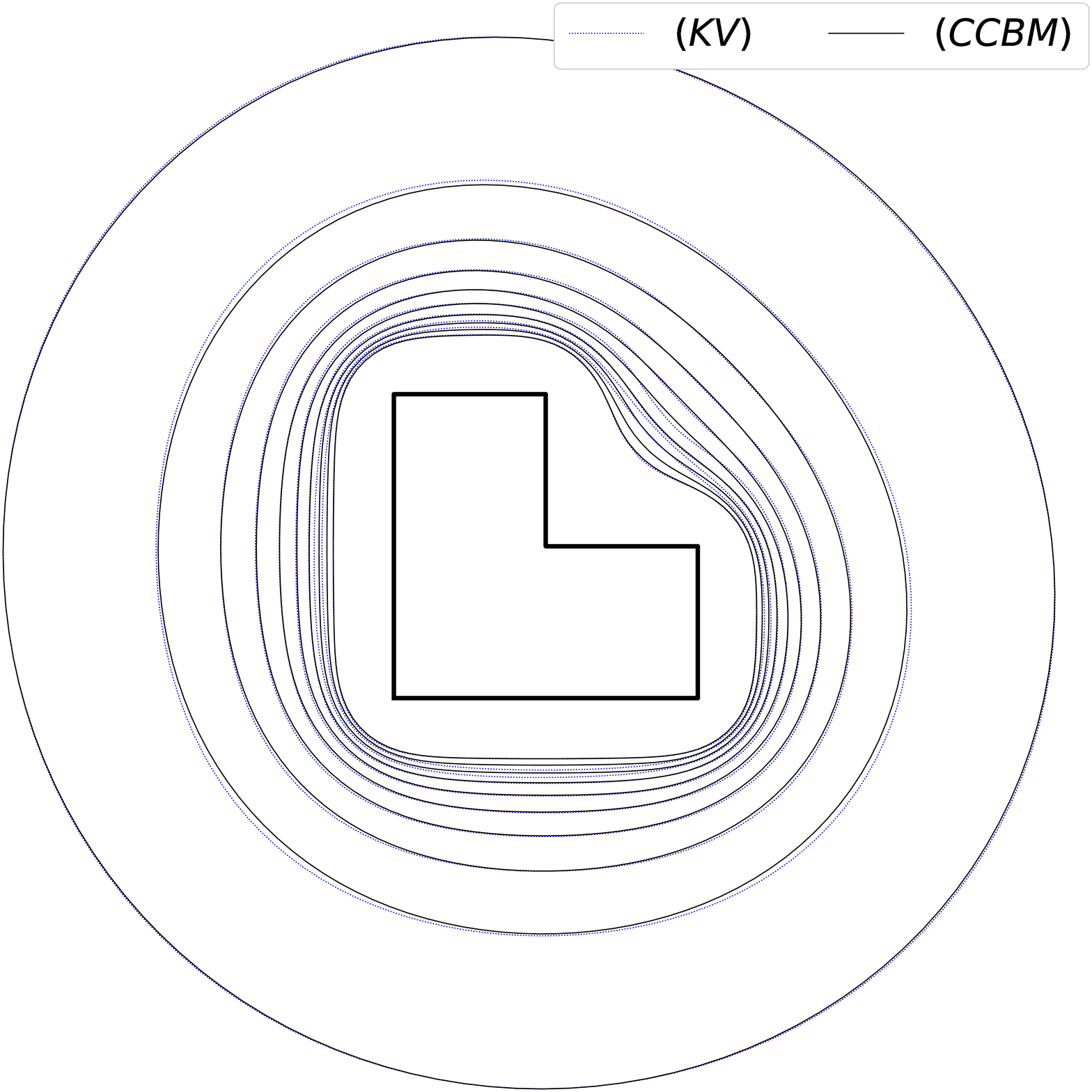}}
\resizebox{0.32\linewidth}{!}{\includegraphics{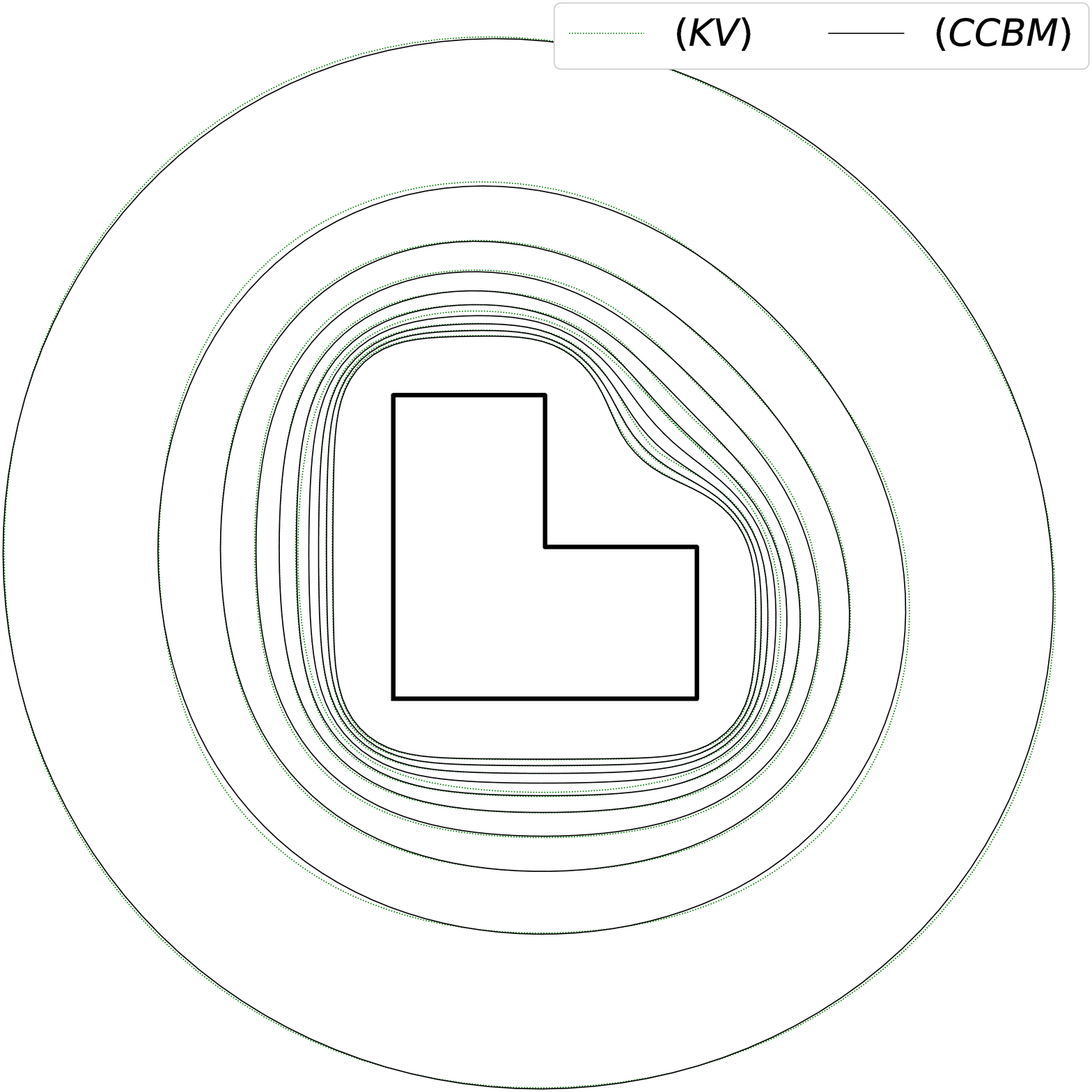}}\\
\resizebox{0.32\linewidth}{!}{\includegraphics{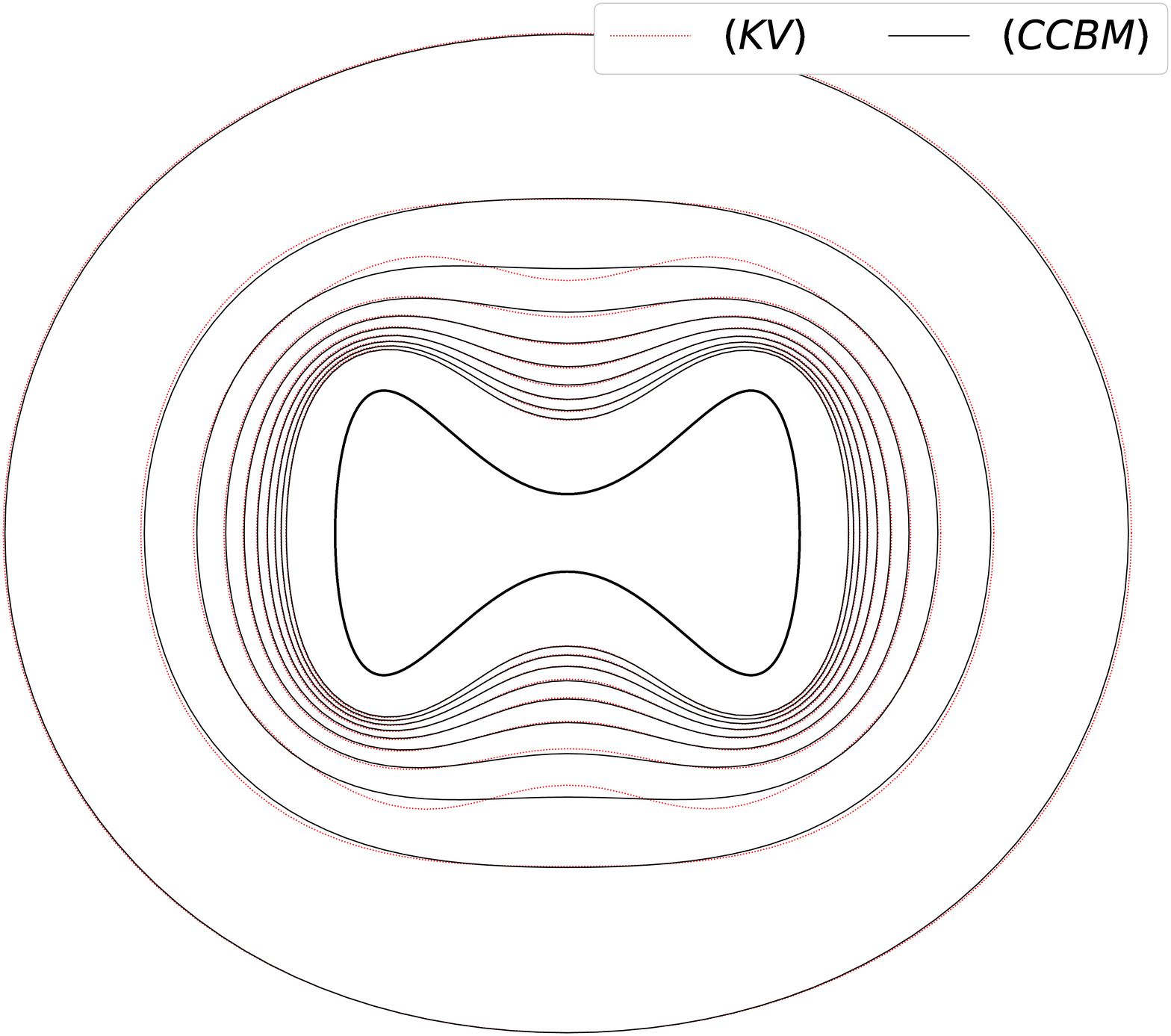}}
\resizebox{0.32\linewidth}{!}{\includegraphics{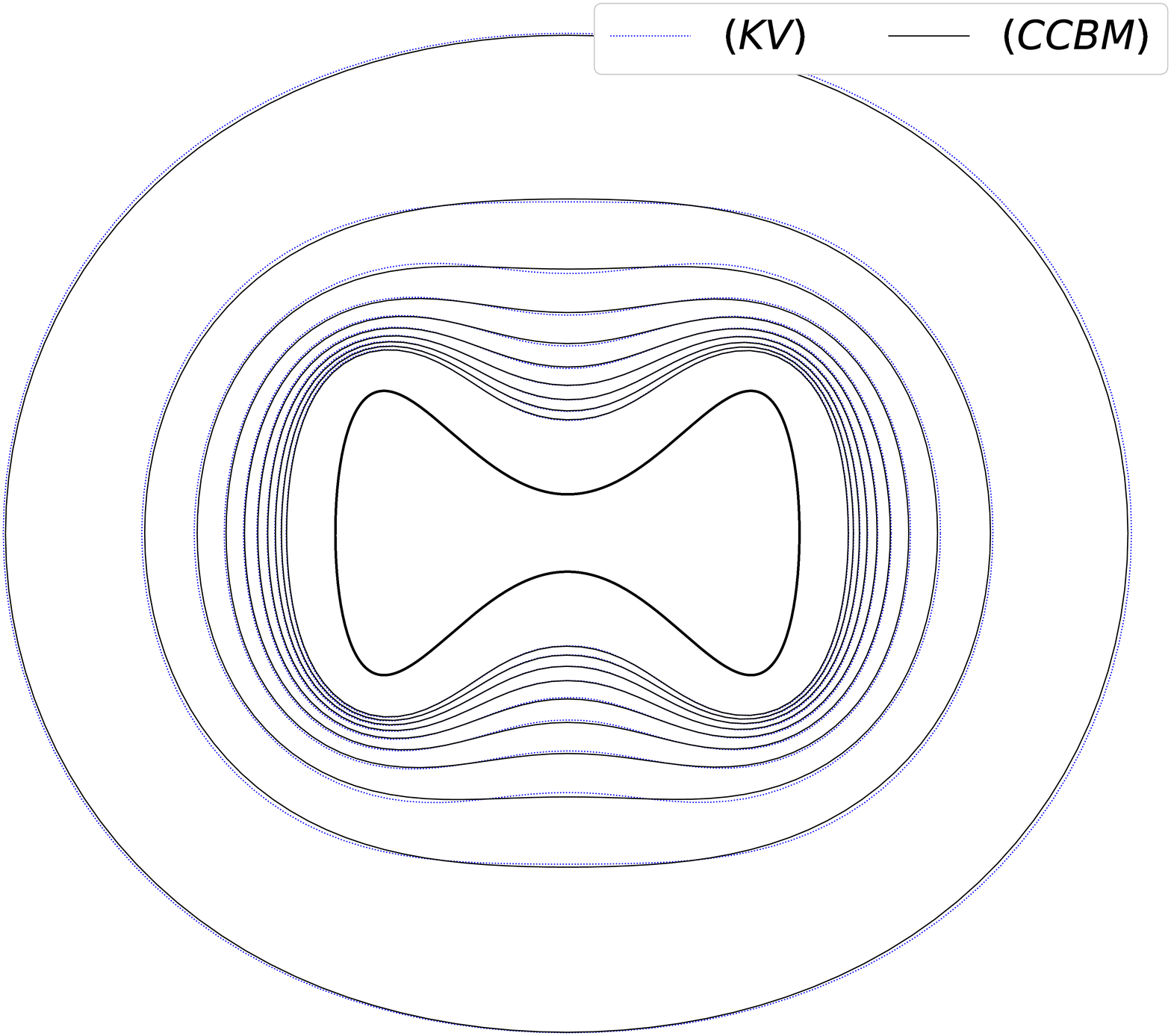}}
\resizebox{0.32\linewidth}{!}{\includegraphics{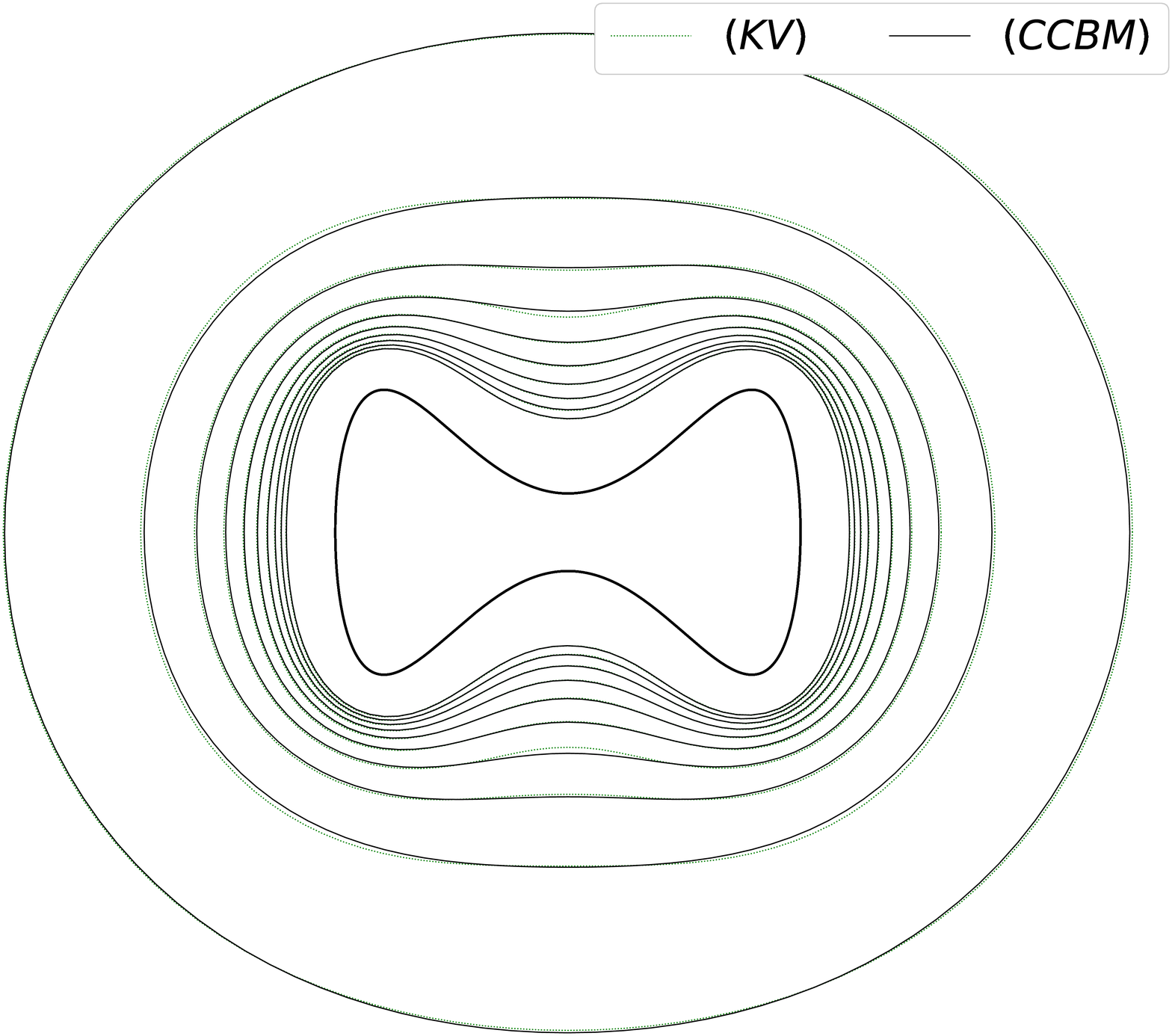}}
\caption{Computational results for Example \ref{example2d2} (first row) and Example \ref{example2d3} (second row)}
\label{fig2D}
\end{figure}

%
%
\begin{figure}[htp!]
\centering
\resizebox{0.19\linewidth}{!}{\includegraphics{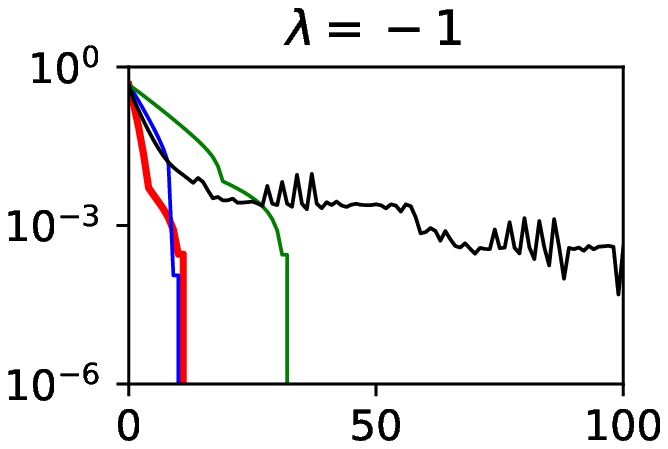}}
\resizebox{0.19\linewidth}{!}{\includegraphics{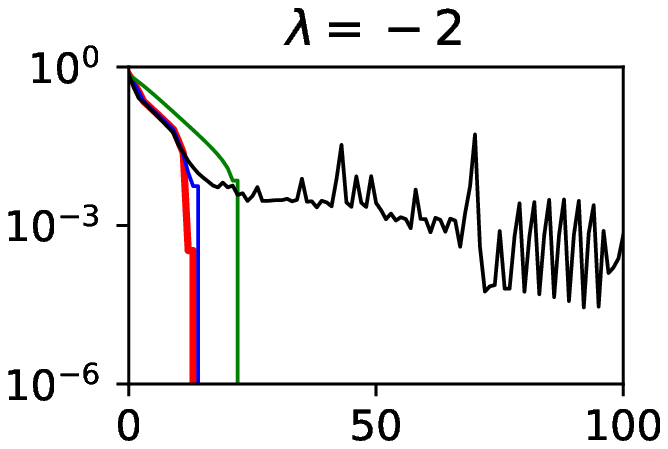}}
\resizebox{0.19\linewidth}{!}{\includegraphics{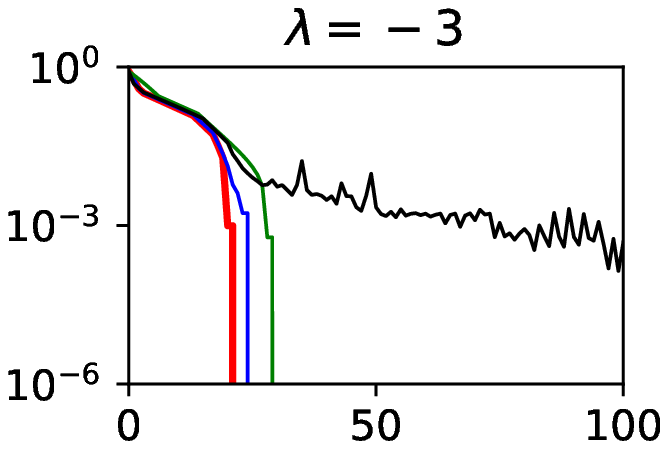}}
\resizebox{0.19\linewidth}{!}{\includegraphics{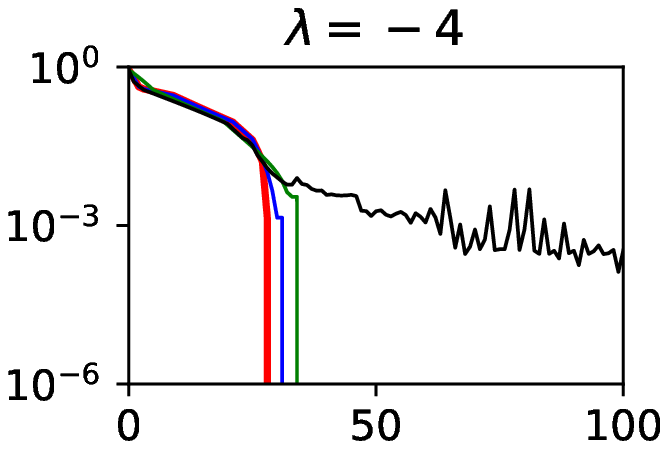}}
\resizebox{0.19\linewidth}{!}{\includegraphics{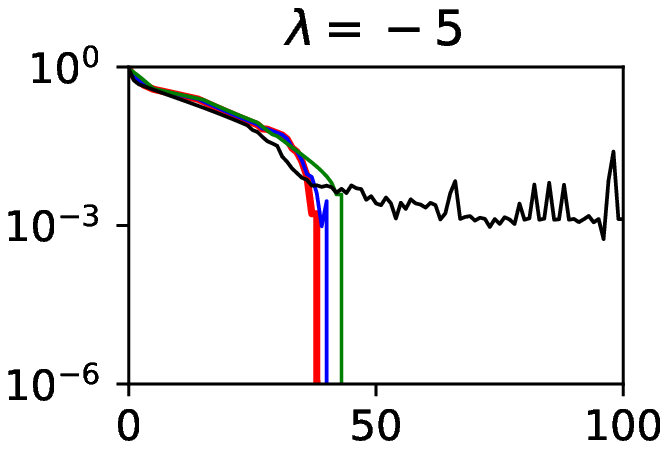}}\\
\resizebox{0.19\linewidth}{!}{\includegraphics{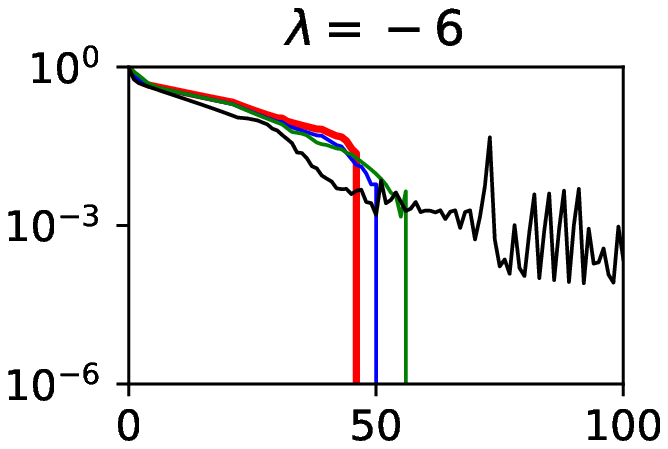}}
\resizebox{0.19\linewidth}{!}{\includegraphics{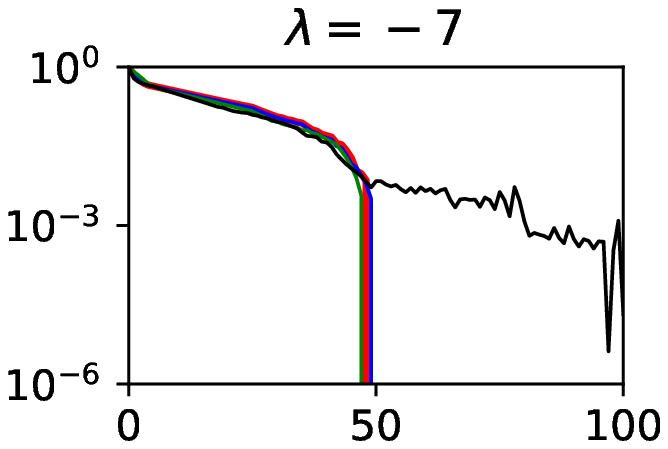}}
\resizebox{0.19\linewidth}{!}{\includegraphics{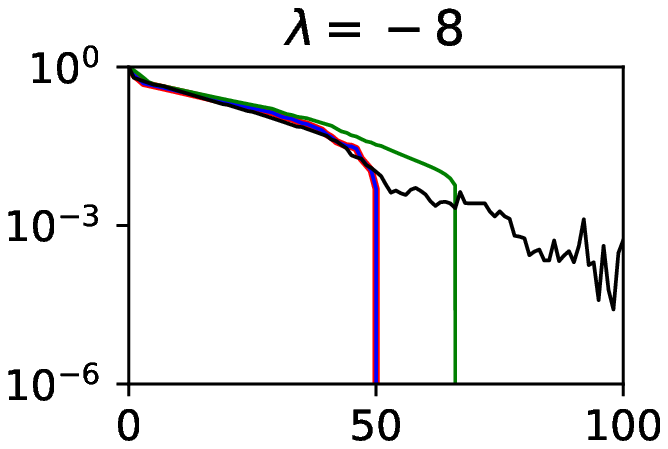}}
\resizebox{0.19\linewidth}{!}{\includegraphics{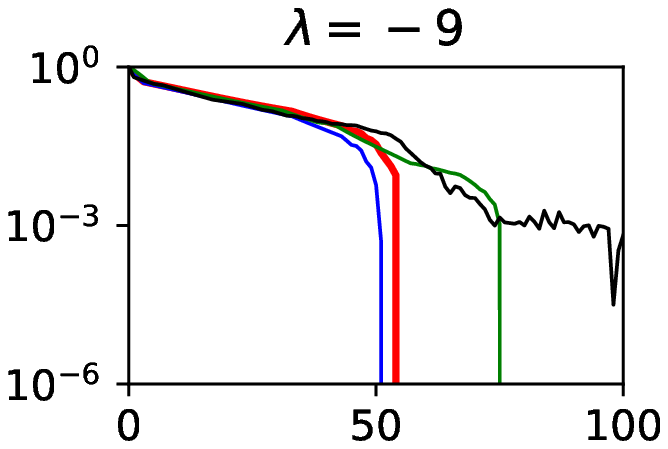}}
\resizebox{0.19\linewidth}{!}{\includegraphics{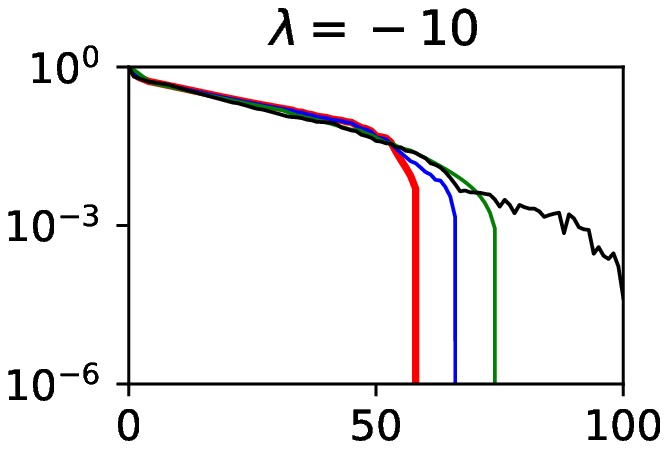}}
\caption{iteration $k$ vs distance $d_{H}(\Sigma^{k},\Sigma_{100})$ \textnormal{(legend: {\color{red}---} (KV)$_{a}$\ {\color{blue}---} (KV)$_{b}$\ {\color{green}---} (KV)$_{c}$\ {\color{black}---}\ (CCBM))}}
\label{fig2D:hd_Lshape}
\end{figure}

%
%
%
\begin{figure}[htp!]
\centering
\resizebox{0.09\linewidth}{!}{\includegraphics{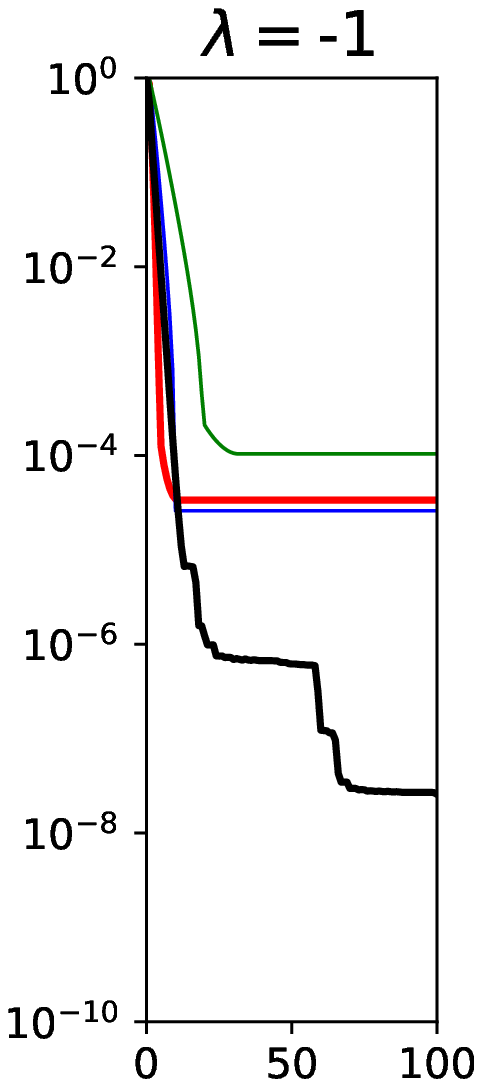}}
\resizebox{0.09\linewidth}{!}{\includegraphics{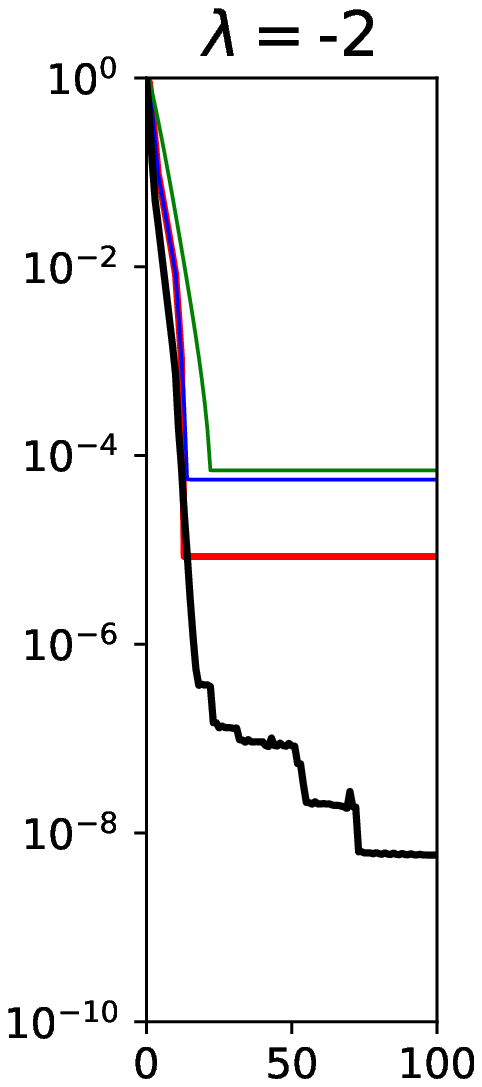}}
\resizebox{0.09\linewidth}{!}{\includegraphics{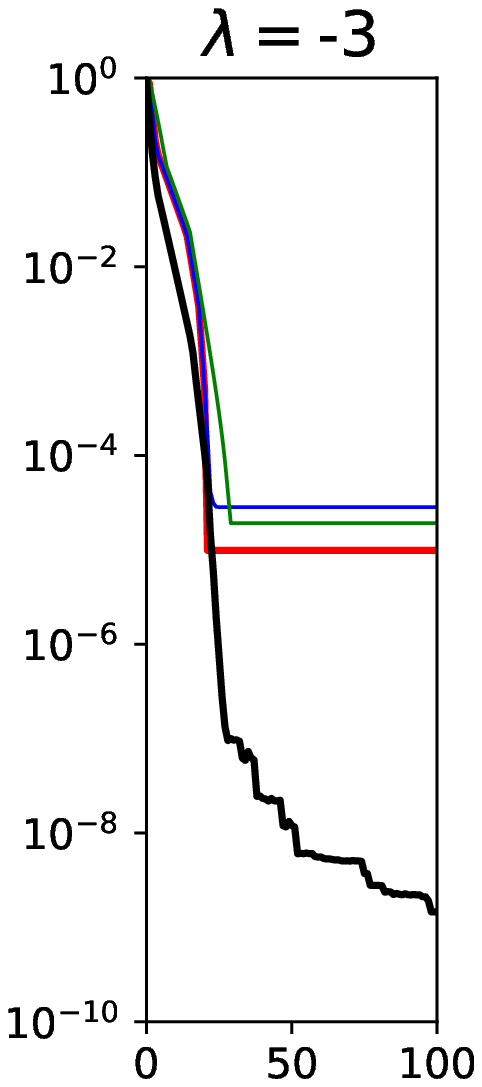}}
\resizebox{0.09\linewidth}{!}{\includegraphics{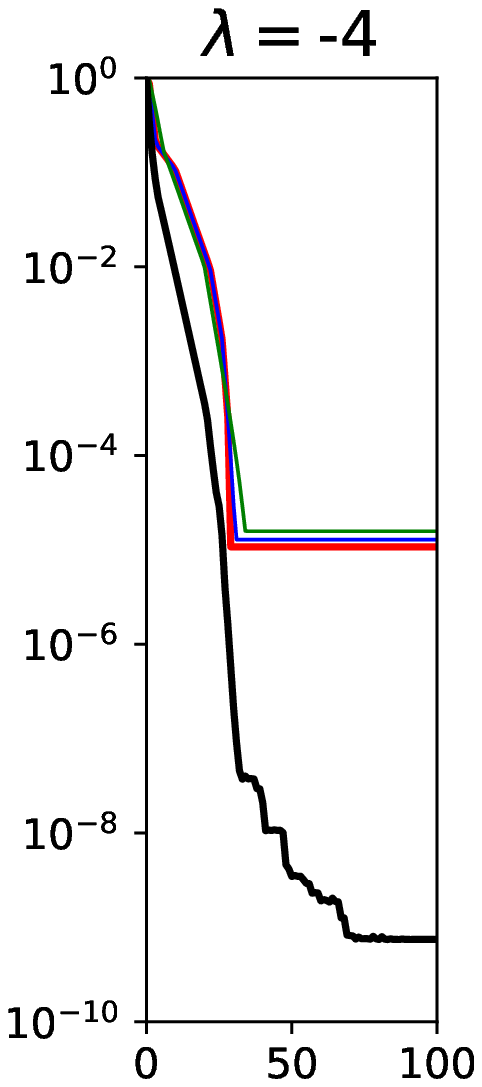}}
\resizebox{0.09\linewidth}{!}{\includegraphics{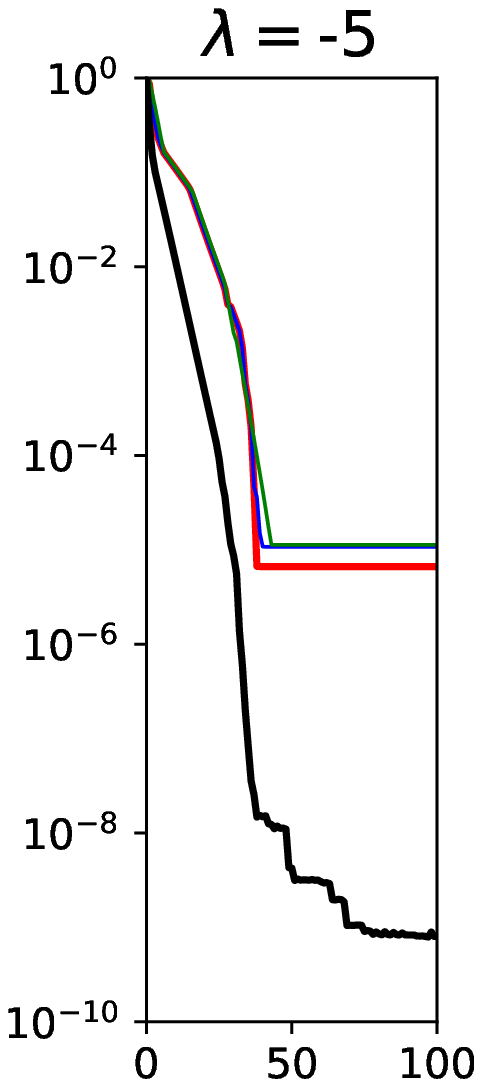}}
\resizebox{0.09\linewidth}{!}{\includegraphics{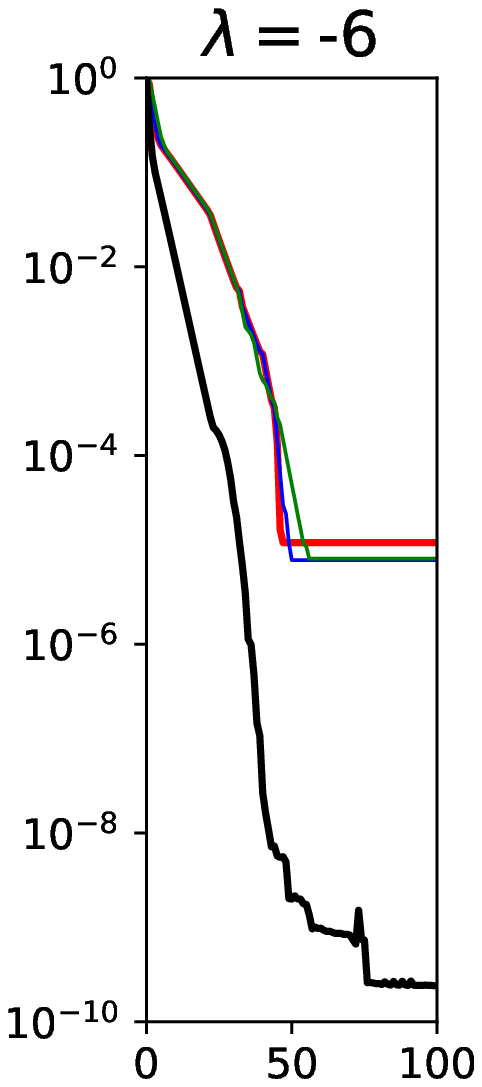}}
\resizebox{0.09\linewidth}{!}{\includegraphics{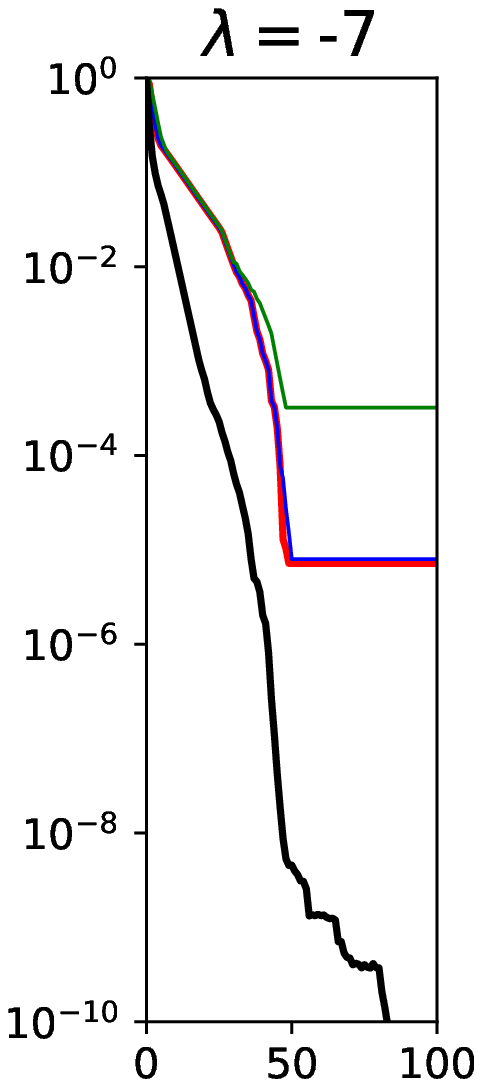}}
\resizebox{0.09\linewidth}{!}{\includegraphics{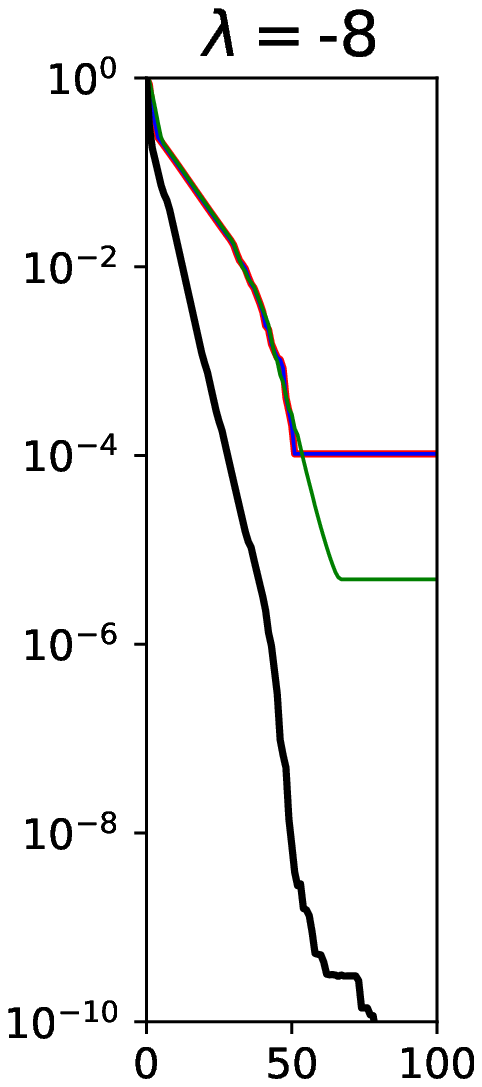}}
\resizebox{0.09\linewidth}{!}{\includegraphics{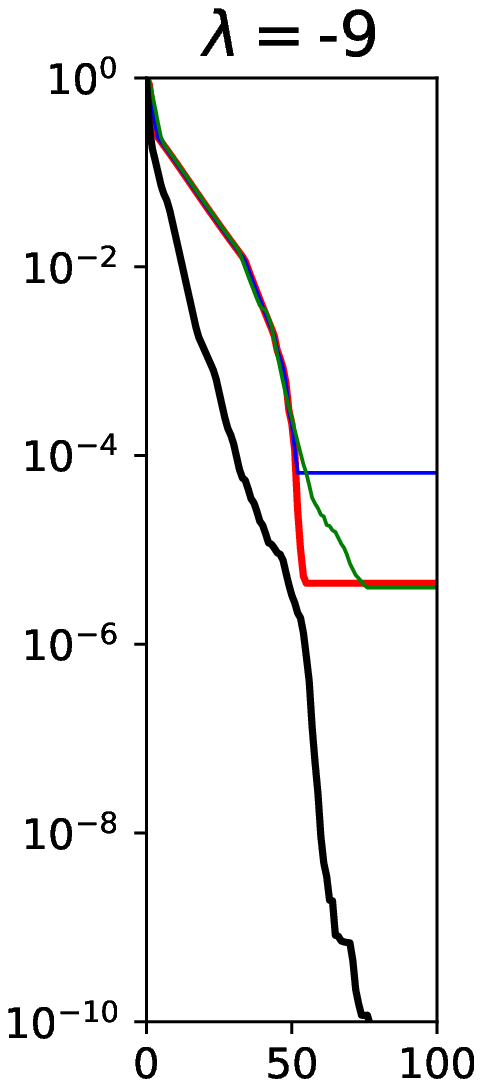}}
\resizebox{0.09\linewidth}{!}{\includegraphics{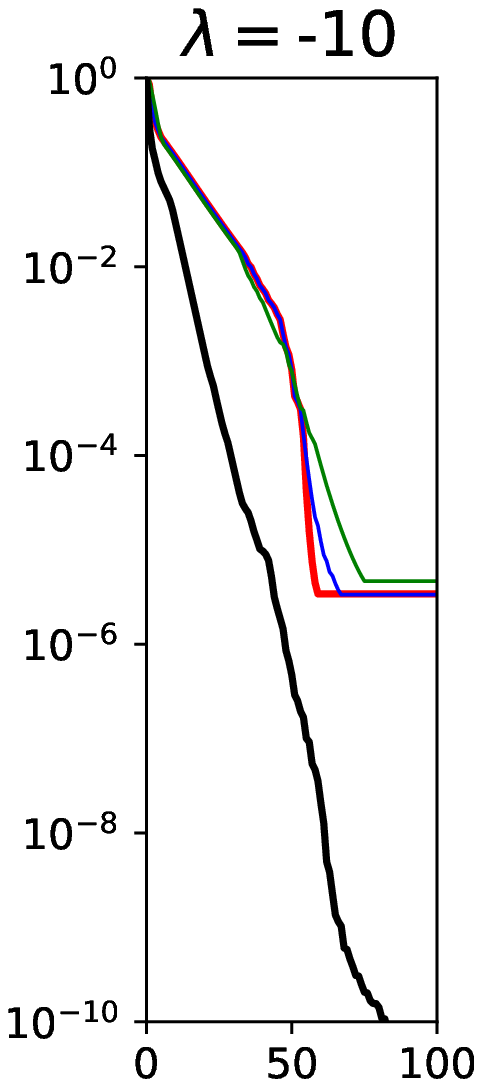}}
\caption{iteration $k$ vs cost values \textnormal{(legend: {\color{red}---} (KV)$_{a}$\ {\color{blue}---} (KV)$_{b}$\ {\color{green}---} (KV)$_{c}$\ {\color{black}---}\ (CCBM))}}
\label{fig2D:J_Lshape}
\end{figure}
%
%
\begin{figure}[htp!]
\centering
\resizebox{0.19\linewidth}{!}{\includegraphics{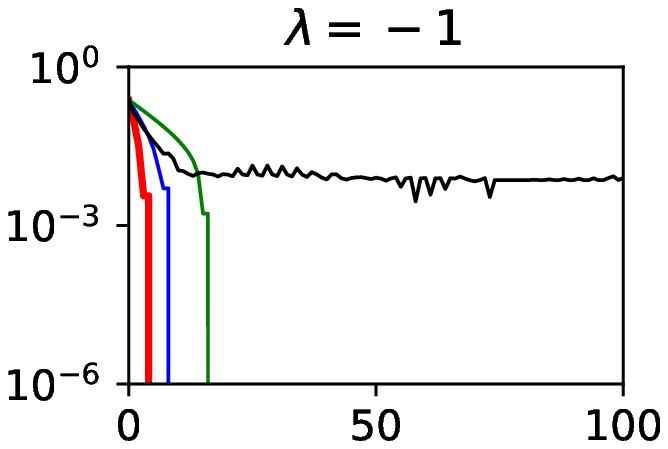}}
\resizebox{0.19\linewidth}{!}{\includegraphics{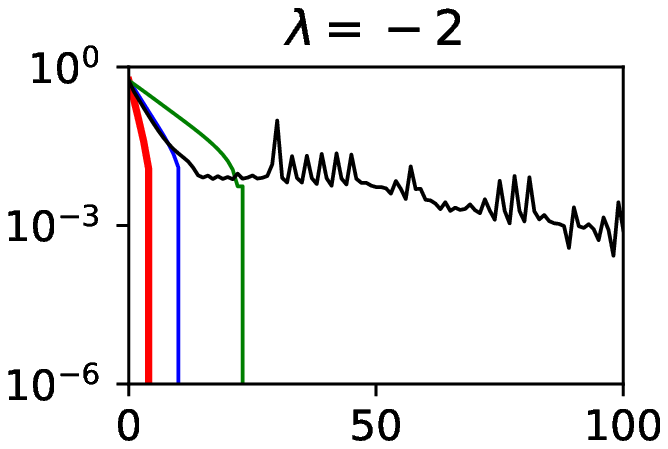}}
\resizebox{0.19\linewidth}{!}{\includegraphics{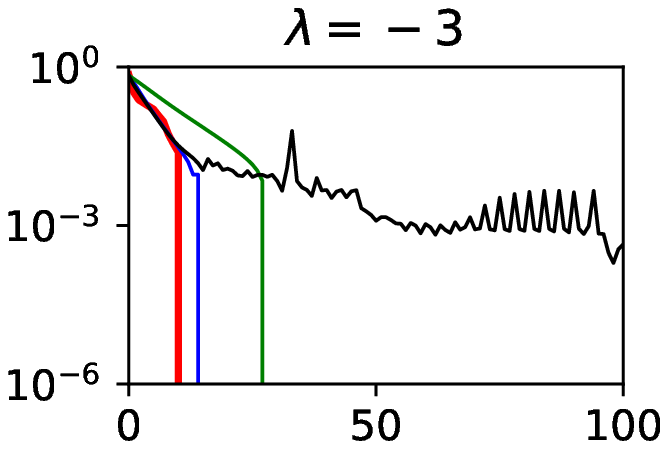}}
\resizebox{0.19\linewidth}{!}{\includegraphics{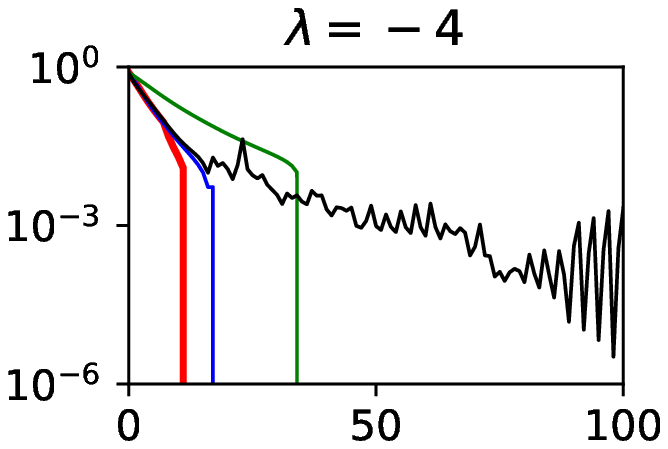}}
\resizebox{0.19\linewidth}{!}{\includegraphics{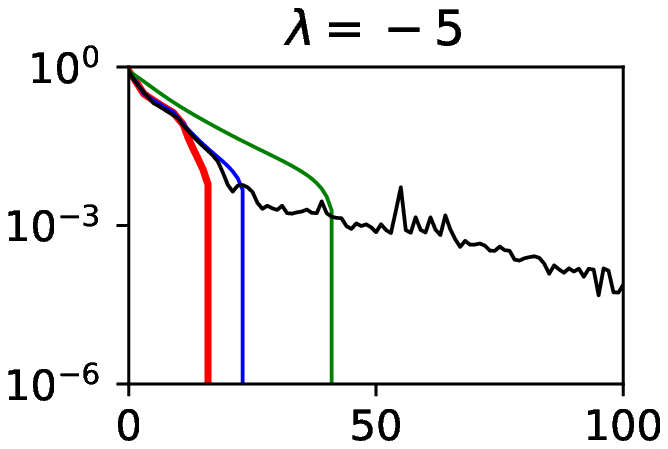}}\\
\resizebox{0.19\linewidth}{!}{\includegraphics{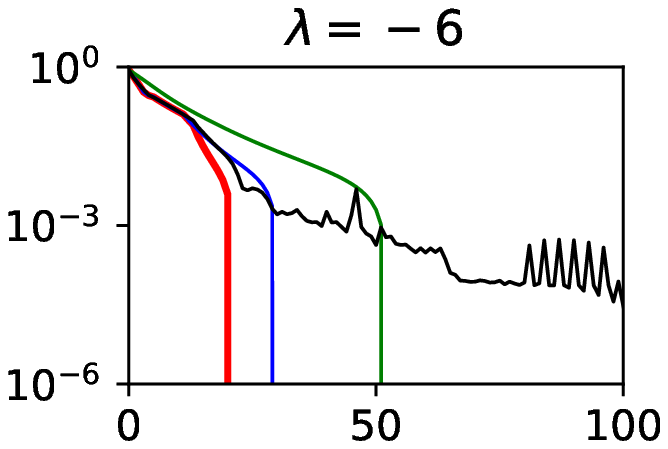}}
\resizebox{0.19\linewidth}{!}{\includegraphics{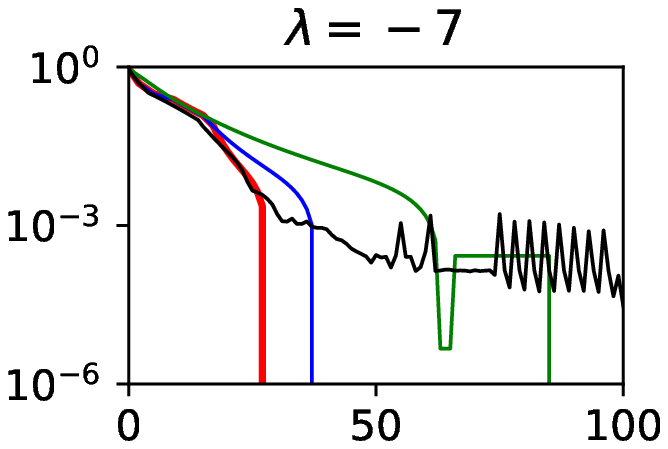}}
\resizebox{0.19\linewidth}{!}{\includegraphics{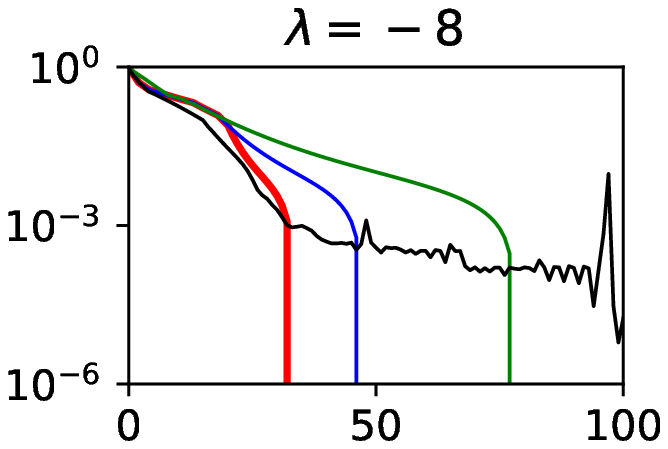}}
\resizebox{0.19\linewidth}{!}{\includegraphics{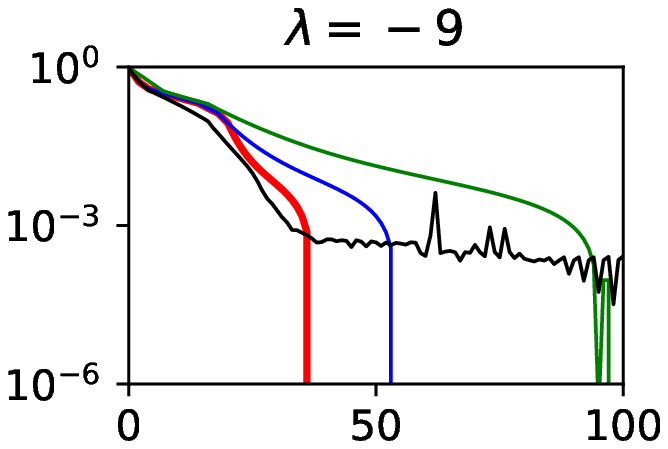}}
\resizebox{0.19\linewidth}{!}{\includegraphics{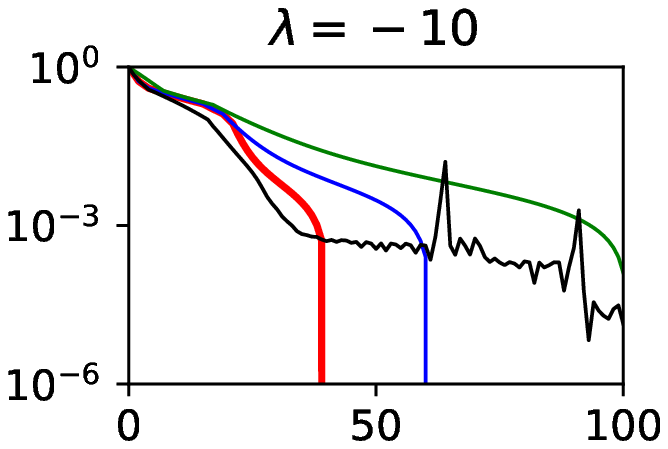}}
\caption{iteration $k$ vs distance $d_{H}(\Sigma^{k},\Sigma_{100})$ \textnormal{(legend: {\color{red}---} (KV)$_{a}$\ {\color{blue}---} (KV)$_{b}$\ {\color{green}---} (KV)$_{c}$\ {\color{black}---}\ (CCBM))}}
\label{fig2D:hd_Rshape}
\end{figure}
%
%
\begin{figure}[htp!]
\centering
\resizebox{0.09\linewidth}{!}{\includegraphics{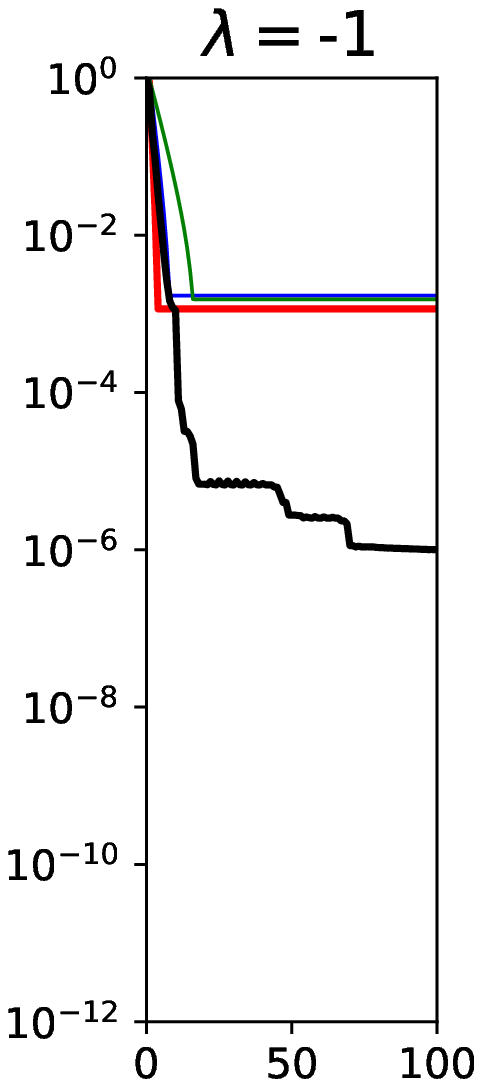}}
\resizebox{0.09\linewidth}{!}{\includegraphics{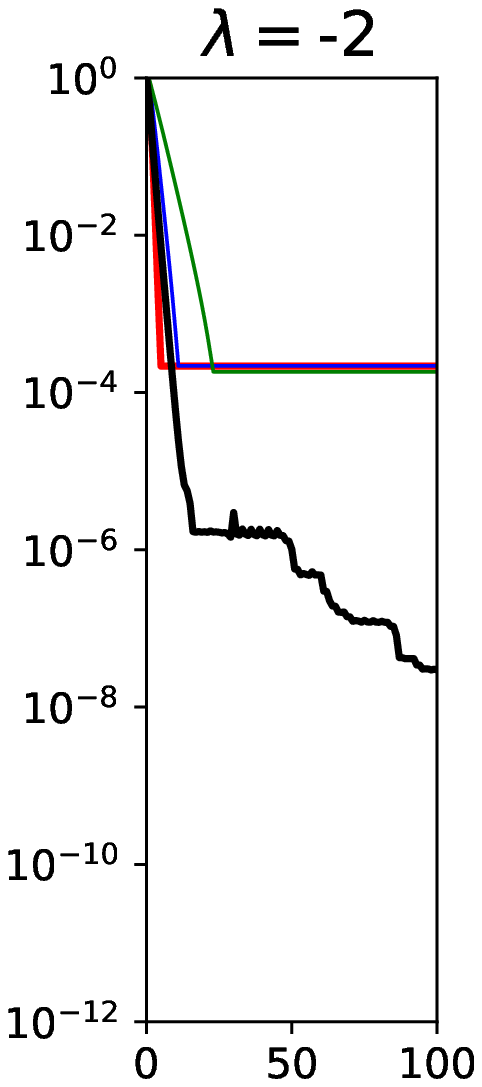}}
\resizebox{0.09\linewidth}{!}{\includegraphics{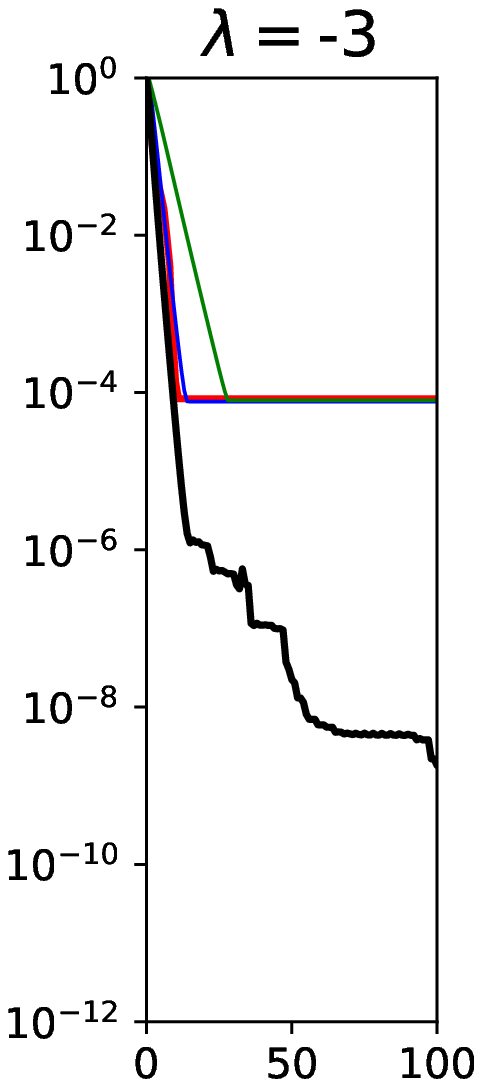}}
\resizebox{0.09\linewidth}{!}{\includegraphics{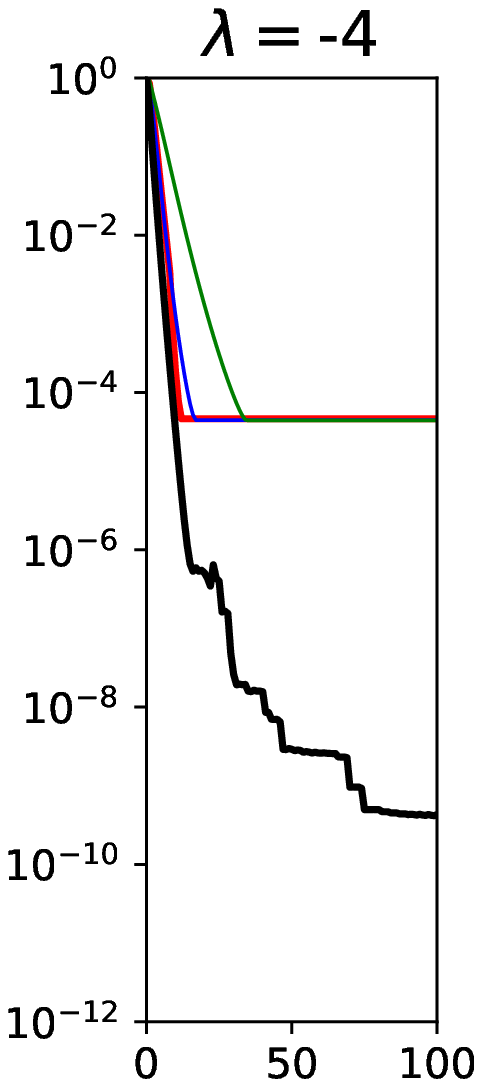}}
\resizebox{0.09\linewidth}{!}{\includegraphics{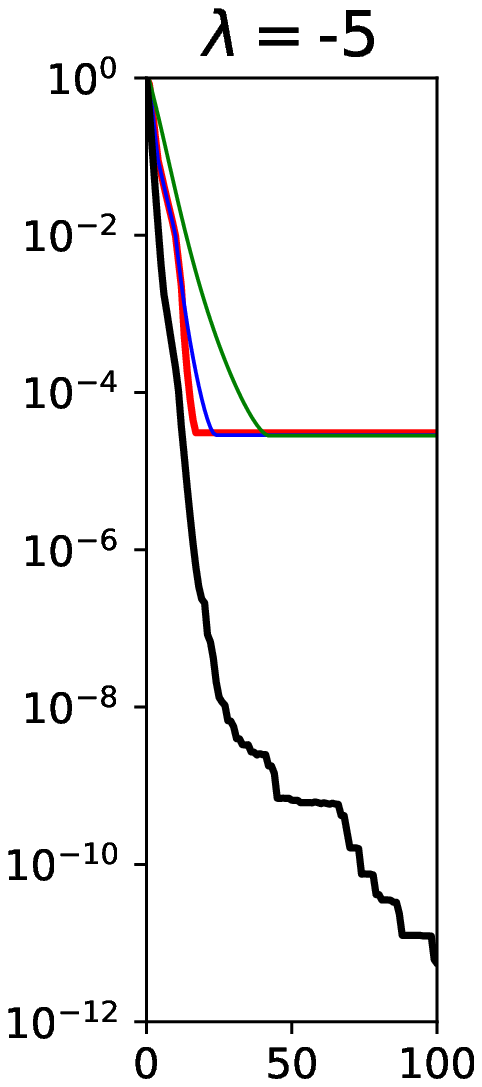}}
\resizebox{0.09\linewidth}{!}{\includegraphics{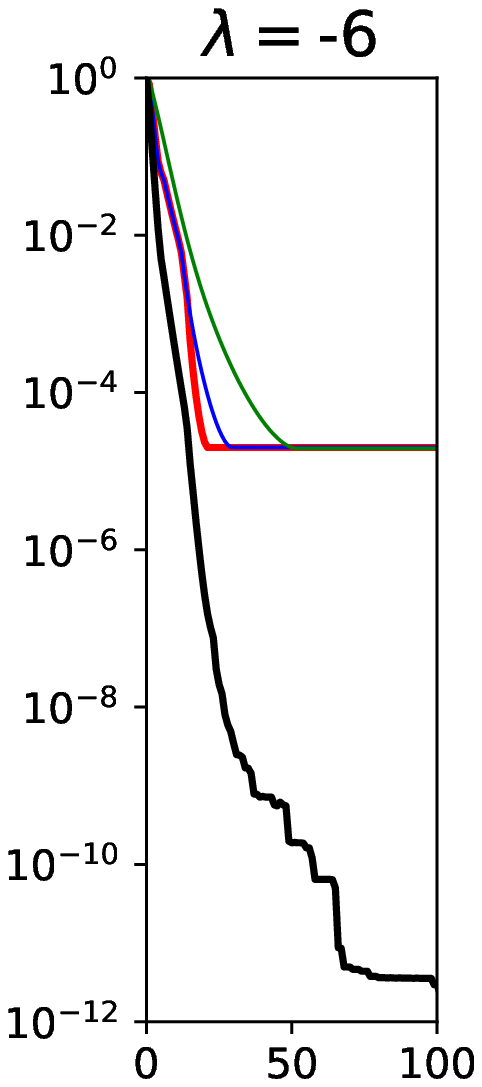}}
\resizebox{0.09\linewidth}{!}{\includegraphics{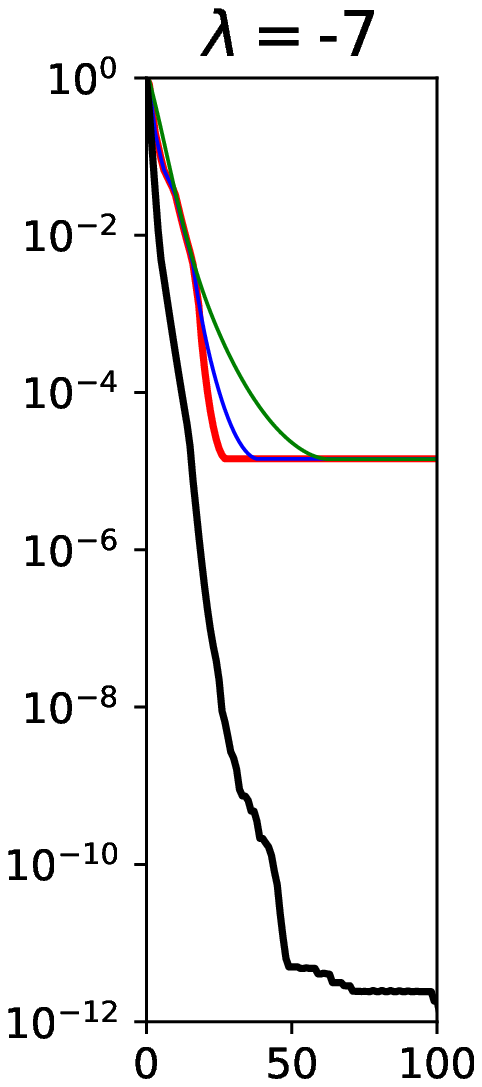}}
\resizebox{0.09\linewidth}{!}{\includegraphics{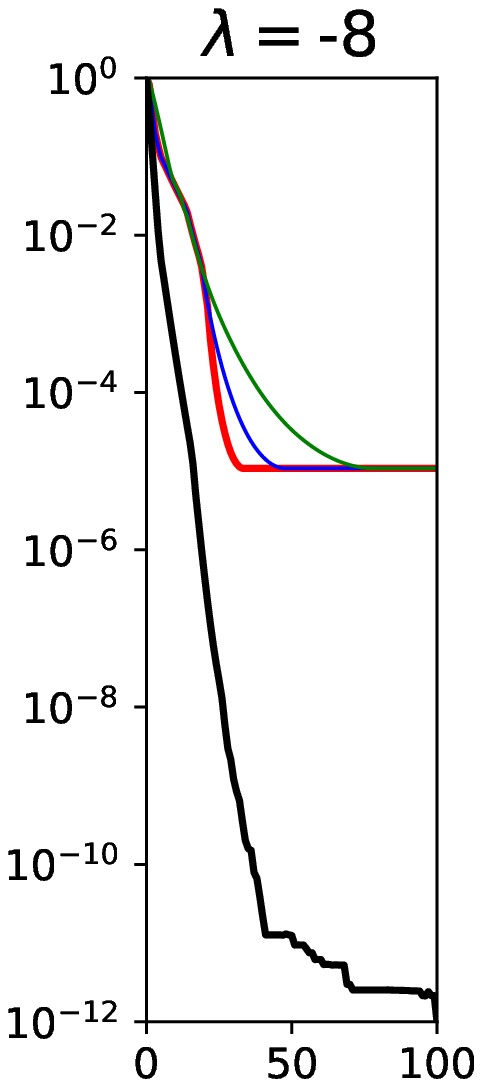}}
\resizebox{0.09\linewidth}{!}{\includegraphics{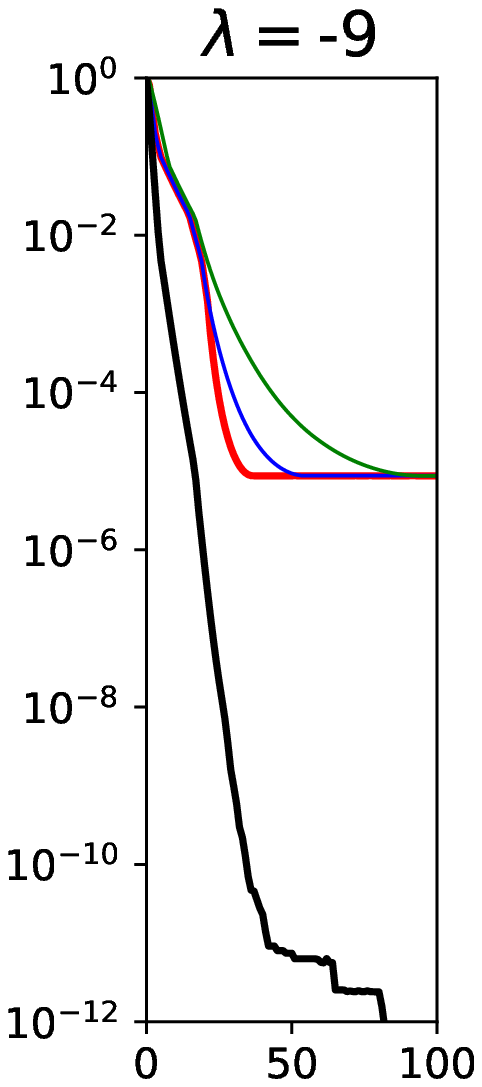}}
\resizebox{0.09\linewidth}{!}{\includegraphics{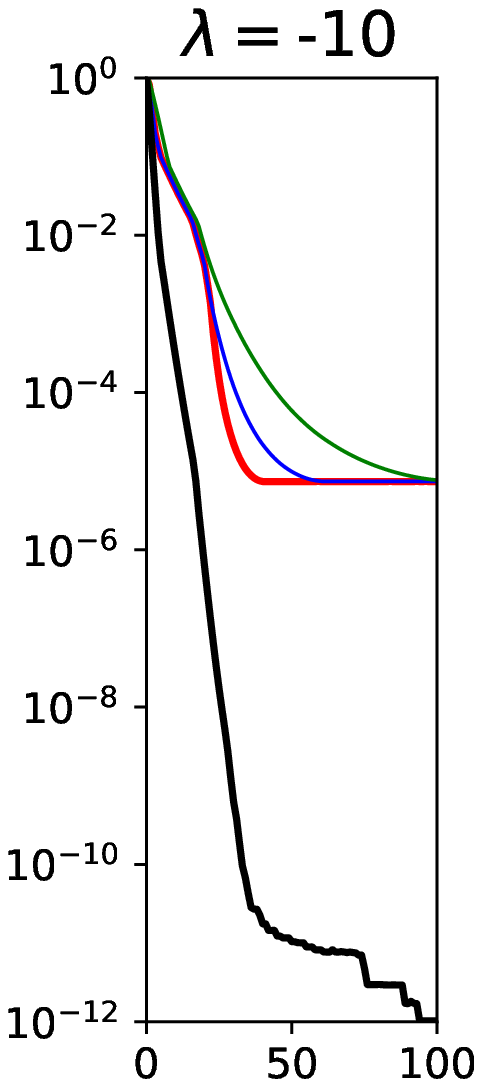}}
\caption{iteration $k$ vs cost values \textnormal{(legend: {\color{red}---} (KV)$_{a}$\ {\color{blue}---} (KV)$_{b}$\ {\color{green}---} (KV)$_{c}$\ {\color{black}---}\ (CCBM))}}
\label{fig2D:J_Rshape}
\end{figure}

\begin{figure}[htp!]
\centering
\resizebox{0.32\linewidth}{!}{\includegraphics{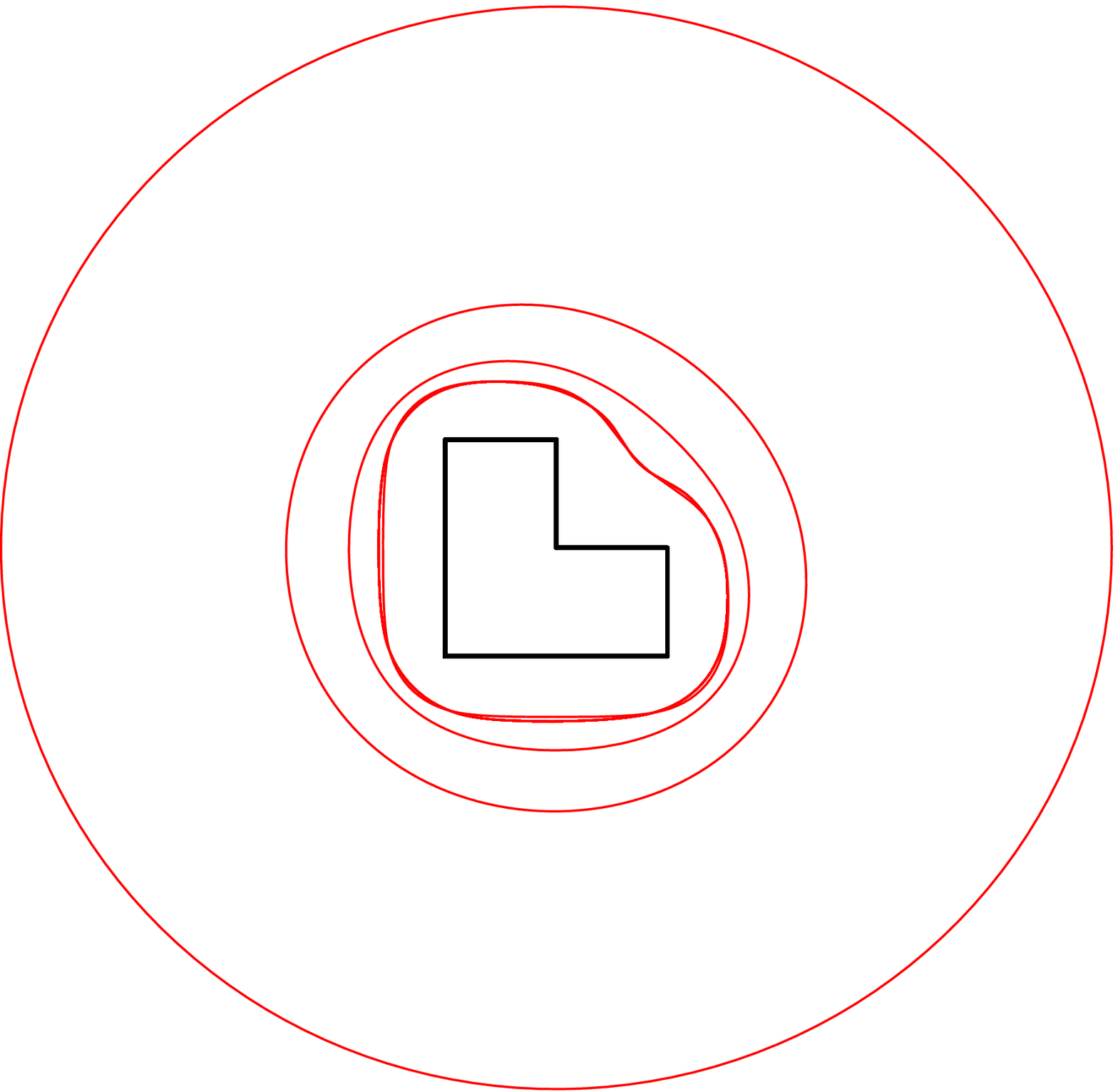}}\hfill
\resizebox{0.32\linewidth}{!}{\includegraphics{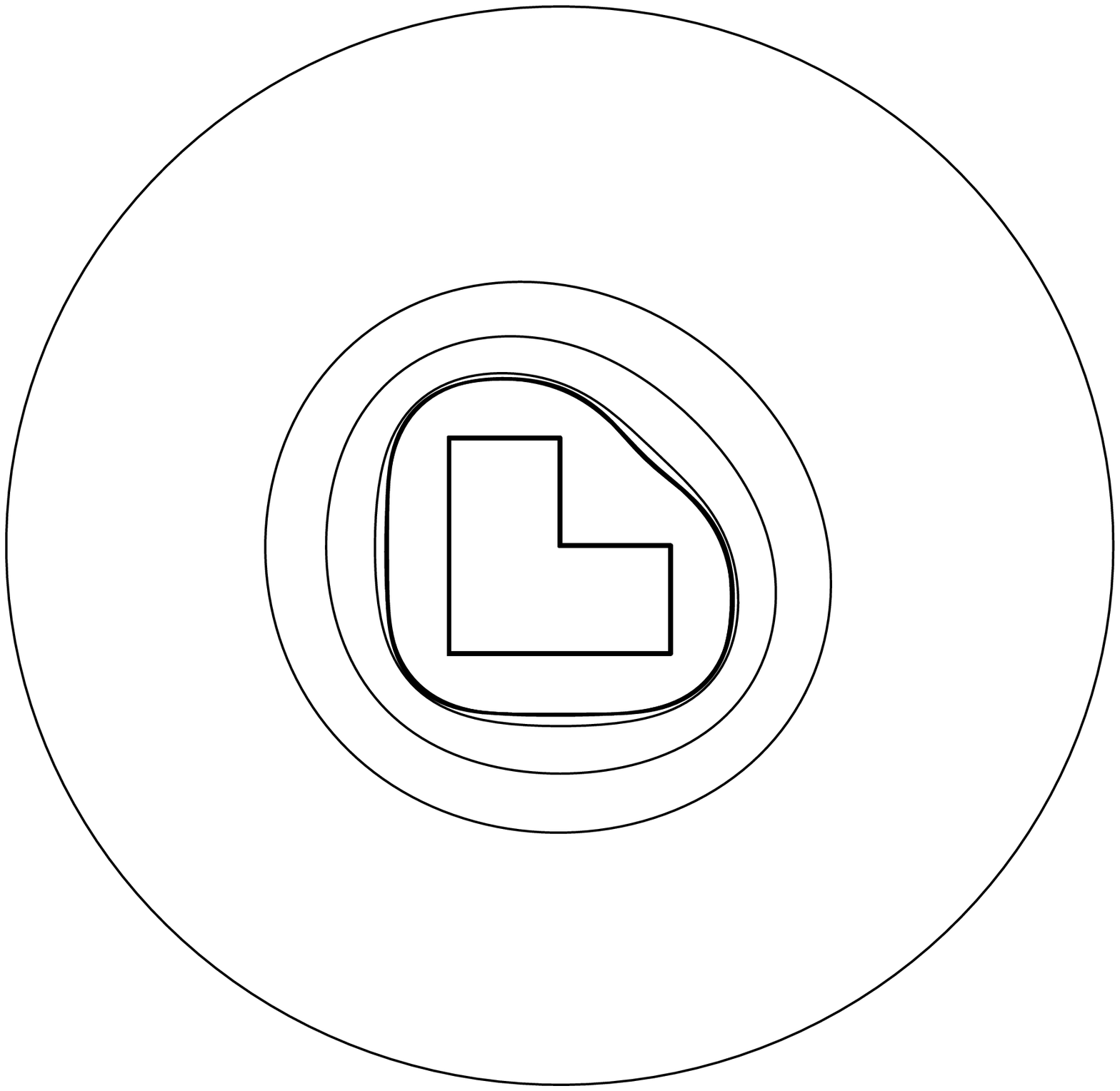}}\hfill
\resizebox{0.31\linewidth}{!}{\includegraphics{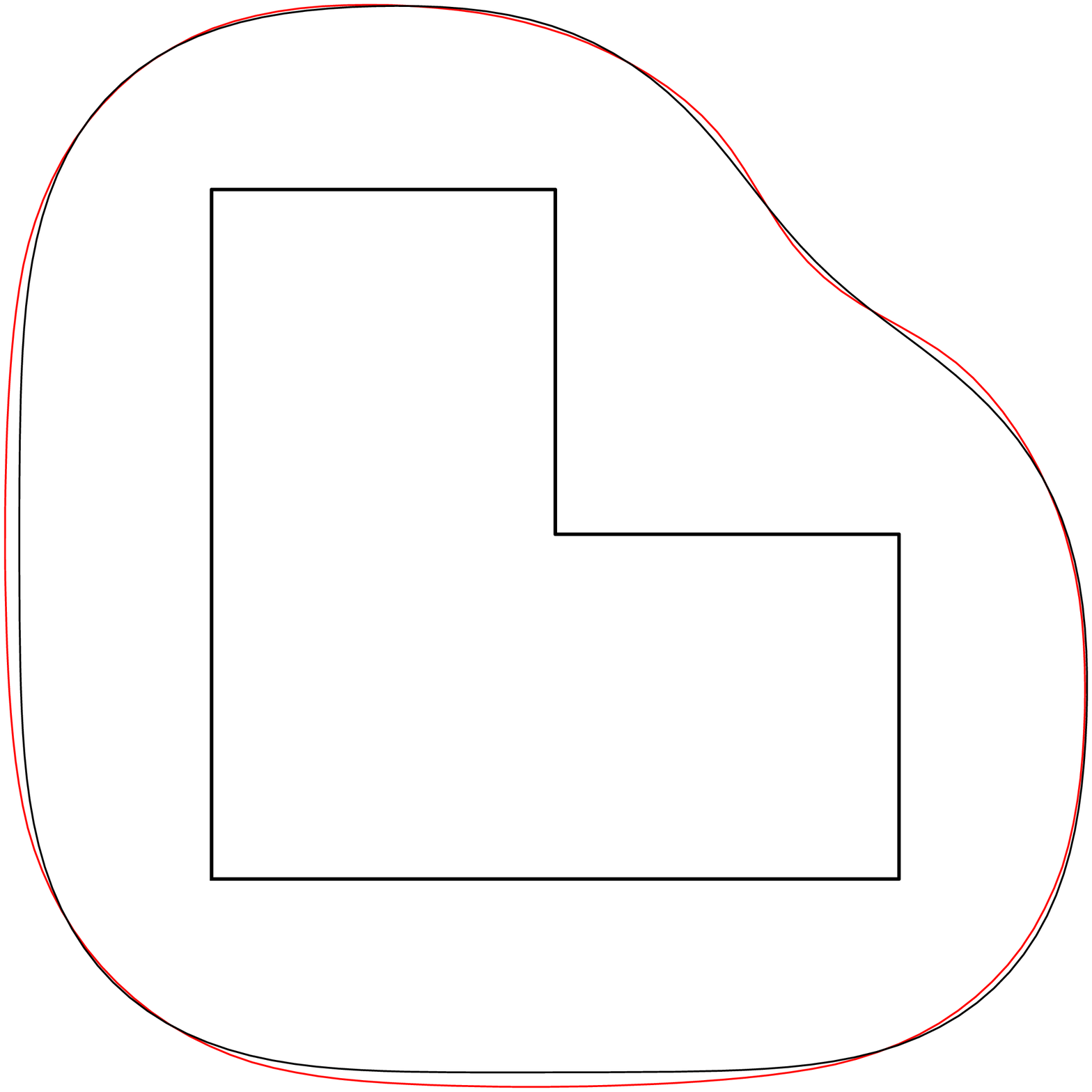}}
\caption{Evolution of shapes (plotted at every ten iterates and were obtained with step size parameter value $\mu = 2.0$) for Example \ref{example2d2} when $\lambda = -7$ using KVM (left-most plot) and CCBM (middle plot), and a direct comparison with the computed optimal shapes (right-most plot)}
\label{fig2D_illustration}
\end{figure}

\begin{figure}[htp!]
\centering
\resizebox{0.48\linewidth}{!}{\includegraphics{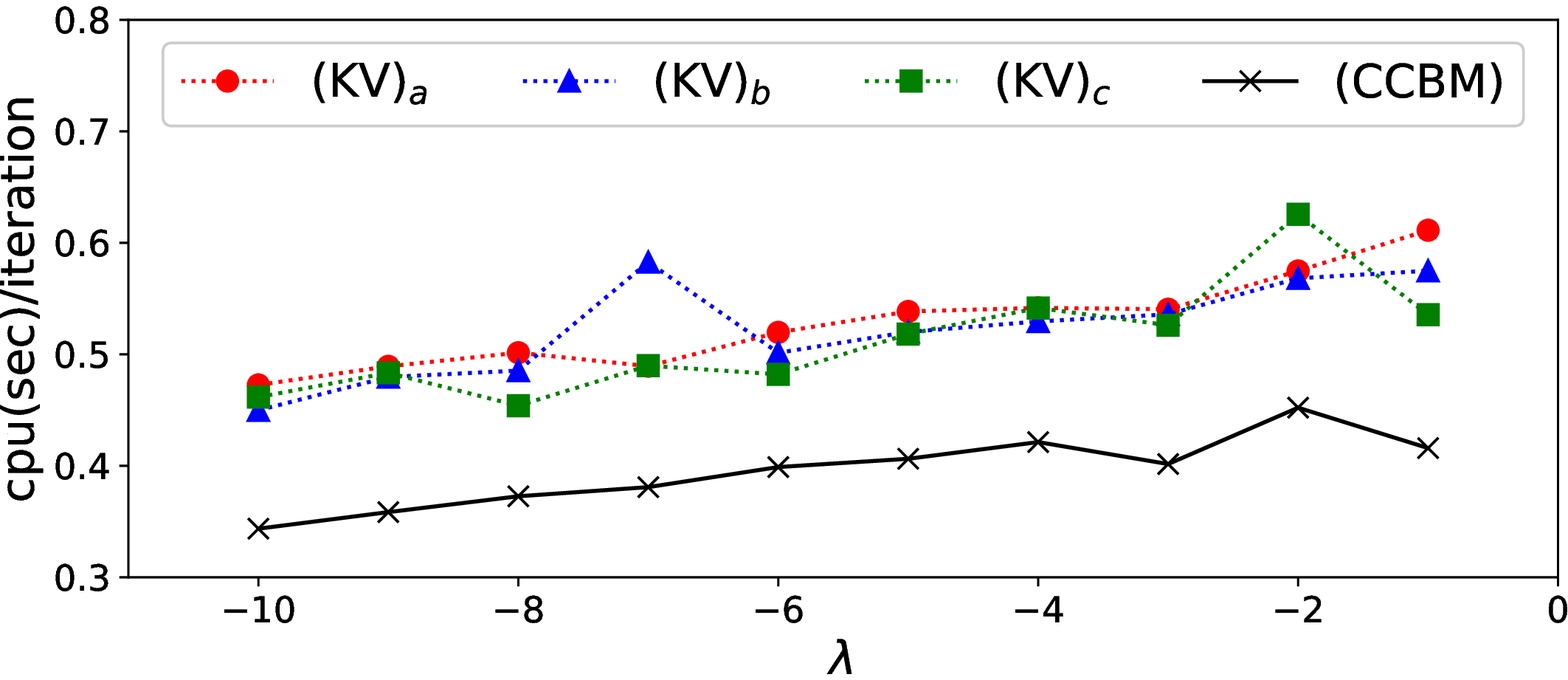}}\quad
\resizebox{0.48\linewidth}{!}{\includegraphics{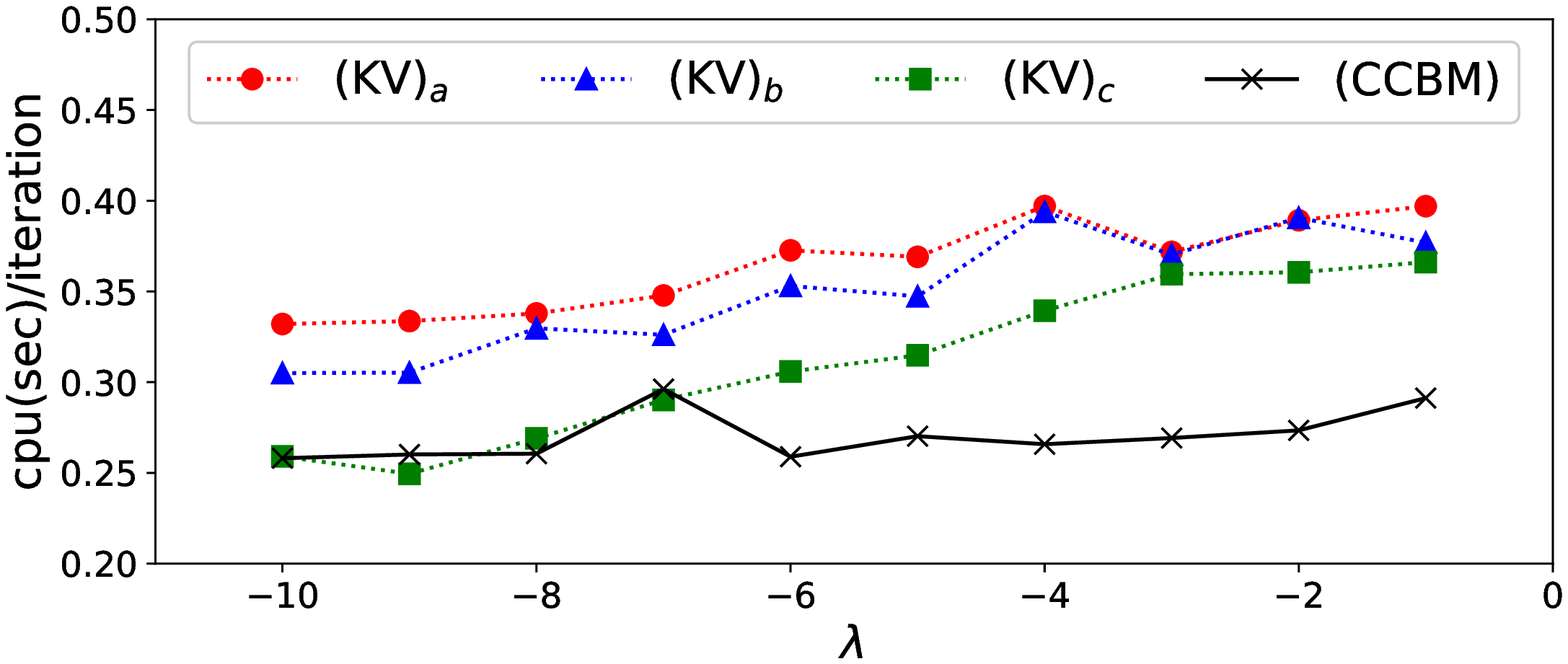}} 
\caption{$\lambda$ vs computational-time-per-iteration for Example \ref{example2d2} (left) and Example \ref{example2d3} (right)}
\label{fig2D:timeGraphs}
\end{figure}
The rest of our examples will focus on three dimensional cases using CCBM.
\begin{example}[Axisymmetric case in 3D]\label{example3d1}
Let us first test the method to a simple axisymmetric 3D-case with an analytical solution.
On this purpose, we consider the spheres $\mathcal{S}(\vect{0},r):=\{\xi \in \mathbb{R}^3 : |\xi| = r\}$ and $\mathcal{S}(\vect{0},R):= \{\xi \in \mathbb{R}^3 : |\xi| = R\}$ centered at $\vect{0}$ with radius $r > 0$ and $R > r$, respectively.
With $u(r) = 1$ and $u(R) = 0$, the solution to the Dirichlet problem \eqref{eq:state_ud} is exactly given by $u(\rho) = r (R - \rho) / [\rho (R-r)] $, $\rho \in (r,R)$ with normal derivative $\partial_{\rho} u(\rho)  =  -  Rr/[\rho^2(R-r)]$.
So, on the exterior surface, we have $\partial_{\rho} u(R)  = -r/[R(R-r)] =: \lambda$.
Thus, problem \eqref{eq:Bernoulli_problem} with $\Gamma =  \mathcal{S}(\vect{0},r)$ and $\lambda = -r/[R(R-r)] $, $0 < r < R$, has the unique exact free boundary solution $\Sigma^\ast = \mathcal{S}(\vect{0},R^\ast)$. 
For a concrete example, we let $r=0.3$ and $R^\ast=0.5$, giving us $\lambda = -3$, and take $\Sigma^{0}=\mathcal{S}(\vect{0},0.6)$ as the initial guess.
With $\texttt{Tol} = 10^{-6}$, initial maximum mesh size $h_{\max} = 0.1$ on the surfaces, and maximum volume $0.001$ for the tetrahedra, the procedure is completed after $53$ sec.
The nodes on $\Sigma_{f}$ have mean radii of $\bar{R} = 0.4787$ which gives a $4.26\%$ error with respect to the exact radius $R^\ast = 0.5$.
%
%
%
\end{example}
In the next three examples, the algorithm is terminated after 600 iterations in addition to setting $\texttt{Tol} = 10^{-8}$ in the stopping condition (see Remark \ref{rem:stopping_condition}).
\begin{example}[Perturbed sphere]\label{example3d2}
Let us define $\Gamma$ as the surface of a perturbed sphere having strict concave regions.
With $\Sigma^{0} = \mathcal{S}(\vect{0},1.5)$ as the initial guess, the results of the computation are summarized in Figure \ref{fig3D} (first row), including the plots for histories of cost values and Sobolev gradient norms.
\end{example}
\begin{example}[Torus]\label{example3d3}
We also look at the case where the fixed surface $\Gamma$ is given by a torus and set $\Sigma^{0} = \mathcal{S}(\vect{0},0.8)$ as our initial guess.
The computational results for this example are shown in Figure \ref{fig3D} (second row).
\end{example}
\begin{example}[Four disjoint spheres]\label{example3d4}
Let us also consider the case where $\Gamma$ is the union of four disjoints spheres having the same exact radius $r = 0.25$, and let $\Sigma^{0} = \mathcal{S}(\vect{0},0.9)$.
The results for this test case are summarized in Figure \ref{fig3D} (third row).
\end{example}
A summary of mesh details (number of boundary elements, triangles, and vertices) used in the last three examples are tabulated in Table \ref{table}.
The computed cost at final iterate and over-all cpu times for the three examples are also shown in the table.
%
%
\begin{figure}[hp!]
\centering
\resizebox{0.66\linewidth}{!}{\includegraphics{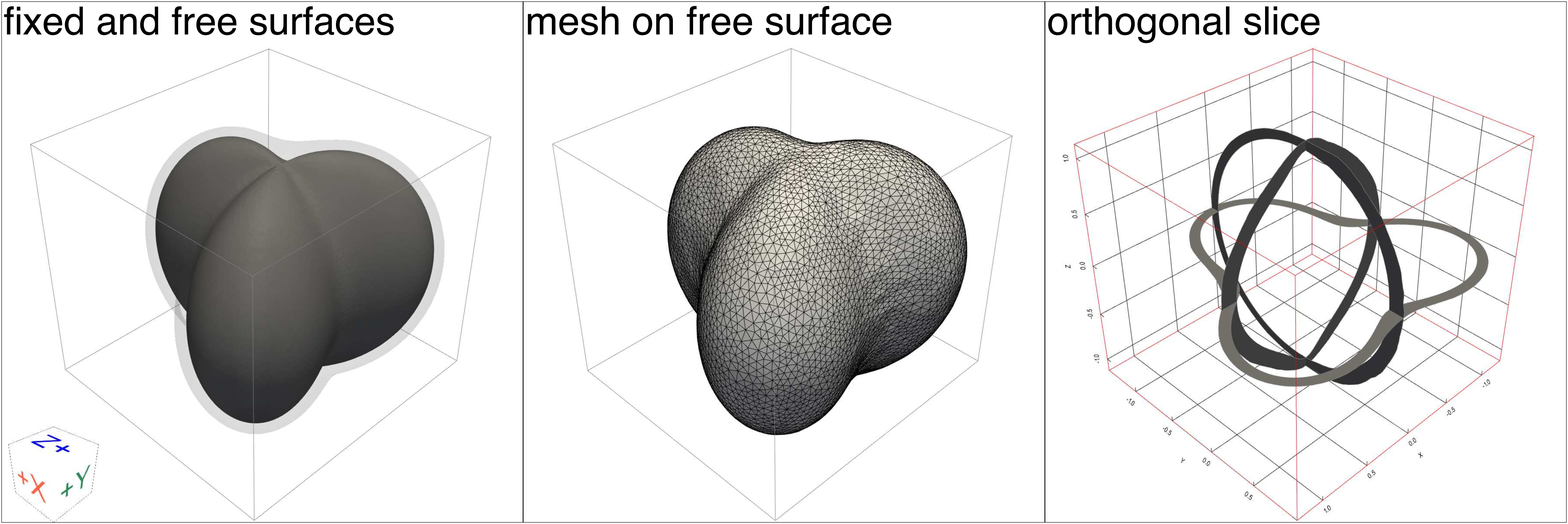}} \resizebox{0.24\linewidth}{!}{\includegraphics{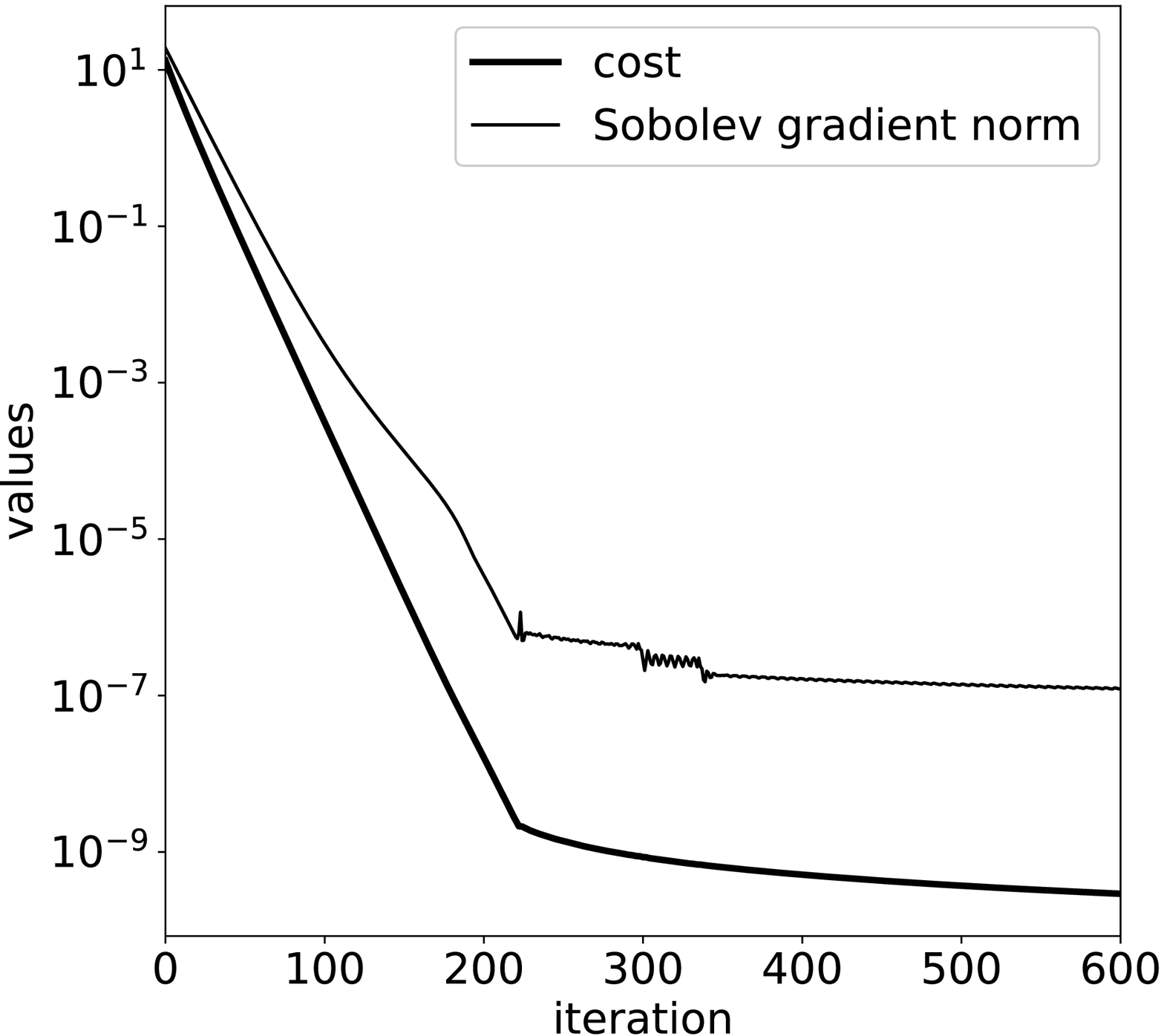}}\\
\resizebox{0.66\linewidth}{!}{\includegraphics{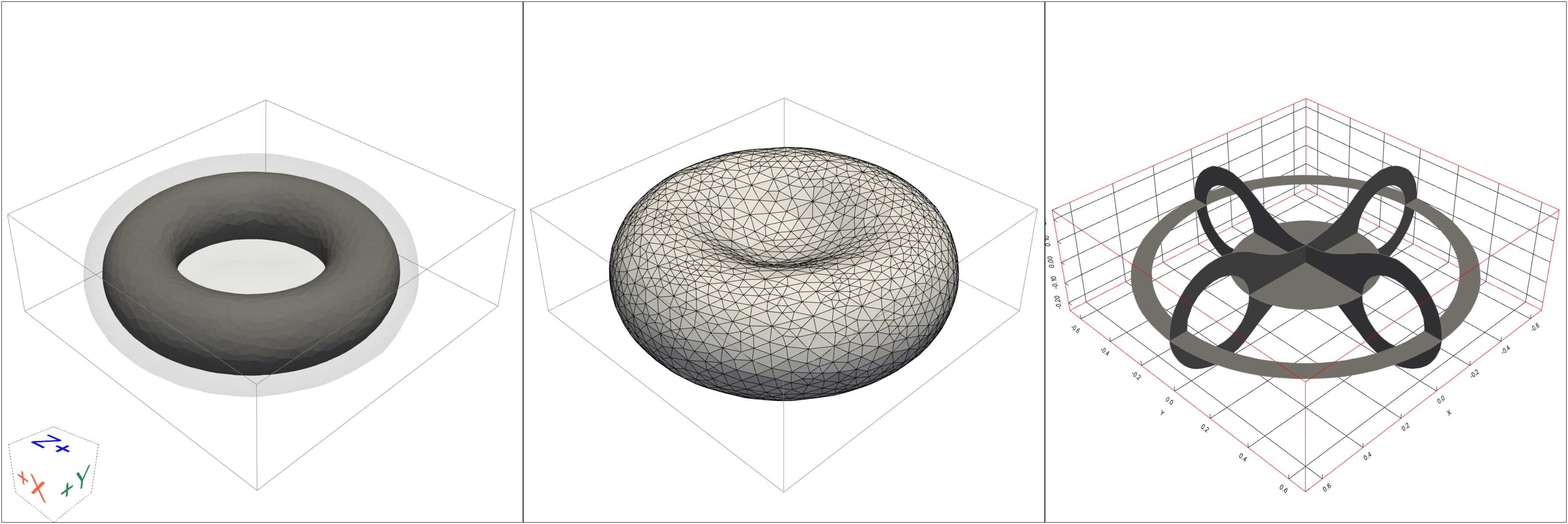}} \resizebox{0.24\linewidth}{!}{\includegraphics{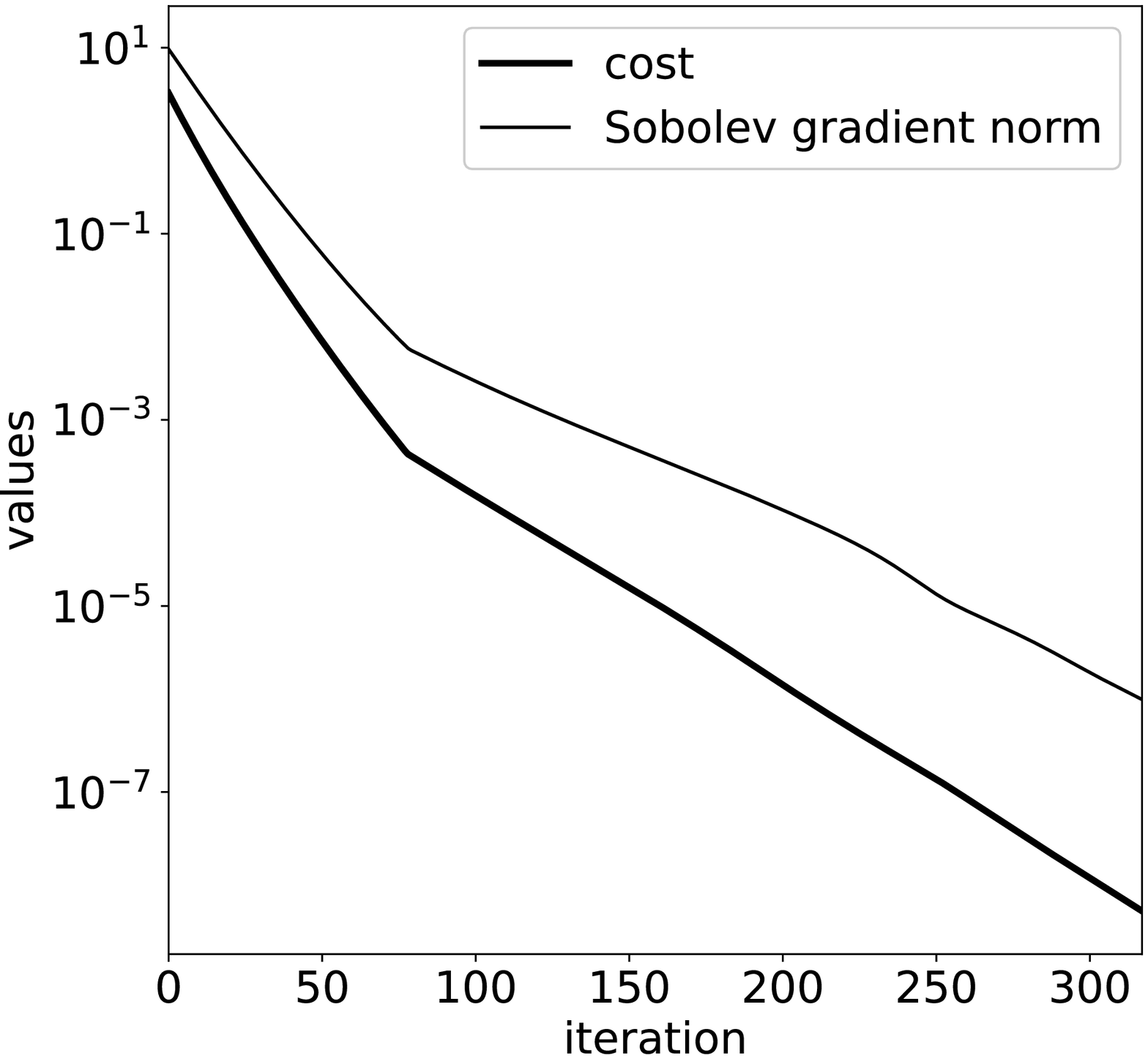}}\\
\resizebox{0.66\linewidth}{!}{\includegraphics{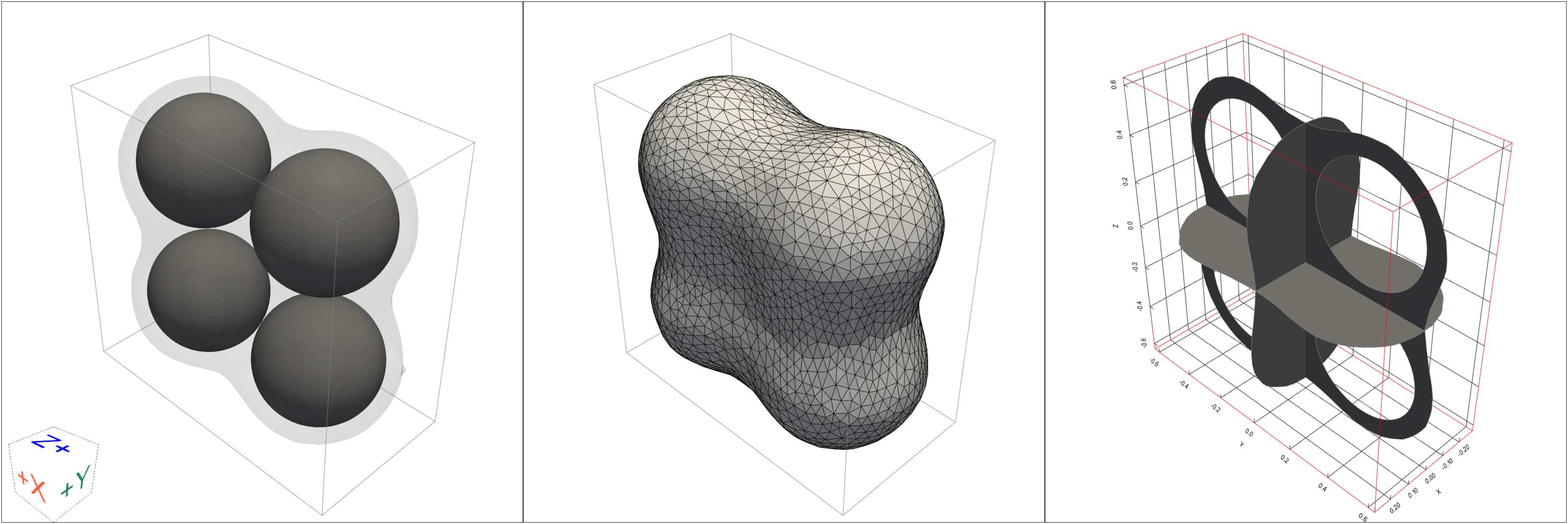}} \resizebox{0.24\linewidth}{!}{\includegraphics{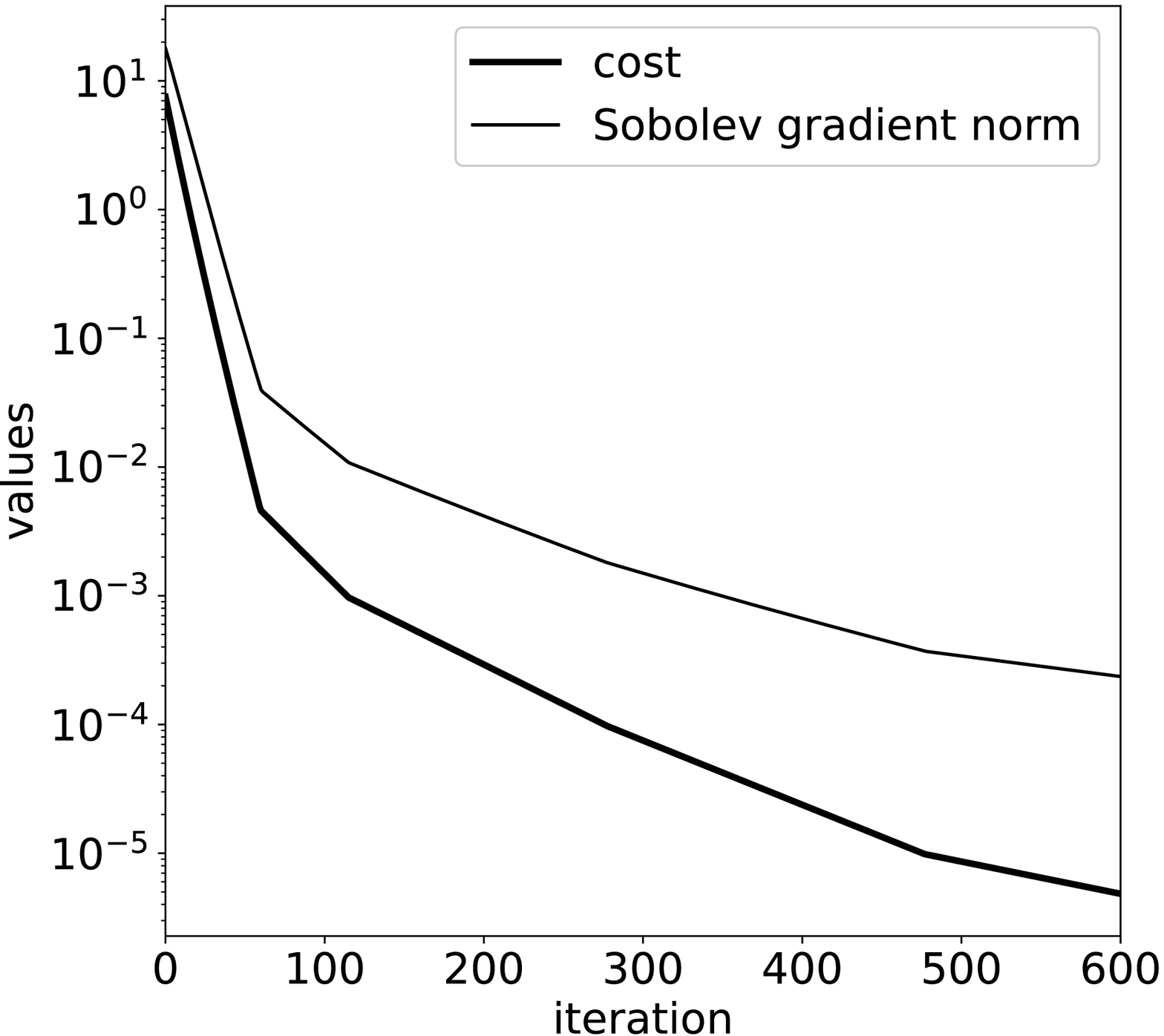}}
\caption{Computational results for Examples \ref{example3d2}--\ref{example3d4}}
\label{fig3D}
\end{figure}
\begin{table}
\centering
{\begin{tabular}{|l|l|l|l|}
\hline
	& Example \ref{example3d2} & Example \ref{example3d3}  & Example \ref{example3d4}\\\hline
number of boundary elements	&30,762	&9,232	&20,838\\
number of triangles 		&65,272	&21,549	&49,900\\
number of vertices 		&18,879	&5,905	&13,629\\
cost 					&2.92e-10	&5.26e-09	&4.84e-06\\
cpu time 				&14,436 s	&2,212 s	&9,680 s\\
\hline
\end{tabular}}
\caption{Mesh details and additional computational results for Examples \ref{example3d2}--\ref{example3d4}}
\label{table}
\end{table}
%
%
\begin{example}[\texttt{L}-block figure]\label{example3d5} Finally, we consider $\Sigma$ as the surface of an \texttt{L}-block figure, and consider the values $\lambda = -1, -7, -10$.
The results of the computations are shown in Figure \ref{fig3D:Lblock}.
The over-all cpu times for the case $\lambda = -7$ and $-10$ are less than $2,200$ sec, while the iteration process was completed after $6,200$ sec for $\lambda = -1$. 
\end{example}
\begin{figure}[hp!]
\centering
\resizebox{0.66\linewidth}{!}{\includegraphics{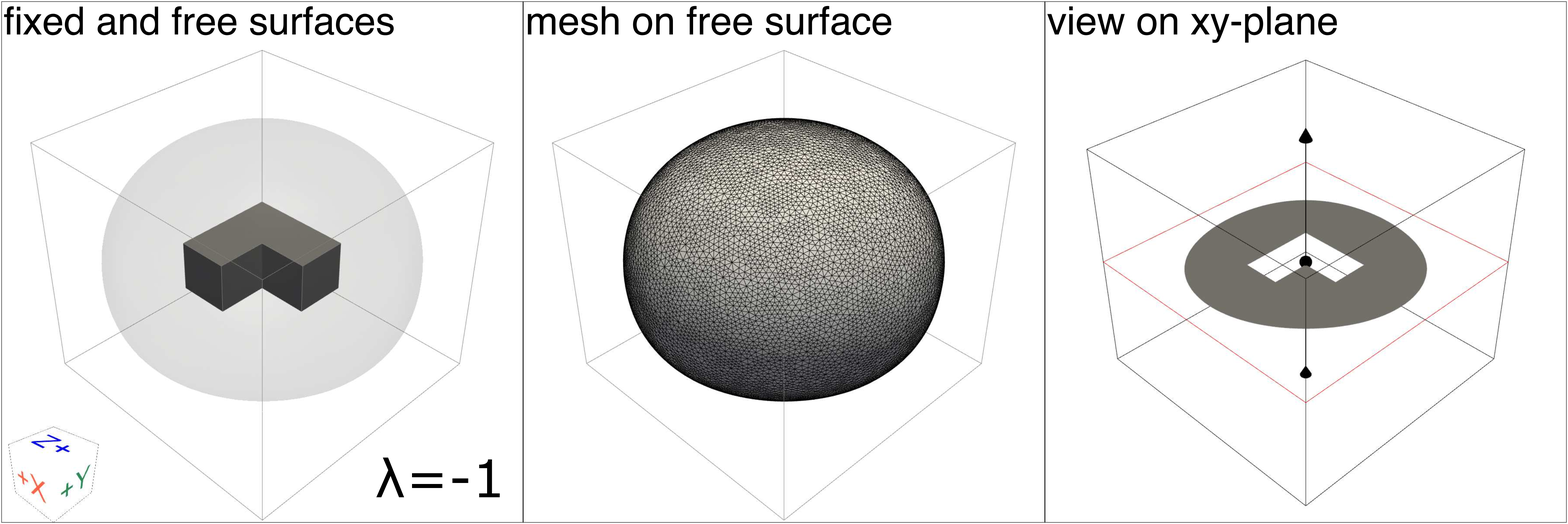}} \resizebox{0.24\linewidth}{!}{\includegraphics{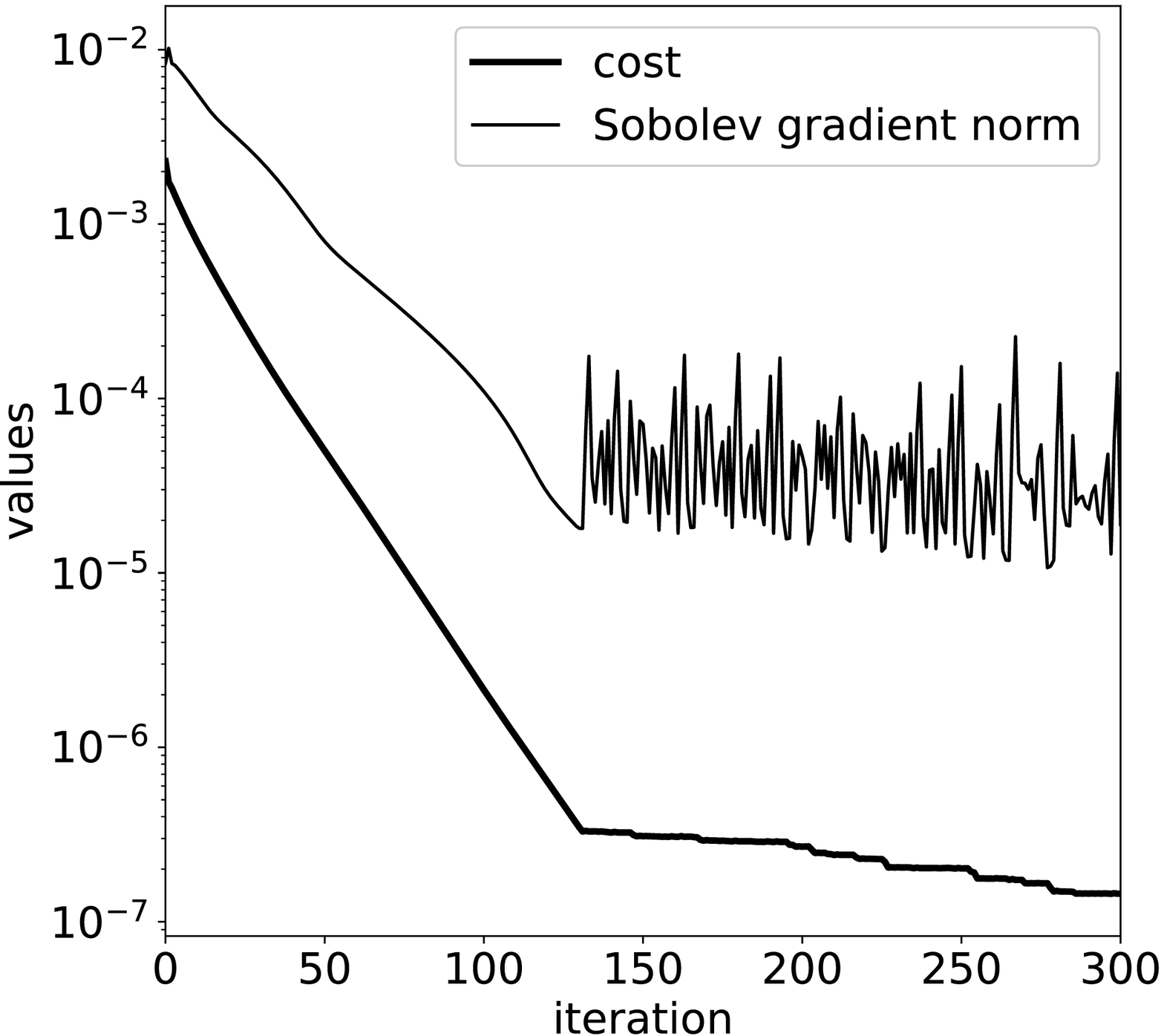}} 
\resizebox{0.66\linewidth}{!}{\includegraphics{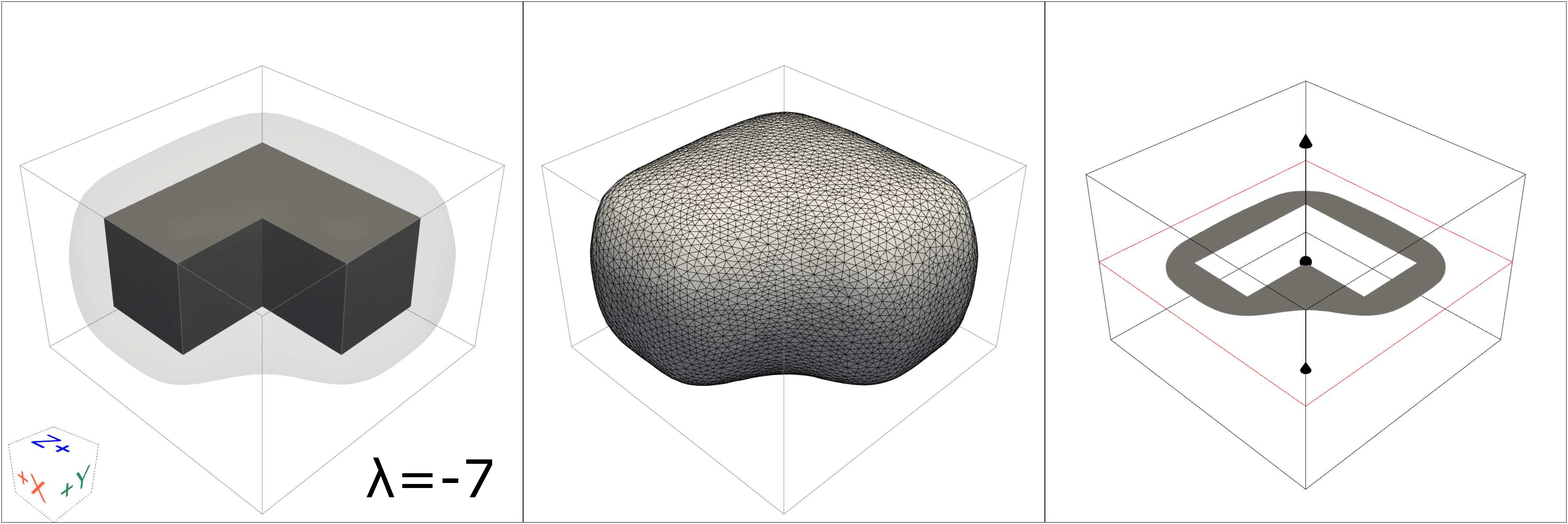}} \resizebox{0.24\linewidth}{!}{\includegraphics{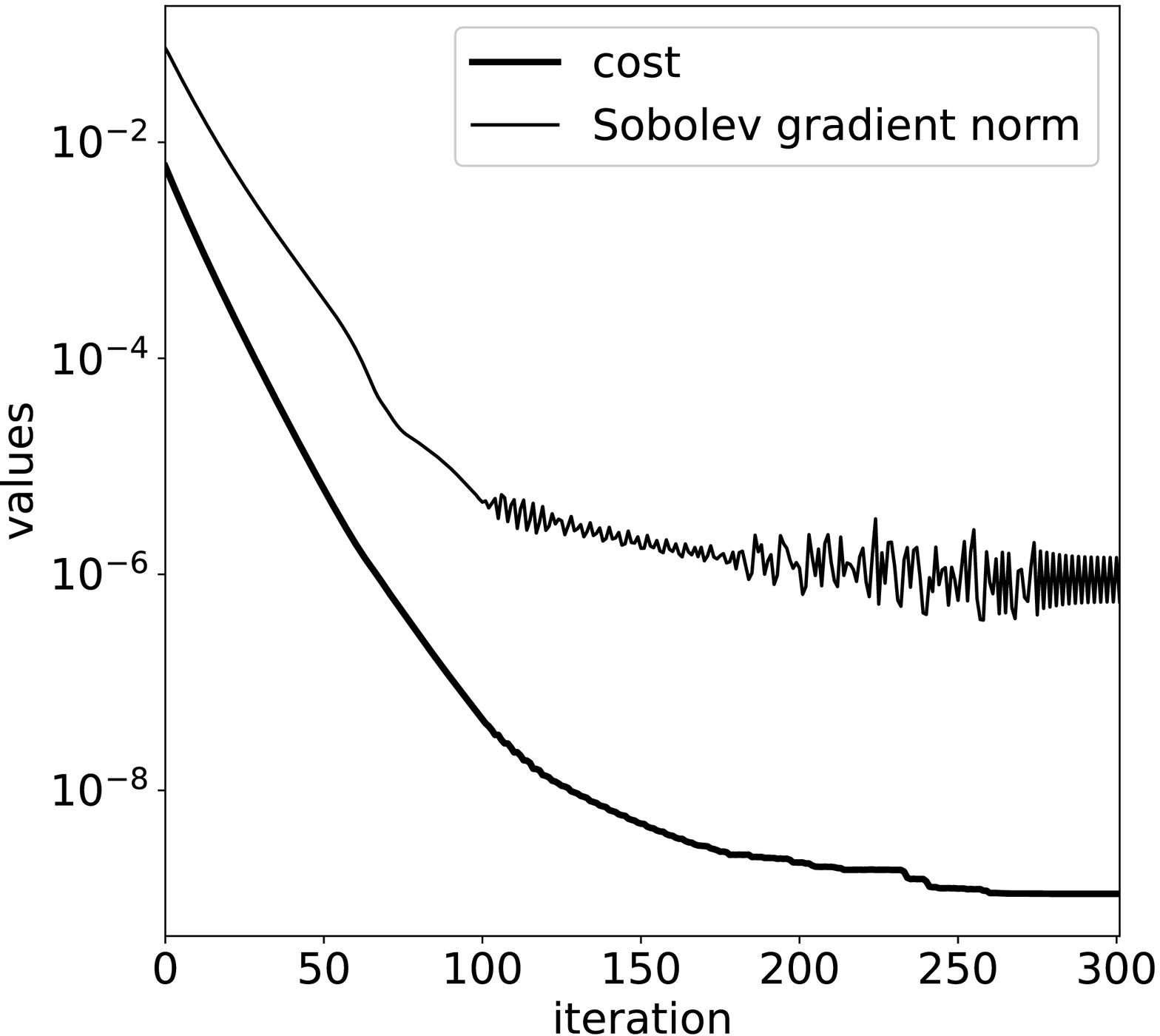}} 
\resizebox{0.66\linewidth}{!}{\includegraphics{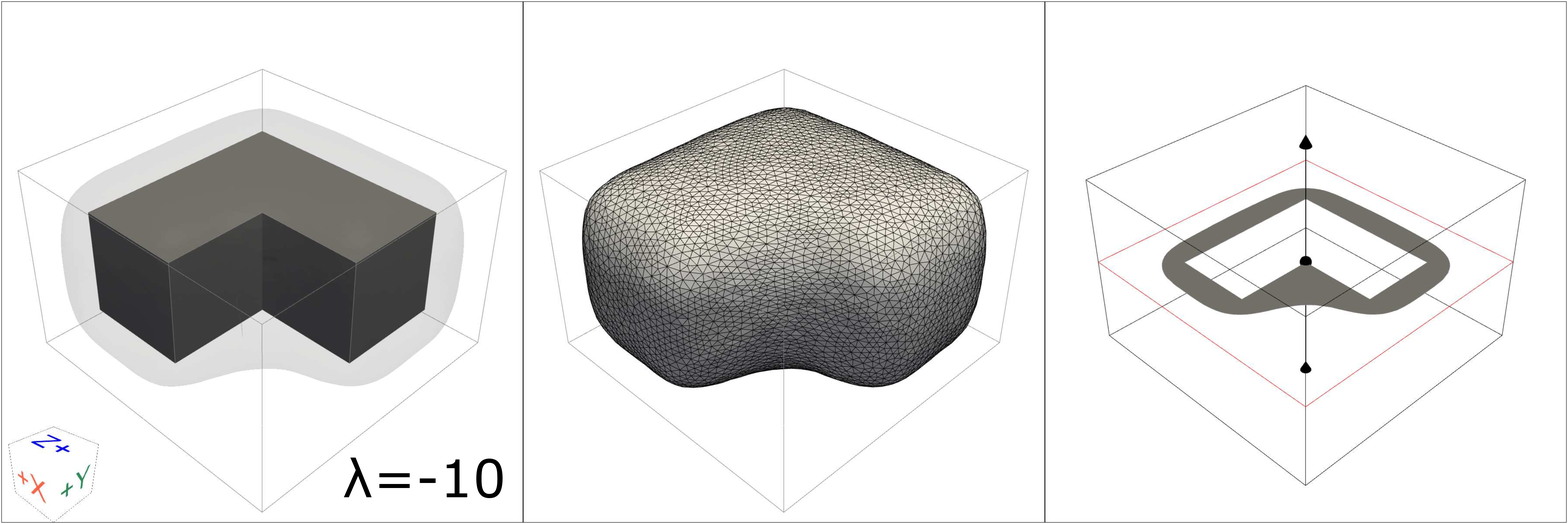}} \resizebox{0.24\linewidth}{!}{\includegraphics{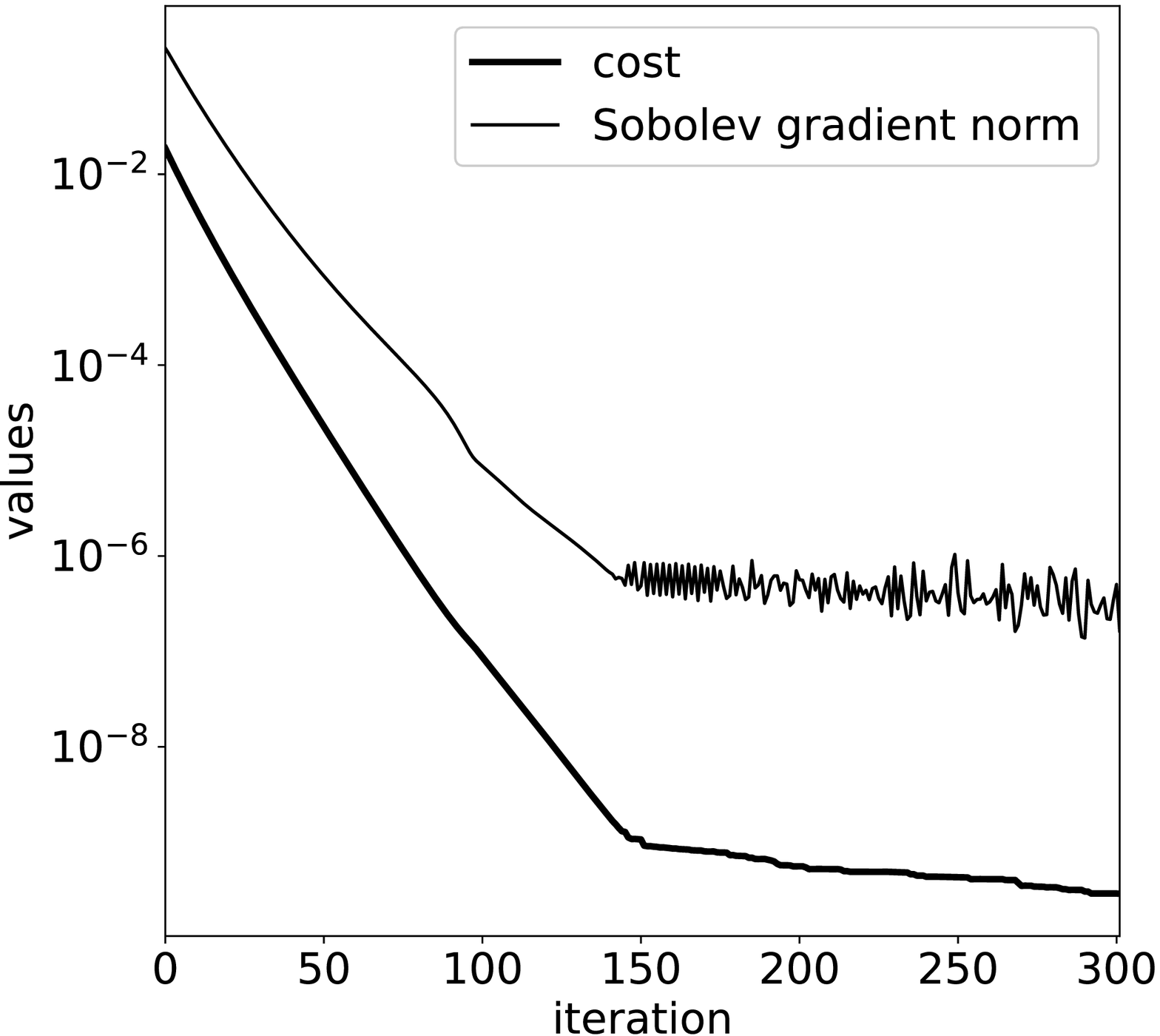}} 
\caption{Computational results for Example \ref{example3d5}}
\label{fig3D:Lblock}
\end{figure}
\section{Conclusions and Future Works}
We have proposed here a complex coupled boundary method in shape optimization framework as a numerical resolution to the exterior Bernoulli problem.
The shape gradient of the cost corresponding to the formulation is computed under a mild regularity assumption on the domain, and by using only the weak form of the equation satisfied by the material derivative of the state problem whose existence is shown in a rigorous manner.
The shape Hessian at a critical shape is also characterized through the chain rule approach under enough smoothness assumption on the domain. 
Also, by the same technique used to derive the shape gradient, the same expression is recomputed, but now with a weaker assumption on the regularity of the domain. 
The aforesaid expression is then analyzed in order to study the algebraic ill-posedness of the proposed method which is done by showing that the Riesz operator associated to the quadratic shape Hessian is compact.
Using the shape gradient information, a Sobolev gradient-based descent scheme was formulated in order to solve the problem numerically via finite element method.
Some numerical experiments in two dimensions are exhibited and are compared with those obtained using the Kohn-Vogelius approach.
Numerical results showed that the new method has some advantages when compared to the conventional KV approach since (1) it requires less overall computational-time-per-iteration when utilized in a Sobolev gradient based algorithm -- at least in the case of the conducted experiments, (2) is less prone to premature convergence under large domain variations, and (3) provides more accurate approximation of the optimal shape for coarser meshes.
Nonetheless, in general, KVM requires less number of iterations than CCBM for the present algorithm to converge.
Even so, for small step sizes, the optimal solutions obtained from the two methods coincide.
The new method is also tested in three dimensions, and various test cases were considered to further illustrate the feasibility and efficiency of the proposed method.
Further application of CCBM in solving free surface problems under shape optimization settings is the subject of our next investigation.
Also, in our future work, we will further study the issue of convergence of an evolving domain to a stationary point under the vanishing of the normal speed of the moving boundary.  \\ 

\textbf{Acknowledgements.}
The author wishes to thank the referees for their helpful comments, suggestions, and remarks which greatly improve the quality of this manuscript.
He is especially grateful to one of the referees for pointing out several inaccuracies and for bringing to his attention references \cite{DambrineLamboley2019} and \cite{Harbrecht2008}.
Lastly, the author acknowledges the support from JST CREST Grant Number JPMJCR2014.

%
%


\appendix

\section{Appendices}
\subsection{Proof of Lemma \ref{lem:tiihonen}}
\label{subsec:shape_derivative_of_the_state}

\begin{proof}[Proof of Lemma \ref{lem:tiihonen}]
	The proof proceeds in three classical steps (sketch here) after formulating the problem onto the fixed domain.
	First, we prove its weak convergence to the material derivative followed then by its convergence in strong sense.
	Afterwards, we deduce the shape derivative of the state using the identity $u'=\dot{u} - \nabla u \cdot \VV$.
	
	{\underline{Step 1.}} We recall from Lemma \ref{lem:transported_problem} the transported problem, and subtract from it the original one to obtain $(1/t)({u^t} - {u}) = 0$ on $\Gamma$, and for all $\overline{v} \in \HHg(\Omega)$, the equation
	\begin{equation}\label{eq:material_derivative_equation}
	\begin{aligned}
	&\intO{A_t \left(\frac{\nabla ({u^t} - {u})}{t}\right) \cdot \nabla \overline{v}} + i \intS{B_t \left( \frac{{u^t} - {u}}{t}\right) \overline{v} } \\
		&\ = \intO{ \left(\frac{\vect{I} - A_t}{t}\right) \nabla {u} \cdot \nabla \overline{v}} + i \intS{ \left(\frac{1 - B_t}{t}\right) {u} \overline{v} } 
					+ \lambda \intS{\left(\frac{B_t -1}{t}\right) \overline{v}}.
	\end{aligned}
	\end{equation}
	By taking $\frac1t({u^t} - {u}) \in \HHg(\Omega)$ as the test function, and using the continuity of the maps $t \mapsto A_t$ and $t \mapsto B_t$ at $t = 0$, and the fact that $A_t$ and $B_t$ are bounded for sufficiently small $t>0$, it can be verified that
	\[
	\vertiii{\frac{{u^t} - {u}}{t}}_{\HHg(\Omega)} 
	\lesssim \max\left\{ \left( \left\|\frac{A_t - \vect{I}}{t} \right\|_{\infty} + \left\|\frac{B_t -1}{t} \right\|_{\infty} \right) \vertiii{\nabla {u}}_{\QQ},  |\Sigma|^{1/2} \left\|\frac{B_t -1}{t} \right\|_{\infty} \right\}.
	\]
	Thus, $\left\{ \frac1t({u^t} - {u}) \right\}$ is bounded in $\HHg(\Omega)$.
	Therefore, the sequence is weakly convergent in $\HHg(\Omega)$ and its weak limit is the material derivative $\dot{u}$ of ${u}$.
	
	{\underline{Step 2.}} By passing to the limit $t \to 0$ in \eqref{eq:material_derivative_equation}, we see that $\dot{u}$ solves equation \eqref{eq:Lagrangian_derivative_of_the_state}.
	We use the said equation to prove the strong convergence in $\HHg(\Omega)$.
	Indeed, setting $v = \frac1t({u^t} - {u}) =: w^t \in \HHg(\Omega)$ in \eqref{eq:material_derivative_equation}, and noting that, on $\Omega$, we have ${u} = 1$ on $\Gamma$ and $\intO{ \nabla {u} \cdot \nabla \overline{v}} + i \intS{ {u} \overline{v} } 
		= \lambda \intS{\overline{v}}$, for all $\overline{v} \in \HHg(\Omega)$, we then get the equation
\begin{align*}
	& \intO{ \nabla  {w}^t \cdot \nabla \overline{w}^t} + i \intS{ {w}^t \overline{w}^t }\\ 
	 	&\ =  \left\{ - \intO{ (A_t - \vect{I} ) \nabla {w}^t \cdot \nabla \overline{w}^t} - i \intS{ (B_t - 1) {w}^t \overline{w}^t }  \right\} \\
		&\quad + \left\{ \intO{ \left(\frac{\vect{I} - A_t}{t}\right) \nabla {u} \cdot \nabla \overline{w}^t} + i \intS{ \left(\frac{1 - B_t}{t}\right) {u} \overline{w}^t } 
			+  \lambda \intS{\left(\frac{B_t -1}{t}\right) \overline{w}^t} \right\}\\
		&\ =: \mathbb{I}_1(t) + \mathbb{I}_2(t).
\end{align*}	
	Using the weak convergence result for the sequence $\{ {w}^t \}$ obtained in the previous step, we easily deduce that $ \lim_{t \to 0} \mathbb{I}_1(t) = 0 $ and	$ \lim_{t \to 0} \mathbb{I}_2(t) = - \intO{ A \nabla {u} \cdot \nabla \overline{\dot{u}}} - i \intS{ ({\operatorname{div}}_{\Sigma} \VV) {u} \overline{\dot{u}} } 
					+ \lambda \intS{ ({\operatorname{div}}_{\Sigma} \VV) \overline{\dot{u}}}$.
	Now, from \eqref{eq:Lagrangian_derivative_of_the_state} in Proposition \ref{prop:umaps}, we conclude that $\mathbb{I}_2(t) \to \intO{ \nabla \dot{u} \cdot \nabla \overline{\dot{u}}} + i \intS{ \dot{u}\overline{\dot{u}} }$ as $t \to 0$.
	This shows the strong convergence of $\nabla {w}^t$ to $\nabla {\dot{u}}$ in $\HHg(\Omega)$.
	The equivalence of norms between the $\HHg(\Omega)$-norm and the usual $\HH^{1}(\Omega)$-norm implies the strong convergence of ${w}^t$ to ${\dot{u}}$ in $\HHg(\Omega)$.
	
	{\underline{Step 3.}} In the last step, we deduce the equations satisfied by the shape derivative of $u$ using the identity $u' = \dot{u} - \VV \cdot \nabla u$.
	Let us first note that $\dot{u} \in \HHg(\Omega)$.
	So, $u' = \dot{u} - \VV \cdot \nabla u = 0$ on $\Gamma$ because $\VV = 0$ on $\Gamma$.
	Next, we let the trace of the normal derivative of the (complex conjugate of the) test function ${v}$ on $\Sigma$ in equation \eqref{eq:Lagrangian_derivative_of_the_state} be zero.
	Then, by expansion \eqref{eq:expansion}, and integration by parts, we get
	$-\intO{\nabla \dot{u} \cdot \nabla \overline{v} } = \intO{ \Delta {u} (\VV \cdot \nabla \overline{v})} + \intO{ \Delta \overline{v} (\VV \cdot \nabla {u})}  
		= -\intO{\nabla (\VV \cdot \nabla {u}) \cdot \nabla \overline{v} }$,
	for any smooth function ${v}$ with compact support on $\Omega$ and such that $\dn{\overline{v}} = 0$ on $\Sigma$.
	Since $-\Delta {u} = 0$ in $\Omega$, we easily find that $\intO{\nabla (\dot{u} - \VV \cdot \nabla {u}) \cdot \nabla \overline{v} } = -\intO{ (\Delta {u}')\overline{v} } = 0$. 
	Varying ${v}$, we obtain $\Delta {u}' = 0$ in $\Omega$. 
	Let us now choose ${v} \in \HH^2(\Omega) \cap \HHg(\Omega)$ such that $\dn{\overline{v}} = 0$ on $\Sigma$.
	Applying Green's theorem and the tangential Green's formula \eqref{eq:tangential_Greens_formula} to equation \eqref{eq:Lagrangian_derivative_of_the_state}, and noting that $\dn{u} + i{u} = \lambda$ on $\Sigma$, we obtain $\intS{ (\dn{{u}'} + i {u}') \overline{v} }  = - \intS{ (\nabla {\overline{v}} \cdot \nabla {u}) \Vn } - \intS{ \overline{v} \left[ i(\dn{u} + \kappa {u}) - \lambda \kappa \right] \Vn }$.
	Since $\dn{\overline{v}} = 0$ on $\Sigma$, then we can write $\intS{ (\nabla {\overline{v}} \cdot \nabla {u}) \Vn } = \intS{ (\nabla_{\Sigma} {\overline{v}} \cdot \nabla_{\Sigma} {u}) \Vn } =  - \intS{ \overline{v} \operatorname{div}_{\Sigma}( \Vn \nabla_{\Sigma} {u}) }$,
	where the second equality follows again from the tangential Green's formula \eqref{eq:tangential_Greens_formula} together with the fact that $(\Vn \nabla_{\Sigma} {u}) \cdot \nn = 0$ on $\Sigma$.
	Putting the computed identity to the previous equation above leads to
	\[
	\intS{ (\dn{{u}'} + i {u}') \overline{v} } 
	= \intS{ {\overline{v}} \left[ \operatorname{div}_{\Sigma}( \Vn \nabla_{\Sigma} {u})  - i(\dn{u} + \kappa {u}) \Vn + \lambda \kappa \Vn \right]  }.
	\]
	Varying ${v}$ yields the equation $\dn{{u}'} + i {u}' = \operatorname{div}_{\Sigma}( \Vn \nabla_{\Sigma} {u})  - i(\dn{u} + \kappa {u}) \Vn + \lambda \kappa \Vn$ on $\Sigma$.
	After collecting all equations for the shape derivative ${u}'$, we finally obtain \eqref{eq:shape_derivative_of_the state}.
\end{proof}
\subsection{Alternative computation of the shape gradient}
\label{subsec:shape_derivatives_of_the_cost}
We give below the computation of the shape gradient \eqref{eq:shape_gradient} under a $\mathcal{C}^{2,1}$ regularity assumption on the domain.
With the given regularity, the existence of the shape derivative of the state is guaranteed and the shape gradient of the cost is easily obtained using Hadamard's domain differentiation formula \eqref{eq:Hadamard_domain_formula} -- assuming the perturbation of $\Omega$ preserves its regularity. 
\begin{proposition}
	Let $\Omega \in \mathcal{C}^{2,1}$ and $\VV \in \sfTheta^2$.
	Then, the shape derivative of $J$ at $\Omega$ along $\VV$ is given by $dJ(\Omega)[\VV] = \intS{\GG{\Vn}}$, where $G$ is the expression in \eqref{eq:shape_gradient}.
\end{proposition}

\begin{proof}
	Let us assume that $\Omega$ is of class $\mathcal{C}^{2,1}$ and $\VV \in \sfTheta^{2}$. 
	By classical regularity theory, $\ur$, $\ui \in H^3(\Omega)$ and so, we can apply formula \eqref{eq:Hadamard_domain_formula} to obtain -- noting that $\VV\big|_{\Gamma} = 0$ -- the derivative 
	${\operatorname{\mathnormal{d}}}J(\Omega)[\VV]  = \intO{\ui \uip } + \frac12 \intS{ |\ui|^2 \Vn }=:\mathbb{I}_1+\mathbb{I}_2$.
We focus on rewriting $\mathbb{I}_1$ through the adjoint method.
To this end, we consider the adjoint problem \eqref{eq:adjoint_system}, multiply it by $\up \in \HHg(\Omega)$, and then apply integration by parts to obtain $\intO{ \nabla \overline{p} \cdot \nabla \up} + i \intS{\overline{p} \up } = \intO{\ui \up}$.
We do the same on \eqref{eq:shape_derivative_of_the state} with the multiplier $\overline{p} \in \HHg(\Omega)$ to obtain $\intO{\nabla \up \cdot \nabla \overline{p}} + i \intS{\up \overline{p}} = \intS{\overline{p} \Upsilon(u)[\Vn]}$.
The last two equations lead us to $\intO{\ui \up} = \intS{\overline{p} \Upsilon(u)[\Vn]}$.
Comparing the respective real and imaginary parts on both sides of this equation will give us the form for $\mathbb{I}_1$, which, upon adding to $\mathbb{I}_2$ finally yield the desired expression.
\end{proof}
\end{document}